\documentclass{hha}

\usepackage[english]{babel}

\usepackage{mathrsfs}
\usepackage{multirow}
\usepackage{enumitem}
\usepackage{tikz}
\usepackage[all]{xy}
\usepackage{graphicx}
\usepackage{shuffle}

\usepackage{color, colortbl}
\definecolor{colorl}{rgb}{.804,.196,.47}
\definecolor{colork}{rgb}{.45,.05,.545}
\definecolor{colori}{rgb}{.545,0,0}
\definecolor{colorj}{rgb}{.024,.15,.645}
\definecolor{colorn}{rgb}{.4,.2,.4}
\definecolor{colorM}{rgb}{.41,.545,.132}
\definecolor{colorN}{rgb}{.21,.545,.332}
\definecolor{lightgrey}{rgb}{.804,.804,.756}

\newcommand*\circled[1]{$\,$\tikz[baseline=(char.base)]{
            \node[shape=circle,draw,inner sep=1pt] (char) {#1};}}

\newcommand{\II}{\mathbf{I}}
\newcommand{\NN}{\mathbb{N}}
\newcommand{\ZZ}{\mathbb{Z}}

\newcommand{\C}{\mathcal C}
\newcommand{\F}{\mathcal F}
\newcommand{\FF}{\mathcal F'}
\newcommand{\Hei}{\mathscr H}
\newcommand{\Dr}{\mathscr D}
\newcommand{\W}{\mathscr W}
\newcommand{\X}{\mathscr X}
\newcommand{\Y}{\mathscr Y}
\newcommand{\Z}{\mathscr Z}

\newcommand{\T}{B}
\newcommand{\For}{For}
\newcommand{\Hom}{\operatorname{Hom}}
\newcommand{\End}{\operatorname{End}}

\newcommand{\Ob}{\operatorname{Ob}}

\renewcommand{\k}{\Bbbk}
\newcommand{\Vect}{\mathbf{Vect}_\k}
\newcommand{\vect}{\mathbf{vect}_\k}
\newcommand{\Alg}{\mathbf{Alg}}

\newcommand{\Bialg}{\mathbf{Bialg}}
\newcommand{\HAlg}{\mathbf{HAlg}}
\newcommand{\ModAlg}{\mathbf{ModAlg}}
\newcommand{\coAlg}{\mathbf{coAlg}}

\newcommand{\BrSyst}{\mathbf{BrSyst}}
\newcommand{\BrSystP}{\BrSyst^\downarrow}
\newcommand{\BrSystCP}{\BrSyst^\uparrow}
\newcommand{\BrSystBP}{\BrSyst^\updownarrow}
\newcommand{\ModCat}{\mathbf{Mod}}
\newcommand{\ModCatN}{\mathbf{Mod}}

\newcommand{\Id}{\operatorname{Id}}

\newcommand{\Bimod}{\mathcal B \mathcal M}
\newcommand{\BBialg}{\mathcal B}
\newcommand{\BBialgg}{\mathcal B'}
\newcommand{\op}{\operatorname{op}}

\newcommand{\leftV}{\overleftarrow{V}}

\newcommand{\oV}{\overline{V}}

\newcommand{\oW}{\overline{W}}

\newcommand{\of}{\overline{f}}

\newcommand{\osigma}{\overline{\sigma}}

\newcommand{\oxi}{\overline{\xi}}
\newcommand{\onu}{\overline{\nu}}
\newcommand{\ossigma}{\boldsymbol{\overline{\sigma}}}
\newcommand{\ocsh}{\underset{\osigma}{\cshuffle}}
\newcommand{\osh}{\underset{\osigma}{\shuffle}}
\newcommand{\Ocsh}{\underset{-\osigma}{\cshuffle}}
\newcommand{\rrho}{\mathcal R}
\newcommand{\llambda}{\mathcal L}
\renewcommand{\le}{\leqslant}
\renewcommand{\ge}{\geqslant}
\newcommand{\bi}[1]{\textbf{\textit{#1}}}
\newcommand{\gsm}{\mbox{$\blacktriangleright \hspace{-0.7mm}<$}}
\newcommand{\gtl}{\mbox{$>\hspace{-0.85mm}\blacktriangleleft$}}

\newcommand\mapsfrom{\mathrel{\reflectbox{\ensuremath{\mapsto}}}}
\newcommand\longmapsfrom{\mathrel{\reflectbox{\ensuremath{\longmapsto}}}}
\newcommand*{\longhookrightarrow}{\ensuremath{\lhook\joinrel\relbar\joinrel\rightarrow}}
\newcommand*{\longhookleftarrow}{\ensuremath{\leftarrow\joinrel\relbar\joinrel\rhook}}

\newcommand*{\longlongleftrightarrow}{\ensuremath{\leftarrow\joinrel\relbar\joinrel\relbar\joinrel\relbar\joinrel\relbar\joinrel\relbar
\joinrel\relbar\joinrel\relbar\joinrel\relbar\joinrel\rightarrow}}

  \theoremstyle{plain}
\newtheorem{observation}[theorem]{Observation}
  \theoremstyle{definition}
\newtheorem{notation}[theorem]{Notation}
\newtheorem{convention}[theorem]{Convention}

\begin{document}

\title{Braided systems: a unified treatment of algebraic structures with several operations}
\shorttitle{Braided systems}

\author{Victoria Lebed}             
\email{lebed.victoria@gmail.com}
\address{School of Mathematics,
         Trinity College,
         Dublin 2,
         Ireland}




\begin{abstract}
Bialgebras and Hopf (bi)modules are typical algebraic structures with several interacting operations. Their structural and homological study is therefore quite involved. We develop the machinery of braided systems, tailored for handling such multi-operation situations. Our construction covers the above examples (as well as Poisson algebras, Yetter--Drinfel$'$d modules, and several other structures, treated in separate publications). In spite of this generality, graphical tools allow an efficient study of braided systems, in particular of their representation and homology theories. These latter naturally recover, generalize, and unify standard homology theories for bialgebras and Hopf (bi)modules (due to Gerstenhaber--Schack, Panaite--{\c{S}}tefan, Ospel, Taillefer); and the algebras encoding their representation theories (Heisenberg double, algebras~$\mathscr X$, $\mathscr Y$, $\mathscr Z$ of Cibils--Rosso and Panaite). Our approach yields simplified and conceptual proofs of the properties of these objects.
\end{abstract}


\maketitle

\section{Introduction}
In~\cite{Lebed1} we developed representation and (co)homology theories for braided objects in a monoidal category~$\C$ (e.g., $\C = \Vect$). We interpreted associative / Lie algebras and self-distributive structures as braided objects, and could thus apply our theories to them. As a result, we unified classical constructions into one, and explained their otherwise mysterious similarities. The aim of this article is to extend the braided approach to more complicated algebraic structures.

Concretely, an object~$V$ in~$\C$ is called \emph{braided} when endowed with a morphism $\sigma\colon V^{\otimes 2} \to  V^{\otimes 2}$ satisfying the \emph{Yang--Baxter equation} (\emph{YBE})
$\sigma^1 \sigma^2 \sigma^1 = \sigma^2 \sigma^1 \sigma^2$, where $\sigma^1 =\sigma \otimes \Id_V$ and $\sigma^2= \Id_V \otimes \sigma$. For instance, in~\cite{Lebed1} we showed that a unital associative algebra is braided, with $\sigma_{Ass} (v \otimes w) =  1 \otimes v \cdot w$. However, this one-object-one-morphism setting is very restrictive. For instance, a {bialgebra} comes with several operations: (co)multi\-plication and (co)unit. Its Gerstenhaber--Shack (co)homology is defined on $\Hom(H^{\otimes n},H^{\otimes m}) \simeq H^{\otimes m} \otimes (H^*)^{\otimes n}$ and involves two objects, $H$ and~$H^*$ (here $H$ is finite-dimensional). A way out is to consider a family of objects $(V_1, \ldots, V_r)$ in~$\C$ endowed with morphisms  $\sigma_{i,j} \colon V_i \otimes V_j \to V_j \otimes V_i$, ${i  \le j}$, satisfying the colored version of the YBE on all tensor products $V_i \otimes V_j \otimes V_k$ with ${i \le j \le k}$. This is what we call a \emph{rank~$r$ braided system}, a notion central to this article.  The $r=2$ case recovers the \emph{$WXZ$-systems} of Hlavat{\'y}--{\v{S}}nobl \cite{HlSn}, motivated by the concept of quantum doubles. They classified such systems in dimension~$2$ and studied their symmetries.

Sections~\ref{sec:BrSyst}-\ref{sec:BrSystemHomology} extend the representation and (co)homology theories of braided objects to braided systems. Multi-versions of \emph{braided modules} and \emph{braided (co)chain complexes} are defined; the latter take the former as coefficients. Further sections explore braided systems encoding various algebraic structure, in the sense of Table~\ref{tab:BrVsStr}. The row $\BrSyst_r(\C) \hookleftarrow \mathbf{Structure}(\C)$ means that the categories of the algebraic structures we work with (e.g., bialgebras in~$\C$) are recovered as subcategories of the category of rank~$r$ braided systems in~$\C$. Properties of our structures and their (co)homologies are then deduces from general results on braided systems.
\begin{table}\centering
\begin{tabular}{|rcl|}
\hline
\rowcolor{lightgrey} braided system & $\mapsfrom$& algebraic structure \\
\hline
braiding components $\sigma_{i,j}$ &$\leftrightarrow$& operations  \\
\hline
colored YBEs &$\Leftrightarrow$& defining relations  \\
\hline
braided morphisms  &$\simeq$& structural morphisms \\
\hline
$\BrSyst_r(\C)$ &$\hookleftarrow$& $\mathbf{Structure}(\C)$ \\
\hline
braided modules &$\supseteq$& usual modules \\
\hline
braided complexes &$\supseteq$& usual complexes \\
\hline
\end{tabular}
\caption{Braided interpretation for algebraic structures}\label{tab:BrVsStr}
\end{table}

The braided systems considered here are composed of {unital associative algebras} (\emph{UAAs}) $(V_i,\mu_i,\nu_i)$, with as diagonal braiding components~$\sigma_{i,i}$ the associativity braidings $\sigma_{Ass}=\nu_i \otimes \mu_i$. 
In Section~\ref{sec:SystemsOfUAAs} we study such systems, and relate them to \emph{braided tensor products of algebras} $\leftV = V_r \otimes \cdots \otimes V_1$. Concretely, we show that morphisms~$\xi_{i,j}$ for $i<j$ complete the associativity braidings $\sigma_{i,i}$ into a braided system structure if and only if they define an associative multiplication on~$\leftV$ by 
\[\mu_{\leftV}=(\mu_r \otimes \cdots \otimes \mu_1)  \xi_{1,2}^{2r-2}  (\xi_{2,3}^{2r-4}  \xi_{1,3}^{2r-3})  \cdots  (\xi_{r-1,r}^2  \cdots  \xi_{2,r}^{r-1}  \xi_{1,r}^r),\]
where~$\xi_{i,j}^p$ denotes the morphism~$\xi_{i,j}$ applied at positions~$p$ and~$p+1$.

Rank~$2$ braided tensor products are at the heart of Majid's \emph{braided geometry} \cite{Majid3,BraidedDiff,MajidBraided}. They provide an algebra analogue of the product of two spaces in non-commutative geometry. A pleasant consequence of Majid's work is the construction of new examples of non-commutative non-cocommutative Hopf algebras as bicross products, which are particular cases of braided tensor products. 

The case of general~$r$ independently appeared in two different frameworks:
\begin{enumerate}
\item Mart{\'{\i}}nez, Pe{\~n}a, Panaite, and Van~Oystaeyen considered \emph{iterated twisted tensor products of algebras} in~$\Vect$, and studied various Hopf algebraic, geometric, and physical examples \cite{IteratedTwisted}. Their motivation came from braided geometry.
\item Cheng~\cite{Cheng}, generalizing Beck~\cite{Beck}, introduced the notion of \emph{iterated distributive laws}. Categorical motivations (a study of interchange laws in a strict $n$-category) led her to work in the monoidal category of the endofunctors of a given category.
\end{enumerate}

All these approaches relate the associativity of~$\mu_{\leftV}$ to the YBEs for the~$\xi_{i,j}$ with $i<j$, combined with the naturality of the~$\xi_{i,j}$ w.r.t. the multiplications~$\mu_i$ and~$\mu_j$. Our main contribution is a treatment of {all} the conditions ensuring the associativity of~$\mu_{\leftV}$ in terms of YBEs: 
\vspace*{-3pt}
\begin{align*}  \left. 
   \begin{array}{r c l}
   \text{associativity of } \mu_i & \Longleftrightarrow &  \text{YBE on } V_i \otimes V_i \otimes V_i\\
   \text{compatibility between } \xi_{i,j} \;\&\; \mu_i & \Longleftrightarrow &  \text{YBE on } V_i \otimes V_i \otimes V_j\\
   \text{compatibility between } \xi_{i,j} \;\&\;  \mu_j & \Longleftrightarrow &  \text{YBE on } V_i \otimes V_j \otimes V_j   \end{array}
  \right\}& \quad \text{new}   \\  
 \left. 
   \begin{array}{r c l}
      \text{compatibilities between the } \xi & \Longleftrightarrow &  \text{YBE on } V_i \otimes V_j \otimes V_k 
   \end{array}
  \right\}& \quad \text{known}  
\end{align*}  
This entirely braided interpretation is made possible by our associativity braiding. Among its advantages is the applicability of the braided (co)homology machinery to braided tensor products of algebras; this turns out to be fruitful in our examples.

Sections \ref{sec:crossed}-\ref{sec:HBimod} explore braided systems of UAAs encoding \textit{generalized two-sided crossed products} (as defined by Bulacu, Panaite, and Van Oystaeyen \cite{PanaiteGen}) and finite-dimensional $\k$-linear \textit{bialgebras}. For the latter we propose two braided systems, recovering Hopf modules and Hopf bimodules as corresponding braided modules, and yielding a graphical interpretation of Hopf (bi)module homology, which is more workable than the original definitions. Both systems are presented in Table~\ref{tab:BasicIngredientsBialg}. Here~$\tau$ is the transposition $v \otimes w \mapsto w \otimes v$ (or the underlying braiding if one works in a symmetric category); $\sigma_{bi} \colon H \otimes H^* \to H^* \otimes H$ is defined, using Sweedler's notation, by
\begin{equation}\label{eqn:sigmabi}
\sigma_{bi}(h \otimes l)= \left\langle l_{(1)},h_{(2)}\right\rangle l_{(2)} \otimes h_{(1)};
\end{equation}
and, when writing $\sigma_{i,j}=\sigma_{Ass}$ or $\sigma_{bi}$, we mean the formulas for $\sigma_{Ass}$ or $\sigma_{bi}$ applied to the (bi)algebra corresponding to $V_i \otimes V_j$ (e.g., $\sigma_{2,4}$ in the last line is calculated according to Formula~\eqref{eqn:sigmabi} for~$H^{op,cop}$). The components $\sigma_{i,i} = \sigma_{Ass}$  are omitted.
\begin{table}\centering
\setlength{\tabcolsep}{3pt}
\begin{tabular}{|c|c|c|c|c|}
\hline  \rowcolor{lightgrey}
{structure} &\multicolumn{2}{|c|}{\cellcolor{lightgrey}  braided system} & {br. modules}   & {br. complexes} \\
\hline 
algebra & $A$ & $\sigma_{1,1}= \sigma_{Ass}$ & algebra mod.& bar,\\
\cline{2-4}
$A$ & $A$, $A^{op}$ &  $\sigma_{1,2} =\tau$  & algebra bimod. & Hochschild\\
\hline 
 & \multirow{2}*{$H,H^*$} & \multirow{2}*{$\sigma_{1,2}=\sigma_{bi}$}  & \multirow{2}*{Hopf mod.} & Gerstenhaber--Schack, \\
bialgebra &  &  &  & Panaite--{\c{S}}tefan \cite{GS90,Panaite}\\
\cline{2-5}
$H$ & $H,H^{op},$ & $\sigma_{1,2} = \tau$, $\sigma_{3,4}=\tau$, & \multirow{2}*{Hopf bimod.} & Ospel, Taillefer \\
 & $H^*,(H^*)^{op}$ & other $\sigma_{i,j}=\sigma_{bi}$ &  &  \cite{OspelThese,Taillefer}\\
\hline 
\end{tabular} \setlength{\tabcolsep}{6pt}
\caption{Braided interpretation of the algebra and the bialgebra structures}\label{tab:BasicIngredientsBialg}
\end{table}

Note the the braiding components in the systems above are not necessarily invertible. For instance, $\sigma_{bi}$ has an inverse if and only if~$H$ is a \emph{Hopf algebra}. This yields a braided interpretation of the existence of an {antipode}.

The braided system from the third line of Table~\ref{tab:BasicIngredientsBialg} yields an inclusion of the category of bialgebras in~$\vect$ into $\BrSyst_2(\vect)$. Nichita's work \cite{NiUAA,YBSyst_Entwining,Nichita_Bialg} can be seen in the same light. To encode associativity, he uses a generalization of the self-inverse braiding $\widetilde{\sigma_{Ass}}=\nu \otimes \mu+ \mu \otimes \nu - \Id_{V^{\otimes 2}}$, proposed by Nuss in the context of descent theory for noncommutative rings~\cite{Nuss}. Our~$\sigma_{Ass}$ works in more general categories, and moreover better suits for {homological applications}.

The representation-theoretic part of the article follows the philosophy of presenting complicated structures using something well understood---here modules over a well-chosen algebra. The complexity is now hidden in this algebra, which for some purposes can be treated as a black box. Table~\ref{tab:ComplicatedAlg} contains examples (for the YD example see~\cite{Lebed2ter}). Notation~$\underline{\otimes}$ is here to stress the use of braided tensor products.
\begin{table}\centering
\begin{tabular}{|c|c|} 
\hline \rowcolor{lightgrey}
{complicated structure }& {corresponding complicated algebra} \\
\hline 
bimodule over an algebra $A$ & enveloping algebra $A \otimes A^{op}$ \ \\
\hline
Hopf module over a bialgebra $H$ & Heisenberg double $\Hei(H)= H^* \underline{\otimes} H$ \\
\hline
Hopf bimodule over & algebras $\X(H) = (H \otimes H^{op})\underline{\otimes}(H^*\otimes (H^*)^{op})$,\\
a Hopf algebra $H$ & $\Y(H)$, and $\Z(H)$\\
\hline 
YD module over a bialgebra $H$ & Drinfel$'$d double $\Dr(H)=H^* \underline{\otimes} H^{op}$ \\
\hline
\end{tabular}
\caption{Algebras encoding Hopf and Yetter--Drinfel$'$d (bi)module structures}\label{tab:ComplicatedAlg}
\end{table}
Concretely, we interpret the structures from the left column as braided modules over certain braided systems of UAAs (e.g., those from Table~\ref{tab:BasicIngredientsBialg}). Further, in a very general setting we identify {braided modules} over a braided system of UAAs with modules over the corresponding braided tensor product algebra~$\leftV$:
\[\ModCatN_{(V_1,\ldots,V_r;\,\sigma_{i,i} = \sigma_{Ass},\,\xi_{i,j})} \simeq \ModCat_{\leftV}.\]
The right column of Table~\ref{tab:ComplicatedAlg} contains the relevant $\leftV$ algebras. Our general braided system theory now applies to the structures from the table. In particular, using our explicit {permutation rules} for components of a braided tensor product, we include the algebra~$\X(H)$ of Cibils--Rosso \cite{CibilsRosso} and its versions~$\Y(H)$ and~$\Z(H)$ described by Panaite~\cite{Panaite2} into a family of $\# S_4 = 24$ algebras. Explicit isomorphisms between these algebras and equivalences between their module categories are given. This circumvents the technical calculations and generalizes some results of~\cite{Panaite2}. Further, we obtain structural results for certain braided complexes---e.g., we recover the Hopf bimodule structure of the bar complex of a bialgebra with coefficients in a Hopf bimodule. 

We finish with a list of other ``braided-systematizable'' structures, the work on which is in progress.
\begin{enumerate}
\item Our braided system for generalized two-sided crossed products works in particular for \textit{$H$-(bi)(co)module algebras}. Repeating our study of bialgebra braided homology in this context, one recovers Yau's deformation bicomplex of module algebras~\cite{Yau}. Braided tools also simplify Kaygun's treatment of $H$-equivariant $A$-bimodule structures used in his Hopf--Hochschild module algebra homology~\cite{Kaygun}.
\item Combining~$\sigma_{Ass}$ with the Lie algebra braiding from~\cite{Lebed1}, one gets a rank~$2$ braided system encoding the non-commutative \textit{Poisson algebra} structure. Its braided homology includes Fresse's Poisson algebra homology~\cite{Fresse}.
\item The braided system machinery also applies to the \textit{quantum Koszul complexes} of Gurevich and Wambst \cite{Gurevich,Wambst}. 
\end{enumerate}

\section*{Notations and conventions} 
All our structures live in a {strict mono\-idal} category $(\C,\otimes,\II)$; the reader can have in mind the category $\Vect$ of vector spaces over a field~$\k$ for simplicity. The word ``strict'' is often omitted for brevity, as well as the word ``monoidal'' in the terms ``braided / symmetric monoidal category''. Given an object~$V$ in~$\C$, we succinctly denote its tensor powers by $V^n=V^{\otimes n}$, $V^0=\II$. Further, given a morphism $\varphi \colon V^{ l}\to V^{ r},$ the following notation is repeatedly used:
\begin{equation}\label{eqn:phi_i}
\varphi^i = \Id_V^{\otimes (i-1)}\otimes\varphi \otimes \Id_V^{\otimes (k-i+1)} \quad \colon \quad V^{k+l}\rightarrow V^{k+r},
\end{equation}
and similarly for morphisms on tensor products of different objects. Working with a family of objects $(V_1,V_2,\ldots)$, we put $\Id_i=\Id_{V_i}$.

The already classical \emph{graphical calculus} is extensively used in this article. Dots denote objects in $\C$; horizontal gluing represents tensor product; graph diagrams encode morphisms from the object corresponding to the lower dots to that corresponding to the upper dots; vertical gluing stands for morphism composition, and vertical strands for identities. All diagrams read {from bottom to top}.

Notations $S_n, B_n, B_n^+$ stand for the symmetric groups, the braid groups, and the positive braid monoids. Their standard generators are denoted by, respectively, $s_i$ and $\sigma_i$, $ 1 \le i \le n-1$.

\ack The author is grateful to Marc Rosso for sharing his passion for quantum shuffles; to Muriel Livernet, Fr\'ed\'eric Chapoton, and Fr\'ed\'eric Patras for illuminating discussions; to Paul-Andr\'e Melli\`es and Eugenia Cheng for pointing out connections between this work and recent results involving distributive laws in category theory; and to the reviewer for helpful questions and remarks.

\section{Braided vocabulary}\label{sec:BrSyst} 

The notion of braided system generalizes the more familiar braided objects.

\begin{definition}
\begin{itemize}
\item 
A \emph{rank~$r$ braided system} in $\C$ is an ordered family $V_1,V_2,\ldots, V_r$ of objects endowed with a \emph{braiding}, i.e., morphisms
 $\sigma_{i,j} \colon V_i \otimes V_j \to V_j \otimes V_i$ for $1\le {i  \le j} \le r$  satisfying the \emph{(colored) Yang--Baxter equation}
 \begin{equation}\label{eqn:YB}
 (\sigma_{j,k}\otimes \Id_i)(\Id_j \otimes \sigma_{i,k})(\sigma_{i,j}\otimes \Id_k) =(\Id_k \otimes \sigma_{i,j})(\sigma_{i,k}\otimes \Id_j)(\Id_i \otimes \sigma_{j,k})
\end{equation}
on all the tensor products $V_i \otimes V_j \otimes V_k$ with $1\le {i \le j \le k} \le r$. Such a system is denoted by $((V_i)_{1\le i\le r};(\sigma_{i,j})_{1\le i\le j\le r})$ or briefly $(\oV,\osigma)$.
\item A \emph{braided morphism} $\of \colon (\oV,\osigma) \to (\oW,\oxi)$ between two braided systems in $\C$ of the same rank $r$ is a collection of morphisms $(f_i \in \Hom_\C(V_i,W_i))_{1\le i \le r}$  respecting the braiding, in the sense that, for all $1\le {i  \le j} \le r$, one has
\begin{equation}\label{eqn:BrMor}
(f_j \otimes f_i)  \sigma_{i,j} = \xi_{i,j}  (f_i \otimes f_j).
\end{equation}
\item The category of rank $r$ braided systems and braided morphisms in $\C$ is denoted by $\BrSyst_r(\C)$.
\item Rank~$1$ braided systems are called \emph{braided objects} in $\C$.
\item For $1\le s  \le t \le r$, the \emph{braided $(s,t)$-subsystem} of $(\oV,\osigma)$, denoted by $(\oV,\osigma)[s,t]$, is the subfamily  $V_s,\ldots, V_t$ with the corresponding components $\sigma_{i,j}$ of $\osigma$.
\end{itemize}
\end{definition}

The notion of braiding thus defined is
\begin{enumerate}
\item { positive}: the $\sigma_{i,j}$ are not supposed to be invertible (the term \emph{pre-braiding} is sometimes used in such situations);

\item { partial}, i.e., defined only on certain couples of objects;

\item { local}: contrary to the usual notion of braiding in a monoidal category, no naturality is imposed.
\end{enumerate}

Graphically, a braiding component is represented as a braid whose strands are ``colored'' with the corresponding objects~$V_i$, or simply with the indices~$i$ (Fig.~\ref{pic:YB}{A}). The definition allows a $j$-colored strand to overcross only strands colored with indices $i \le j$.
The diagrammatic counterpart of the (colored) YBE is now the (colored) third Reidemeister move (Fig.~\ref{pic:YB}{B}), which is at the heart of braid theory. One can thus work with braided systems by manipulating positive braid diagrams.

\begin{figure}\centering
\begin{tikzpicture}[scale=0.5]
\node  at (-1.8,0.5)  {$\sigma_{i,j} \, \longleftrightarrow $};
\draw [thick, colori, rounded corners](0,0)--(0,0.25)--(0.4,0.4);
\draw [thick, colori, rounded corners](0.6,0.6)--(1,0.75)--(1,1);
\draw [thick, colorj, rounded corners](1,0)--(1,0.25)--(0,0.75)--(0,1);
\node  at (0,0) [colori, below] {$\scriptstyle i$};
\node  at (1,0) [colorj, below] {$\scriptstyle j$};
\node  at (2.7,0)  {\circled{A}};
\node  at (6,-1.5)  { };
\end{tikzpicture}
\begin{tikzpicture}[xscale=0.5,yscale=0.4]
\node  at (-3,1.5)  {YBE $\quad \longleftrightarrow $};
\draw [thick, colori, rounded corners](0,0)--(0,0.25)--(0.4,0.4);
\draw [thick, colori, rounded corners](0.6,0.6)--(1,0.75)--(1,1.25)--(1.4,1.4);
\draw [thick, colori, rounded corners](1.6,1.6)--(2,1.75)--(2,3);
\draw [thick, colorj, rounded corners](1,0)--(1,0.25)--(0,0.75)--(0,2.25)--(0.4,2.4);
\draw [thick, colorj, rounded corners](0.6,2.6)--(1,2.75)--(1,3);
\draw [thick, colork, rounded corners](2,0)--(2,1.25)--(1,1.75)--(1,2.25)--(0,2.75)--(0,3);
\node  at (0,0) [colori, below] {$\scriptstyle i$};
\node  at (1,0) [colorj, below] {$\scriptstyle j$};
\node  at (2,0) [colork, below] {$\scriptstyle k$};
\node  at (3.2,1.5){$=$};
\end{tikzpicture}
\begin{tikzpicture}[xscale=0.5,yscale=0.4]
\node  at (-0.2,1.5){};
\draw [thick, colorj, rounded corners](1,1)--(1,1.25)--(1.4,1.4);
\draw [thick, colorj, rounded corners](1.6,1.6)--(2,1.75)--(2,3.25)--(1,3.75)--(1,4);
\draw [thick, colori, rounded corners](0,1)--(0,2.25)--(0.4,2.4);
\draw [thick, colori, rounded corners](0.6,2.6)--(1,2.75)--(1,3.25)--(1.4,3.4);
\draw [thick, colori, rounded corners](1.6,3.6)--(2,3.75)--(2,4);
\draw [thick, colork, rounded corners](2,1)--(2,1.25)--(1,1.75)--(1,2.25)--(0,2.75)--(0,4);
\node  at (0,1) [colori, below] {$\scriptstyle i$};
\node  at (1,1) [colorj, below] {$\scriptstyle j$};
\node  at (2,1) [colork, below] {$\scriptstyle k$};
\node  at (4,1)  {\circled{B}};
\end{tikzpicture}
\caption{Braided systems versus colored braids}\label{pic:YB}
\end{figure}
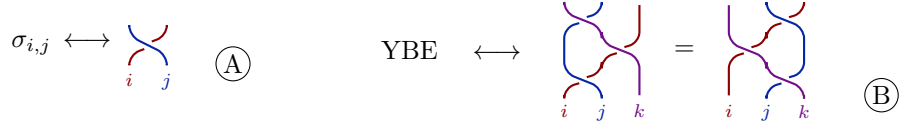 

Each component of a braided system is a braided object. Even better:

\begin{proposition}\label{thm:TrivialBrSystem}
Given a {braided} category $(\C,\otimes,\II,c)$, one has, for all $r \in \NN$, a fully faithful functor
\begin{align}
(\BrSyst_1(\C))^{\times r} &\longhookrightarrow \BrSyst_r(\C),\notag\\
(V_i,\sigma_i)_{1 \le i \le r} &\longmapsto (V_1,\ldots,V_r;\sigma_{i,i}:= \sigma_{i}, \sigma_{i,j} := c_{V_i,V_j} \text{ for } i < j),\label{eqn:TrivialBrSystem}\\
(f_i:V_i\rightarrow W_i)_{1 \le i \le r} &\longmapsto \of := (f_i)_{1 \le i \le r}.\notag
\end{align}
\end{proposition}

\begin{proof}
There are three types of tensor products on which one should check the colored YBE~\eqref{eqn:YB} in order to verify that \eqref{eqn:TrivialBrSystem} defines a braided system:
\begin{enumerate}
\item On $V_i\otimes V_i\otimes V_i$, \eqref{eqn:YB} is simply the YBE for $\sigma_{i}$.
\item On $V_i\otimes V_i\otimes V_j$ and $V_i\otimes V_j\otimes V_j$, $i<j$, \eqref{eqn:YB} expresses the naturality of $c$ w.r.t. $\sigma_{i}$ and $\sigma_{j}$ respectively, which always holds in a braided category.
\item On $V_i\otimes V_j\otimes V_k$, $i<j<k$, \eqref{eqn:YB} coincides with the YBE for the categorical braiding $c$, which is again automatic in a braided category.
\end{enumerate}

Now, for morphisms, condition \eqref{eqn:BrMor} is automatic for $i < j$ thanks to the naturality of $c$, and for $i=j$ it is equivalent to $f_i$ being a braided morphism. Thus our functor is well defined, full, and faithful on morphisms.
\end{proof}

\begin{observation}\label{obs:NegSigma}
If $\C$ is {preadditive}, then for all~$r$ one has a category automorphism
\begin{align*}
\BrSyst_r(\C) &\overset{\sim}{\longleftrightarrow} \BrSyst_r(\C),\\
(\oV; (\sigma_{i,j})_{1\le i\le j\le r}) &\longleftrightarrow (\oV; (-\sigma_{i,j})_{1\le i\le j\le r}),\\
\of &\longleftrightarrow \of.
\end{align*}
\end{observation}

\begin{definition}
\begin{itemize}
\item A \emph{right (braided) module} over $(\oV,\osigma) \in \BrSyst_r(\C)$  is an object~$M$ equipped with morphisms $\overline{\rho}=(\rho_i \colon M\otimes V_i \to M)_{1 \le i \le r}$ satisfying, for all $1 \le i \le j \le r$,
\begin{equation}\label{eqn:MultiBrMod}
\rho_j  (\rho_i \otimes \Id_{j})=\rho_i  (\rho_j \otimes \Id_{i}) (\Id_M\otimes \sigma_{i,j})\quad : \quad M\otimes V_i \otimes V_j \rightarrow M. 
\end{equation} 
\item Left braided modules and left/right braided comodules, as well as braided (co)module morphisms, are defined in a similar way. 
\item The category of right braided modules and their morphisms is denoted by $\ModCat_{(\oV,\osigma)}$. Notation $_{(\oV,\osigma)}\!\ModCat$ is used in the left case, and $\ModCat^{(\oV,\osigma)}$  and $^{(\oV,\osigma)}\!\ModCat$ in the co-cases.
\end{itemize}
\end{definition}

As shown in Fig.~\ref{pic:MultiBrMod}, braided modules can be handled by manipulating a particular type of knotted trivalent graphs; see. \cite{KauffmanGraphs,YamadaGraphs,YetterGraphs} for the theory of the latter.
\begin{figure}\centering
\begin{tikzpicture}[xscale=0.4,yscale=0.35]
 \node at (-8,1) {braided module $\qquad \longleftrightarrow $};
 \draw [ultra thick, colorM] (0,0) -- (0,3);
 \draw [thick, colori] (1,0) -- (0,1);
 \draw [thick, colorj] (2,0) -- (0,2);
 \node at (0,2) [left]{$\scriptstyle \rho_{\color{colorj} j}$};
 \node at (0,1) [left]{$\scriptstyle \rho_{\color{colori} i}$};
 \fill[colorj] (0,2) circle (0.2);
 \fill[colori] (0,1) circle (0.2); 
 \node at (2,0) [colorj,below] {$\scriptstyle j$};
 \node at (1,0) [colori,below] {$\scriptstyle i$};
 \node at (0,0) [colorM,below] {$\scriptstyle M$};
 \node  at (3.5,1.5){$=$};
\end{tikzpicture}
\begin{tikzpicture}[xscale=0.4,yscale=0.35]
\node  at (-1.5,1.5){};
 \draw [colorM, ultra thick] (0,0) -- (0,3);
 \draw [thick, colori] (1,0) -- (0.87,0.35);
 \draw [thick, colori] (0.55,0.9) -- (0,2);
 \draw [thick, colorj] (2,0) -- (0,1);
 \node at (0,1) [left]{$\scriptstyle \rho_{\color{colorj} j}$};
 \node at (0,2) [left]{$\scriptstyle \rho_{\color{colori} i}$};
 \fill[colori] (0,2) circle (0.2);
 \fill[colorj] (0,1) circle (0.2); 
 \node at (2,0) [colorj,below] {$\scriptstyle j$};
 \node at (1,0) [colori,below] {$\scriptstyle i$};
 \node at (0,0) [colorM,below] {$\scriptstyle M$};
 \node at (1.5,1.2) {$ \sigma_{{\color{colori} i},{\color{colorj} j}}$};
\end{tikzpicture}
\caption{Right braided module}\label{pic:MultiBrMod}
\end{figure}
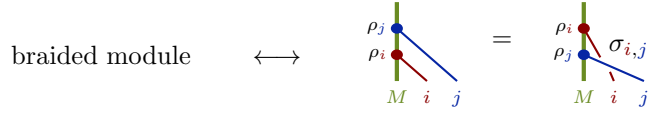

In this article and in~\cite{Lebed2ter} we interpret, among others, algebra bimodules and Hopf and Yetter--Drinfel$'$d modules as modules over certain braided systems.

\begin{observation}\label{rmk:MultiMod}
A $(\oV,\osigma)$-module structure on~$M$ boils down to a collection of $(V_i, \sigma_{i,i})$-module structures on~$M$, compatible in the sense of~\eqref{eqn:MultiBrMod}. 
\end{observation}

\begin{observation}\label{rmk:MultiMod2}
In an additive category, $(\oV,\osigma)$-modules can also be viewed as modules over the associative algebra $\raisebox{.5mm}{$T(V)$}\big / \raisebox{-.5mm}{$\langle \sigma - \Id \rangle$}$, where $V=V_1\oplus V_2 \oplus\cdots\oplus V_r$ amalgamates all the components of our system, and $\langle \sigma - \Id \rangle$ is the ideal generated by the images of the maps $\sigma_{i,j} - \Id_i \otimes \Id_j \colon V_i \otimes V_j \to V_j \otimes V_i + V_i \otimes V_j \hookrightarrow V \otimes V$.
\end{observation}

The notions of right and left $(\oV,\osigma)$-modules coincide for the unit object $\II$. Condition~\eqref{eqn:MultiBrMod} takes in this case a simpler form
$(\rho_j \otimes \rho_i ) \sigma_{i,j} =\rho_i \otimes \rho_j \, \colon \, V_i \otimes V_j \to \II$.

\begin{definition}
A \emph{braided character} is a right (= left) $(\oV,\osigma)$-module structure on~$\II$.
\end{definition}

\begin{example}\label{rmk:single_br_char}
In a preadditive~$\C$, a braided character $\varepsilon_i$ on any~$V_i$ extended to other components by zero becomes a $(\oV,\osigma)$-character.
\end{example}

The invertibility of some of the $\sigma_{i,j}$ is helpful in extending braided structures. It also allows one to interchange the corresponding components of a braided system without changing the module category:

\begin{proposition}\label{thm:BrSystInv}
Take $(\oV,\osigma) \in \BrSyst_r(\C)$  with $ \sigma_{p,p+1}$ invertible for some~$p$.
 \begin{enumerate}
\item The family $(V_1, \ldots, V_{p-1}, V_{p+1}, V_p, V_{p+2},\ldots, V_r)$, equipped with the old $\sigma_{i,j}$ on the tensor products $V_i \otimes V_j$ with $(i,j) \neq (p+1,p)$ and with $\sigma_{p,p+1}^{-1}$ on $V_{p+1} \otimes V_p$, is a braided system, denoted by $s_p(\oV,\osigma)$.
\item The categories of braided modules for the original and the rearranged systems are equivalent: 
$\ModCat_{(\oV,\osigma)} \simeq \ModCat_{s_p(\oV,\osigma)}$.
\end{enumerate}
\end{proposition}

\begin{proof} Notation~\eqref{eqn:phi_i} is used throughout the proof.
 \begin{enumerate}
\item\label{it:perm1} One has to check four types of new instances of the colored YBE. 
\begin{enumerate}
\item On $V_i \otimes V_{p+1} \otimes V_p$ with $i < p$, the YBE reads
\[\sigma_{i,p+1}^2  \sigma_{i,p}^1  (\sigma_{p,p+1}^{-1})^2 = (\sigma_{p,p+1}^{-1})^{1}  \sigma_{i,p}^2  \sigma_{i,p+1}^1,\]
or equivalently
\[\sigma_{p,p+1}^1  \sigma_{i,p+1}^2  \sigma_{i,p}^1  = \sigma_{i,p}^2  \sigma_{i,p+1}^1  \sigma_{p,p+1}^2.\]
This is precisely the YBE on $V_i \otimes V_p \otimes V_{p+1}$ for the original system $(\oV,\osigma)$.

The remaining types are similar, and can be summarized as follows:

\item For $j > p+1$, the YBE on $V_{p+1} \otimes V_p \otimes V_j$ for $s_p(\oV,\osigma)$ is equivalent to the YBE on $V_p \otimes V_{p+1} \otimes V_j$ for $(\oV,\osigma)$.

\item  The YBE on $V_{p+1} \otimes V_{p+1} \otimes V_p$ for $s_p(\oV,\osigma)$ is equivalent to the YBE on $V_p \otimes V_{p+1} \otimes V_{p+1}$ for $(\oV,\osigma)$.

\item  The YBE on $V_{p+1} \otimes V_p \otimes V_p$ for $s_p(\oV,\osigma)$ is equivalent to the YBE on $V_p \otimes V_p \otimes V_{p+1}$ for $(\oV,\osigma)$.
\end{enumerate} 
 
\item
Given an object $M$ equipped with the morphisms $\rho_i \colon M\otimes V_i \to M$, the list of compatibility conditions~\eqref{eqn:MultiBrMod} one has to check for $(\oV,\osigma)$ differs from the list for $s_p(\oV,\osigma)$ only in the conditions for $i=p$, $j=p+1$: 
\begin{align*}
&\rho_{p+1}  (\rho_p \otimes \Id_{{p+1}})=\rho_p  (\rho_{p+1} \otimes \Id_{p}) (\Id_M\otimes \sigma_{p,p+1})\\
\text{versus}\qquad\qquad & \rho_p  (\rho_{p+1} \otimes \Id_{p})=\rho_{p+1}  (\rho_p \otimes \Id_{{p+1}}) (\Id_M\otimes \sigma_{p,p+1}^{-1}).
\end{align*}
The second one composed with the invertible morphism $\Id_M\otimes \sigma_{p,p+1}$ on the right yields the first one.
So the identity functor of~$\C$ and the permutation $\rho_p \leftrightarrow \rho_{p+1}$ of the components of $\overline{\rho}$ give the announced category equivalence. \qedhere
\end{enumerate}
\end{proof}

\begin{remark}\label{thm:BrSystInvSn}
More generally, fix a permutation $\theta \in S_r$, and take $(\oV,\osigma) \in \BrSyst_r(\C)$ with the $\sigma_{i,j}$ invertible for all~$i,j$ reversed by $\theta$. The family $(V_{\theta^{-1}(1)}, \ldots, V_{\theta^{-1}(r)})$, equipped with the old $\sigma_{i,j}$ on $V_i \otimes V_j$ with $\theta(i) < \theta(j)$ and with $\sigma_{i,j}^{-1}$ on the remaining couples, is a braided system, denoted by $\theta(\oV,\osigma)$. This yields a \emph{partial $S_r$-action} on $\BrSyst_r(\C)$ and equivalences between the corresponding braided module categories. Notations $s_p(\oV,\osigma)$ and $\theta(\oV,\osigma)$ are motivated by this remark.
\end{remark}

\begin{corollary}\label{thm:BrSystInvGlue}
 Let $(\oV,\osigma)$ be a braided system in an {additive} monoidal~$\C$, with $\sigma_{i,j}$ invertible for all $s\le i < j \le t$. Then one can glue the objects $V_s,\ldots,V_t$ together into $ V_{s:t}:= \bigoplus_{i=s}^t V_i$, and extend the braiding onto $(V_1,\ldots,V_{s-1},V_{s:t},$ $V_{t+1}, \ldots,V_r)$ by putting $\sigma|_{V_j \otimes V_i} := \sigma_{i,j}^{-1}$ for all $s\le i < j \le t$.
\end{corollary}

Note that the invertibility of $\sigma_{i,i}$ is not required here even for $s\le i \le t$. 

\begin{proof}
We consider only the case $s=t-1=:p$; the general case follows by induction. The colored YBEs appearing here come from the systems $(\oV,\osigma)$ and $s_p(\oV,\osigma)$, except for the YBE on $V_p \otimes V_{p+1} \otimes V_p$ and on $V_{p+1} \otimes V_p \otimes V_{p+1}$. To handle these last two cases, observe that the argument from the proof of Proposition~\ref{thm:BrSystInv}, Point~\ref{it:perm1} remains valid for $i=p$ and $j=p+1$.
\end{proof}

A particular case of Corollary~\ref{thm:BrSystInvGlue} yields the \emph{gluing procedure} for Yang--Baxter operators (or, in our terms, for braided objects), studied by Majid and Markl~\cite{MajidMarkl}.

\section{A homology theory for braided systems}\label{sec:BrSystemHomology}

We now generalize the \emph{braided (co)homology} theory, developed in~\cite{Lebed1} for braided objects in~$\C$, to braided systems $(\oV,\osigma)$. In this section $\C$ is additive monoidal. In particular, the collection $\osigma$ assembles into a \emph{partial braiding}, still denoted by $\osigma$, on
\[V:=V_1\oplus V_2\oplus\cdots\oplus V_r,\]
and the family $\overline{\rho}$ defining a right $(\oV,\osigma)$-module $M$ assembles into $\rho \colon M \otimes V \to M$.

We first show that the collection~$\osigma$ suffices for a partial version of Rosso's quantum shuffle (co)products \cite{Rosso1Short,Rosso2}. Recall that the \emph{shuffle sets} are the permutation sets
\[Sh_{p_1,p_2,\ldots, p_k}=\Bigg\{ \theta \in S_{p_1+p_2+\cdots+p_k}  \;
\begin{array}{|l}
\scriptstyle \theta(1)<\theta(2)<\ldots<\theta(p_1), \\
\scriptstyle \theta(p_1+1)<\ldots<\theta(p_1+p_2), \\
\scriptstyle \ldots,\,
\scriptstyle \theta(p+1)<\ldots<\theta(p+p_k)
 \end{array}
 \Bigg\} \]
with $p=p_1+\cdots+p_{k-1}$. Think of permuting $p_1+p_2+\cdots+p_{k}$ elements preserving the order within $k$ consecutive blocks of size $p_1, \ldots, p_k$, just like when shuffling cards.

Recall further the projection $B_n^+ \twoheadrightarrow S_n$ sending a generator $\sigma_i$ to the corresponding generator $s_i$, and its set-theoretic \emph{Matsumoto section}
\begin{align}
 S_n & \longhookrightarrow  B_n^+, \notag\\
 \theta = s_{i_1}s_{i_2}\cdots s_{i_k}& \longmapsto \sigma_{i_1}\sigma_{i_2}\cdots \sigma_{i_k}, \notag \end{align}
where $s_{i_1}s_{i_2}\cdots s_{i_k}$ is any of the shortest words representing $\theta\in S_n$.

\begin{notation}
We denote by $\T_\theta$ the image of $\theta \in S_n$ under this map.
\end{notation}

We also need a partial $B_d^+$-action on~$V^d$ for  $(\oV,\osigma)\in \BrSyst_r(\C)$. For a generator~$\sigma_i$ of~$B_d^+$ and a summand $V_{k_1} \otimes \ldots \otimes V_{k_d}$ of~$V^d$, $k_1 \le \ldots \le k_d$, it reads
\begin{align*}
\sigma_i \qquad  \longmapsto & \qquad \sigma_{k_i,k_{i+1}}^i \quad = \quad \begin{tikzpicture}[baseline=2pt,scale=0.4]
\draw (-3,0)--(-3,1);
\draw (-1,0)--(-1,1);
\draw [ colori, rounded corners](0,0)--(0,0.25)--(0.4,0.4);
\draw [ colori, rounded corners](0.6,0.6)--(1,0.75)--(1,1);
\draw [ colorj, rounded corners](1,0)--(1,0.25)--(0,0.75)--(0,1);
\draw (2,0)--(2,1);
\draw (4,0)--(4,1);
\node  at (-3,0) [below] {$\scriptstyle k_1$};
\node  at (-1.7,0) [below] {$\cdots$};
\node  at (-0.2,0) [colori, below] {$\scriptstyle k_i$};
\node  at (1.2,0) [colorj, below] {$\scriptstyle k_{i+1}$};
\node  at (2.7,0) [below] {$\cdots$};
\node  at (4,0) [below] {$\scriptstyle k_d$};
\end{tikzpicture}.
\end{align*}
Here $\sigma_{k_i,k_{i+1}}^i \in \Hom_\C(V_{k_1} \otimes \cdots \otimes V_{k_d},V_{k_1} \otimes \cdots \otimes V_{k_{i+1}} \otimes V_{k_i} \otimes \cdots \otimes V_{k_d})$. This action agrees with the usual graphical depiction of braids from~$B_d^+$.

\begin{notation}
The partial action described above is denoted by $B_d^+ \ni b \mapsto b^{\osigma}$.
\end{notation} 

\begin{definition}
A degree~$d$ \emph{(reverse) ordered tensor product} for $(\oV,\osigma)\in \BrSyst_r(\C)$ is a tensor product $V_{k_1} \otimes \ldots \otimes V_{k_d}$ with $k_1 \le \ldots \le k_d$ (respectively, $k_1 \ge \ldots \ge k_d$). The direct sum of all such products is denoted by $T(\oV)_d^\rightarrow$ (respectively, $T(\oV)_d^\leftarrow$).
\end{definition}

The $T(\oV)_d^\rightarrow$ sum up to $T(\oV)^\rightarrow :=T(V_1)\otimes T(V_2)\otimes\cdots \otimes T(V_r)$, and the $T(\oV)_d^\leftarrow$ sum up to $T(\oV)^\leftarrow :=T(V_r)\otimes T(V_{r-1})\otimes\cdots \otimes T(V_1)$.

Armed with these notations, we give multi-versions of quantum shuffle operations.
 
\begin{definition}\label{def:qu_sh}
Take a braided system $(\oV,\osigma)$ in $\BrSyst_r(\C)$. 
\begin{itemize}
\item The \emph{multi-quantum shuffle product} is defined by
\begin{equation}\label{eqn:qu_sh}
 \osh {}_{p,q}= \sum_{\theta\in Sh_{p,q}} (\T_\theta)^{\osigma} \quad : \quad  {T(\oV)_p^\leftarrow}\otimes T(\oV)_q^\leftarrow \rightarrow  T(\oV)_{p+q}^\leftarrow,
\end{equation} 
where $(\T_\theta)^{\osigma} (W \otimes U)$ is declared zero when it is undefined or misses the $T(\oV)_{p+q}^\leftarrow$ part of $V^{p+q}$.
\item Dually, the \emph{multi-quantum shuffle coproduct} is defined by
\begin{equation}\label{eqn:qu_cosh}
\ocsh {}_{p,q}= \sum_{\theta\in Sh_{p,q}} (\T_{\theta^{-1}})^{\osigma}  \quad : \quad  T(\oV)_{p+q}^\rightarrow\rightarrow T(\oV)_p^\rightarrow \otimes T(\oV)_q^\rightarrow.
\end{equation}
\item Replacing $Sh_{p,q}$ with $Sh_{p_1, \ldots, p_k}$, one gets morphisms $\osh {}_{p_1, \ldots, p_k}$ and $\ocsh {}_{p_1, \ldots, p_k}$.
\end{itemize}
\end{definition}

Condition~\eqref{eqn:qu_sh} should be thought of as the dual of the more intuitive condition~\eqref{eqn:qu_cosh}. 

For an ordered tensor products~$W$ in $T(\oV)_{p+q}^\rightarrow$, its image $\ocsh {}_{p,q}(W)$ lives in several summands of $T(\oV)_p^\rightarrow \otimes T(\oV)_q^\rightarrow$. That is why $\C$ has to be additive. The case of rank $r=1$ is exceptional: one needs a preadditive~$\C$ only.

\begin{proposition}\label{prop:MultiShuffle}
Morphisms \eqref{eqn:qu_sh}-\eqref{eqn:qu_cosh} are well defined, and give an associative multiplication (respectively, a coassociative comultiplication). 
\end{proposition}

\begin{proof}
When (reverse) ordered products are fed into formulas \eqref{eqn:qu_sh}-\eqref{eqn:qu_cosh}, the braiding~$\osigma$ is applied only to products $V_i\otimes V_j$ with $i \le j$, on which it is defined. The (co)associativity is proved as in the rank~$1$ case (see \cite[Theorem 1]{Lebed1}).
\end{proof}

We now explain what we mean by a \emph{homology theory} for a braided system $(\oV,\osigma)$.

\begin{definition}\label{def:DiffCat}
\begin{itemize}
\item  \emph{A  differential} for $(\oV,\osigma)$ is a morphism family $\{\, d_n \colon T(\oV)_n^\rightarrow \to T(\oV)_{n-1}^\rightarrow \,\}_{n>0}$ satisfying $d_{n-1} d_n =0$ for all $n>1$.

\item \emph{A bidifferential} for $(\oV,\osigma)$ consists of $2$ families $\{\, d_n,d'_n\colon T(\oV)_n^\rightarrow \to T(\oV)_{n-1}^\rightarrow\, \}_{n>0}$  satisfying $d_{n-1}  d_n = d'_{n-1}  d'_n = d'_{n-1}  d_n +d_{n-1}  d'_n = 0$ for all $n>1$. 

\item Replacing $T(\oV)_n^\rightarrow$ with $M \otimes T(\oV)_n^\rightarrow \otimes N$ (for some objects~$M$ and~$N$) above, one gets the notion of (bi)differentials with \emph{coefficients} in~$M$ and~$N$.
\end{itemize}
\end{definition}

Everything is now ready for constructing a multi-version of braided complexes.

\begin{theorem}\label{thm:BraidedSimplHomCat}
Take a braided system $(\oV,\osigma)$ in an additive monoidal category~$\C$. Let $(M,\overline{\rho})$ and $(N,\overline{\lambda})$ be a right and, respectively, left $(\oV,\osigma)$-modules. The morphisms 
\begin{align*}
({^{\rho}}\! d)_n &= ({\rho} \otimes \Id_{T(\oV)_{n-1}^\rightarrow \otimes N}) (\Id_M \otimes \Ocsh {}_{1,n-1} \otimes \Id_N),\\
(d^{\lambda})_n &= (-1)^{n-1}(\Id_{M \otimes T(\oV)_{n-1}^\rightarrow} \otimes \lambda) (\Id_M \otimes \Ocsh {}_{n-1,1} \otimes \Id_N)
\end{align*}
from $M \otimes T(\oV)_{n}^\rightarrow \otimes N$ to $M \otimes T(\oV)_{n-1}^\rightarrow \otimes N$ then define a bidifferential with coefficients in~$M$ and~$N$. (Here~$-\osigma$ is the braiding obtained from~$\osigma$ as in Observation~\ref{obs:NegSigma}.)
\end{theorem}

\begin{proof}
Our verifications use (A) the coassociativity of~$\Ocsh$ (Proposition~\ref{prop:MultiShuffle}), and (B) the definition of braided modules, reformulated in an additive~$\C$ as
\begin{align*}
\rho  (\rho \otimes \Id_V) (\Id_M \otimes \Ocsh {}_{1,1}) &= 0,&
\lambda  (\Id_V \otimes \lambda) (\Ocsh{}_{1,1} \otimes \Id_N) &= 0.
\end{align*}
Concretely, writing $\Ocsh$ instead of $\Id_M \otimes\Ocsh\otimes \Id_N$ or $\Ocsh\otimes \Id_N$  for brevity, one has

\begin{align*}
({^{\rho}}\! d)_{n-1}  ({^{\rho}}\! d)_n  &= (\rho \otimes \Id_{\cdots})  \Ocsh{}_{1,n-2}  (\rho \otimes \Id_{\cdots})  \Ocsh{}_{1,n-1}  \\
  &= (\rho \otimes \Id_{\cdots}) (\rho \otimes \Id_{\cdots}) (\Id_{M\otimes V} \otimes \Ocsh{}_{1,n-2})  \Ocsh{}_{1,n-1}  \\
  & \overset{(A)}{=} (\rho \otimes \Id_{\cdots}) (\rho \otimes \Id_{\cdots}) (\Id_M \otimes \Ocsh{}_{1,1}\otimes \Id_{\cdots})  \Ocsh{}_{2,n-2}  \\
& = ((\rho  (\rho \otimes \Id_V) (\Id_M \otimes \Ocsh{}_{1,1}))\otimes \Id_{\cdots})  \Ocsh{}_{2,n-2}   \overset{(B)}{=} 0,
\end{align*}
and similarly for~$d^\lambda$. In the same way, one calculates 
\begin{align*}
&(d^\lambda)_{n-1}  ({^{\rho}}\! d)_n  = (-1)^{n-2} (\rho \otimes \Id \otimes \lambda)  \Ocsh{}_{1,n-2,1} = - ({^{\rho}}\!d)_{n-1}  ( d^\lambda)_n. \qedhere
\end{align*}
\end{proof}

The differential $({^\rho}\! d)_n$ is a signed sum (due to the negative braiding $-\osigma$) of the form  $\sum_{i=1}^n (-1)^{i-1} ({^\rho}\! d)_{n;i}$. The term $({^\rho}\! d)_{n;i}$ is presented in Fig.~\ref{pic:BrDiffCoeffs}. Its sign can be read off its diagram as the crossing number. Similar holds for $(d^\lambda)_n$.
\begin{figure}\centering
\begin{tikzpicture}[xscale=0.6,yscale=0.5]
 \node at (-2,1.5) {$ ({^\rho}\! d)_{n;i} =$}; 
 \draw [colorM, ultra thick] (0,0) -- (0,3);
 \draw [colorN, ultra thick] (8,0) -- (8,3);
 \draw [colori,thick] (1,0) -- (1,1.3);
 \draw [colori,thick] (1,1.7) -- (1,3);
 \draw [colorj,thick] (3,0) -- (3,0.3);
 \draw [colorj,thick] (3,0.7) -- (3,3);
 \draw [colork,thick] (4,0) -- (0,2);
 \draw [colorl,thick] (5,0) -- (5,3); 
 \draw [colorn,thick] (7,0) -- (7,3); 
 \node at (0.9,1.65) [right]{$ \sigma_{{\color{colori} k_1},{\color{colork} k_i}}$};
 \node at (2.9,0.65) [right]{$ \sigma_{{\color{colorj} k_{i-1}},{\color{colork} k_i}}$};
 \node at (0,2) [left]{$ \rho_{\color{colork} k_i}$};
 \fill[colork] (0,2) circle (0.1);
 \node at (0,0) [below,colorM] {$\scriptstyle  M$};
 \node at (8,0) [below,colorN] {$\scriptstyle  N$};
 \node at (1,0) [below,colori]{$\scriptstyle {k_1}$};
 \node at (3,0) [below,colorj]{$\scriptstyle {k_{i-1}}$};
 \node at (2,0) [below]{$\scriptstyle \ldots$};
 \node at (4,0) [below,colork]{$\scriptstyle {k_i}$};
 \node at (5,0) [below,colorl]{$\scriptstyle {k_{i+1}}$};
 \node at (6,0) [below]{$\scriptstyle \ldots$};
 \node at (7,0) [below,colorn]{$\scriptstyle {k_n}$};
\end{tikzpicture}
\caption{Braided left differential}\label{pic:BrDiffCoeffs}
\end{figure}
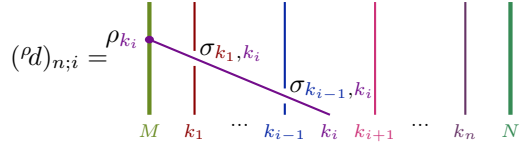

\begin{corollary}
Any $\ZZ$-linear combination of the families $({^{\rho}}\! d)_n$ and $(d^{\lambda})_n$ from the theorem is a differential for $(\oV,\osigma)$  with {coefficients} in~$M$ and~$N$.
\end{corollary}

The {(bi)differentials} from the above theorem and corollary are called \emph{braided}.

\begin{remark}
\begin{itemize}
\item Braided differentials are \emph{functorial}. Concretely, take systems $(\oV,\osigma)$ and $(\oV',\osigma')$; braided modules $(M,\overline{\rho}) \in \ModCat_{(\oV,\osigma)}$, $(N,\overline{\lambda})\in {_{(\oV,\osigma)}}\!\ModCat$, and similarly for $(\oV',\osigma')$; braided morphism  $\of \colon (\oV,\osigma) \to (\oV',\osigma')$; and morphisms ${\varphi} \colon M \to M'$, ${\psi} \colon N \to N'$, compatible with~$\of$ in the sense of $\rho'_i  (\varphi \otimes f_i) = \varphi  \rho_i$ and $\lambda'_i  (f_i \otimes \psi) = \psi  \lambda_i$ for all~$i$. Then one has the intertwining diagram
\[\ \xymatrix @!0 @R=1cm @C=6cm{
    M \otimes T(\oV)_{n}^\rightarrow \otimes N \ar_{d_n}[d] \ar^{\varphi \otimes \of^{\otimes n} \otimes \psi}[r]  & M' \otimes T(\oV')_{n}^\rightarrow \otimes N' \ar_{d'_n}[d] \\
    M \otimes T(\oV)_{n-1}^\rightarrow \otimes N \ar^{\varphi \otimes \of^{\otimes n-1} \otimes \psi}[r] & M' \otimes T(\oV')_{n-1}^\rightarrow \otimes N' } \]
\item There is a dual \emph{cohomology theory} for $(\oV,\osigma)$ with coefficients in braided {co}modules. Here one has to work with $T(\oV)_n^\leftarrow$, since a braiding on $(V_1, \ldots, V_r)$ in~$\C$ corresponds to a braiding on the reversed system $(V_r, \ldots, V_1)$ in~$\C^{\op}$.
\item Braided bidifferentials refine to a \emph{precubical structure}, enriched with degeneracies if the braided system carries a ``good'' comultiplication (i.e., compatible with~$\osigma$ and $\osigma$-cocommutative); see~\cite{Lebed1} for details in the braided object case.
\item Braided differentials $({^{\rho}}\! d)_n$ (or $(d^{\lambda})_n$) can be defined with coefficients on one side only, i.e., on $M \otimes T(\oV)_{n}^\rightarrow$ (or $T(\oV)_{n}^\rightarrow \otimes N$).
\end{itemize}
\end{remark}

\begin{notation}\label{not:ossigma}
The obvious morphism from $(V_{i_1} \otimes \cdots \otimes V_{i_s})\otimes (V_{j_1} \otimes \cdots \otimes V_{j_t})$ to $(V_{j_1} \otimes \cdots \otimes V_{j_t})\otimes (V_{i_1} \otimes \cdots \otimes V_{i_s})$, induced by~$\osigma$ and diagrammatically presented as 
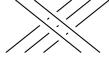
\begin{tikzpicture}[xscale=0.25,yscale=0.2]
\node  at (0.5,3.6){ };
\draw (6,0)--(2.5,3.5);
\draw (5,0)--(1.5,3.5);
\draw (1,0)--(2.8,1.8);
\draw (3.2,2.2)--(3.3,2.3);
\draw (3.7,2.7)--(4.5,3.5);
\draw (2,0)--(3.3,1.3);
\draw (3.7,1.7)--(3.8,1.8);
\draw (4.2,2.2)--(5.5,3.5);
\draw (3,0)--(3.8,0.8);
\draw (4.2,1.2)--(4.3,1.3);
\draw (4.7,1.7)--(6.5,3.5);
\end{tikzpicture}, is denoted by $\ossigma$. (Here we suppose $i_n \le j_m$ for all $n,m$, so that~$\osigma$ is applicable to $V_{i_n} \otimes V_{j_m}$.) 
\end{notation}

\begin{proposition}\label{thm:adjoint_multi_coeffs}
Take a braided system $(\oV,\osigma)\in \BrSyst_r(\C)$ and cut it at some level~$t$,  $1 \le t \le r$. That is, consider the $(1,t)$-subsystem $(\oV,\osigma)[1,t]$. Denote it by $(\oV',\osigma)$. Take also a braided module $(M,\overline{\rho}) \in \ModCat_{(\oV,\osigma)}$.
\begin{enumerate}
\item For any $n$, $M\otimes T(\oV')_n^\rightarrow$ is a $(\oV,\osigma)[t,r]$-module: for $t \le i \le r$, define
\begin{align*}
{^\rho}\!\pi_i=(\rho_i \otimes \Id_{T(\oV')_{n}^\rightarrow}) & (\Id_M \otimes \ossigma_{T(\oV')_n^\rightarrow,V_i}) \, : \,
M\otimes T(\oV')_n^\rightarrow\otimes V_i\rightarrow M\otimes T(\oV')_n^\rightarrow.
\end{align*}
\item The braided differentials ${^{\rho}}\! d$ on $(\oV',\osigma)$ with coefficients in~$M$ are braided module morphisms for the structure above.
\end{enumerate}
\end{proposition}

\begin{proof}
Let us prove the compatibility relation~\eqref{eqn:MultiBrMod} for ${^\rho}\!\pi_i$ and ${^\rho}\!\pi_j$ with $t \le i \le j \le r$. Working on $M\otimes T(\oV')_n^\rightarrow \otimes V_i \otimes V_j$, and using notation~\eqref{eqn:phi_i}, one gets
\begin{align*}
{^\rho}\!\pi_i  ({^\rho}\!\pi_j &\otimes \Id_i)  (\Id_{M\otimes T(\oV')_n^\rightarrow} \otimes \sigma_{i,j}) 
\,=\, \rho_i^1  \ossigma_{T(\oV')_n^\rightarrow,V_i}^2  \rho_j^1  \ossigma_{T(\oV')_n^\rightarrow,V_j}^2  \sigma_{i,j}^{n+2} \\
& \,=\, \rho_i^1  \rho_j^1  \ossigma_{T(\oV')_n^\rightarrow, V_j \otimes V_i}^2  \sigma_{i,j}^{n+2} \stackrel{(A)}{\,=\,}  \rho_i^1  \rho_j^1  \sigma_{i,j}^{2}  \ossigma_{T(\oV')_n^\rightarrow, V_i \otimes V_j}^2  \\
&\stackrel{(B)}{\,=\,} \rho_j^1  \rho_i^1   \ossigma_{T(\oV')_n^\rightarrow, V_i \otimes V_j}^2  \,=\, \rho_j^1  \ossigma_{T(\oV')_n^\rightarrow,V_j}^2  \rho_i^1  \ossigma_{T(\oV')_n^\rightarrow,V_i}^2 \,=\, {^\rho}\!\pi_j  ({^\rho}\!\pi_i \otimes \Id_j),
\end{align*}
where $(A)$ is a repeated application of~\eqref{eqn:YB}, and $(B)$ follows from relation~\eqref{eqn:MultiBrMod} for~$\rho_i$ and~$\rho_j$. The compatibility between ${^\rho}\!\pi_i$ and ${^{\rho}}\! d$, $t \le i \le r$, is verified similarly.
\end{proof} 

Applied to a braided object $(V,\sigma)$ and a braided character on it, Proposition~\ref{thm:adjoint_multi_coeffs} endows all the tensor powers $V^n$ with a braided $(V,\sigma)$-module structure. Inspired by this example, we call \emph{adjoint} the braided modules from the proposition.

\section{A proto-example: braided systems of algebras}\label{sec:SystemsOfUAAs}

The braided systems studied in this section have unital associative algebras (\emph{UAAs}) as components~$V_i$. 
We exhibit a bijection between such systems and \emph{braided tensor products of algebras}, identifying braided modules over the former with usual modules over the latter. Proposition~\ref{thm:BrSystInv} then yields rules for permuting the factors of braided tensor products of algebras. Examples will follow. In this section $\C$ is {monoidal}, not necessarily additive. 

The braidings we use in the associative setting come with additional structure:

\begin{definition}\label{def:Norm1}
\begin{itemize}
\item 
Denote by $\BrSystP_r(\C)$ the category of
\begin{itemize}
\item braided systems $(\oV,\osigma) \in \BrSyst_r(\C)$ enriched with distinguished morphisms $\onu= (\nu_i \colon \II \to V_i)_{1 \le i \le r}$, called \emph{units}, and
\item morphisms from $\BrSyst_r(\C)$ preserving the units.
\end{itemize}
Objects $(\oV,\osigma,\onu)$ of  $\BrSystP_r(\C)$ are called rank~$r$ \emph{pointed braided systems}.
\item A \emph{right module} over $(\oV,\osigma,\onu) \in \BrSystP_r(\C)$ is a right $(\oV,\osigma)$-module $(M,\overline{\rho})$ satisfying $\rho_i  (\Id_M \otimes \nu_i) = \Id_M$ for $1 \le i \le r$ (i.e., units act trivially). The category of such modules and their morphisms is denoted by $\ModCatN_{(\oV,\osigma,\onu)}$.  Similar definitions and notations are assumed for left modules.
\end{itemize}
\end{definition}

We now show that different aspects of the UAA structure for $(V,\mu,\nu)$ are captured by the \emph{associativity braiding}
\[\sigma_{Ass}:=\nu \otimes \mu \quad : \quad V \otimes V = \II \otimes V \otimes V \rightarrow V \otimes V.\]
When working with several UAAs, notation~$\sigma_{Ass}(V)$ or $\sigma_{Ass}(V,\mu,\nu)$ helps avoid confusion. The category of UAAs and algebra morphisms in~$\C$ is denoted by $\Alg(\C)$.

\begin{theorem}[\cite{Lebed1}]\label{thm:UAA}
\begin{enumerate}
\item\label{item:functor} 
One has a fully faithful functor
\begin{align}
\Alg(\C) &\longhookrightarrow \BrSystP_1(\C)\label{eqn:CatInclUAA}\\
(V,\mu,\nu) &\longmapsto (V,\sigma_{Ass},\nu),\notag\\
f &\longmapsto f.\notag
\end{align}

\item\label{item:non-inv} The associativity braiding~$\sigma_{Ass}$ is idempotent: $\sigma_{Ass}  \sigma_{Ass} = \sigma_{Ass}$.

\item\label{item:EqnEquiv} The YBE for~$\sigma_{Ass}$ is equivalent to the associativity for~$\mu$, under the assumption that~$\nu$ is a unit for~$\mu$ (i.e., $\mu  (\Id_V \otimes \nu) = \mu  (\nu \otimes \Id_V) = \Id_V$).

\item\label{item:ModEquiv} For a UAA $(V, \mu, \nu)$ in~$\C$, one has an equivalence of right module categories
\begin{align*}
\ModCat_{(V, \mu, \nu)} &\stackrel{\sim}{\longleftrightarrow} \ModCatN_{(V, \sigma_{Ass}, \nu)}\\
(M,\rho) &\longleftrightarrow (M,\rho),
\end{align*}
where on the left one considers usual modules over UAAs, and on the right the pointed version of braided modules.

\item\label{item:bar} Let~$\C$ be preadditive. For a module $(M,\rho) \in \ModCat_{(V, \mu, \nu)} \simeq \ModCatN_{(V, \sigma_{Ass}, \nu)}$, the left braided differential $^\rho\!d$ on $(M\otimes V^n)_{n \ge 0}$ coincides with the classical \emph{bar differential}  $d_n = \rho^1 + \sum_{i=1}^{n-1}(-1)^i\mu^i$. 
\end{enumerate}
\end{theorem} 

\begin{remark}
\begin{itemize}
\item A more elegant functor $\Alg(\C) \rightarrow \BrSyst_1(\C)$ is obtained by composing~\eqref{eqn:CatInclUAA} with a forgetful functor; however, it is not full.
\item Point~\ref{item:non-inv} shows that the braiding~$\sigma_{Ass}$ is {highly non-invertible} in general. 
\item The equivalence in~\ref{item:EqnEquiv} holds under a mild unitality assumption; such {normalization conditions} are ubiquitous in our braided approach.
\item Point~\ref{item:ModEquiv} for $M=\II$ ensures that an algebra character is a braided character.
\item Dualizing, one recovers the category of \emph{coalgebras} in~$\C$ inside the category of \emph{co-pointed} (= endowed with a distinguished co-element) braided objects:
\begin{align*}
{\coAlg(\C)}&{\longhookrightarrow \BrSystCP_1(\C)},\\
(V,\Delta,\varepsilon) &\longmapsto (V,\sigma_{coAss}=\varepsilon \otimes\Delta,\varepsilon),\\
f &\longmapsto f.
\end{align*}
The algebra-coalgebra duality in a preadditive~$\C$ can now be seen inside the category of bipointed braided objects $\BrSystBP_1(\C)$. Indeed, this category is self-dual, the notion of braiding being so; and it encompasses both $\Alg(\C)$ and $\coAlg(\C)$ (take zero maps as the missing (co)units):
\[\coAlg(\C)\longhookrightarrow \BrSystBP_1(\C)\longhookleftarrow\Alg(\C).\]
\item In the theorem, the associativity braiding can be replaced with its \emph{right version}
$\sigma_{Ass}^r:=\mu \otimes \nu$. In this case left modules should be taken as coefficients in the last point. The diagrams of the two associativity braidings are shown in Fig.~\ref{pic:BrRightUAA}. 
\end{itemize}
\end{remark}
\begin{figure}\centering
\begin{tikzpicture}[scale=0.7]
 \node at (-2,0.5) {$\sigma_{Ass} \quad \longleftrightarrow \qquad$}; 
 \draw (0,0) -- (1,1);
 \draw (1,0) -- (0.5,0.5);
 \draw (0.3,0.7) -- (0,1);
 \node at (0.5,0.5) [right] {$\mu$};
 \node at (0.3,0.7) [left] {$\nu$}; 
 \fill (0.3,0.7) circle (0.1);
 \fill (0.5,0.5) circle (0.1);
 \node at (4,0.5) {};
\end{tikzpicture}
\begin{tikzpicture}[scale=0.7]
 \node at (-2,0.5) {$\sigma_{Ass}^{r} \quad \longleftrightarrow \qquad$}; 
 \draw  (0,0) -- (0.5,0.5);
 \draw  (1,0) -- (0,1);
 \draw  (0.7,0.7) -- (1,1);
 \node at (0.5,0.5) [left]{$\mu$};
 \node at (0.7,0.7) [right] {$\nu$}; 
 \fill (0.7,0.7) circle (0.1);
 \fill (0.5,0.5) circle (0.1);
\end{tikzpicture}
\caption{Associativity braidings: $\sigma_{Ass}$ and its vertical mirror version $\sigma_{Ass}^{r}$}\label{pic:BrRightUAA}
\end{figure}

From now on, we work with several interacting UAAs~$V_i$. After some technical definitions, we study compatibilities between the braidings $\sigma_{Ass}(V_i)$, and interpret them in terms of (a multi-version of) braided tensor products of algebras.

\begin{definition}\label{def:NormNat}
\begin{itemize}
\item  Take a $V \in \Ob(\C)$. A pair of morphisms $(\eta\colon \II \to V,\epsilon \colon V \to \II)$ is called \emph{normalized} if $\epsilon  \eta = \Id_\II$.

\item  Take $V,W \in \Ob(\C)$. A morphism $\xi \colon V \otimes W \to W \otimes V$ is \emph{natural with respect to} a morphism $\varphi \colon V^{n} \to V^{m}$ (or $\psi\colon W^{n} \to W^{m}$) if
\[\xi^1  \cdots  \xi^m  (\varphi \otimes \Id_W) = (\Id_W \otimes \varphi) \xi^1  \cdots  \xi^n\]
(recall Notation~\eqref{eqn:phi_i}), or, respectively,
\[\xi^m  \cdots  \xi^1  (\Id_V \otimes \psi)= (\psi \otimes \Id_V) \xi^n  \cdots  \xi^1.\]
In the case $V = W$ both conditions are required.
\end{itemize}
\end{definition}

The naturality conditions for $n=1$, $m=2$ and $V = W$ are diagrammatically presented in Fig.~\ref{pic:Nat}. In this example, one recovers two of the six Reidemeister moves from the theory of knotted trivalent graphs \cite{KauffmanGraphs,YamadaGraphs,YetterGraphs}.
\begin{figure} \centering
\begin{tikzpicture}[xscale=0.4, yscale = 0.35]
 \draw[rounded corners](0.65,0.65) -- (2,2)-- (2,2.5);
 \draw (0,0) -- (0.35,0.35);
 \draw (1,1) -- (1,2.5);
 \draw[rounded corners](1,0) -- (0,1)-- (0,2.5);
 \fill (1,1) circle (0.2);
 \node at (1,1) [right] {$\varphi$};
 \node at (3,1) {$=$};
\end{tikzpicture}
\begin{tikzpicture}[xscale=0.4, yscale = 0.35]
 \draw[rounded corners](0,-0.5) -- (0,0)--(0.8,0.8);
 \draw (1.15,1.15) --   (2,2);
 \draw[rounded corners](0,-0.5) -- (0,1)--(0.35,1.35);
 \draw (0.65,1.65) --  (1,2);
 \draw[rounded corners] (1,-0.5) -- (1,1) -- (0,2);
 \fill (0.05,0) circle (0.2);
 \node at (0.05,0) [left] {$\varphi$};
\end{tikzpicture}
\begin{tikzpicture}[xscale=0.4, yscale = 0.35]
 \node at (-3,1) {};
 \draw[rounded corners](1.65,0.65) -- (2,1)-- (2,2.5);
 \draw (1,0) -- (1.35,0.35);
 \draw (1,1) -- (1,2.5);
 \draw[rounded corners](2,0) -- (0,2)-- (0,2.5);
 \fill (1,1) circle (0.2);
 \node at (1,1) [left] {$\varphi$};
 \node at (3,1) {$=$};
\end{tikzpicture}
\begin{tikzpicture}[xscale=0.4, yscale = 0.35]
 \draw[rounded corners](0,-0.5) -- (0,0)-- (0.85,0.85);
 \draw (1.15,1.15) -- (1.35,1.35);
 \draw (1.65,1.65) -- (2,2);
 \draw[rounded corners](2,-0.5) -- (2,1)-- (1,2);
 \draw[rounded corners](2,-0.5) -- (2,0)-- (0,2);
 \fill (2,0) circle (0.2);
 \node at (2,0) [right] {$\varphi$};
\end{tikzpicture}
\caption{Naturality}\label{pic:Nat}
\end{figure}

\begin{theorem}\label{thm:ProdOfAlgebrasBrFamily}
In a monoidal category~$\C$, take UAAs $(V_i,\mu_i,\nu_i)_{1 \le i \le r}$, and morphisms $\xi_{i,j}$, $1 \le {i<j} \le r$, natural with respect to~$\nu_i$ and~$\nu_j$. Let each unit~$\nu_i$ be a part of a normalized pair $(\nu_i,\epsilon_i)$. Then the following statements are equivalent:
\begin{enumerate}[label=(\Alph*)]
\item\label{item:braided} The morphisms $\xi_{i,i} := \sigma_{Ass}(V_i)$, $1 \le i \le r$, complete the~$\xi_{i,j}$ and the~$\nu_i$ into a pointed braided system structure on~$\oV$. 
\item\label{item:mixed} Each $\xi_{i,j}$ is natural with respect to~$\mu_i$ and~$\mu_j$, and, for each triple $i<j<k$, the~$\xi$ satisfy the YBE on $V_i \otimes V_j \otimes V_k$.
\item\label{item:alg} A UAA structure on $\leftV:=V_r \otimes V_{r-1} \otimes \cdots \otimes V_1$ can be defined by putting
\begin{align}
\mu_{\leftV} &=(\mu_r \otimes \cdots \otimes \mu_1)  \xi_{1,2}^{2r-2}  (\xi_{2,3}^{2r-4}  \xi_{1,3}^{2r-3})  \cdots  (\xi_{r-1,r}^2  \cdots  \xi_{2,r}^{r-1}  \xi_{1,r}^r),\label{eqn:LeftV}\\
\nu_{\leftV} &= \nu_r \otimes \nu_{r-1} \otimes \cdots \otimes \nu_1. \label{eqn:LeftV'}
\end{align} 
\end{enumerate}
\end{theorem}

The multiplication~\eqref{eqn:LeftV} for $r=3$ is diagrammatically presented in Fig.~\ref{pic:YBiij}{A}. Note the inverse component order in the definition of~$\leftV$, ensuring that~$\mu_{\leftV}$ is well-defined.

\begin{proof}
We show that assertions~\ref{item:braided} and~\ref{item:alg} are both equivalent to~\ref{item:mixed}.

Start with~\ref{item:braided}. The YBE on each $V_i \otimes V_i \otimes V_i$ is guaranteed by Theorem~\ref{thm:UAA}. On $V_i \otimes V_i \otimes V_j$ with $i<j$, the YBE becomes
\begin{align*}
(\xi_{i,j} \otimes \Id_{i})  (\Id_{i} \otimes \xi_{i,j}) (\nu_i \otimes \mu_i \otimes \Id_{j}) &= ( \Id_{j} \otimes \nu_i \otimes \mu_i)  (\xi_{i,j} \otimes \Id_{i})  (\Id_{i} \otimes \xi_{i,j})
\end{align*}
(see Fig.~\ref{pic:YBiij}{B} for a graphical version). But this is equivalent to $\xi_{i,j}$ being natural w.r.t. $\mu_i$ (Fig.~\ref{pic:YBiij}{C}): compose the former with $\Id_{j} \otimes \mu_i$ to get the latter, and compose the latter with $(\xi_{i,j}\otimes \Id_i)  (\nu_i \otimes \Id_j \otimes \Id_i)$ to get the former (in each case, use the naturality of $\xi_{i,j}$ w.r.t. the units to pull the truncated strands out of all crossings). Similarly, the YBE on $V_i \otimes V_j \otimes V_j$, $i<j$, is equivalent to $\xi_{i,j}$ being natural w.r.t. $\mu_j$. This yields the equivalence \ref{item:braided} $\Leftrightarrow$ \ref{item:mixed}.
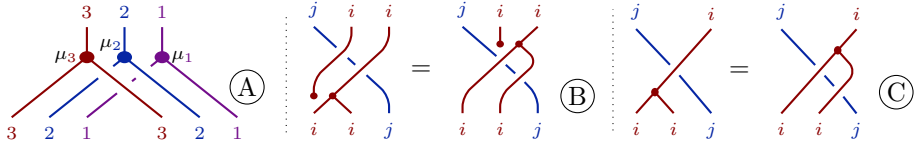
\begin{figure}\centering
\begin{tikzpicture}[xscale=0.5,yscale=0.4]
\draw [thick,colork] (7,0)--(5,2)--(5,3);
\draw [thick,colorj] (6,0)--(4,2)--(4,3);
\draw [thick,colori] (5,0)--(3,2)--(3,3);
\draw [thick,colori] (1,0)--(3,2);
\draw [thick,colorj] (2,0)--(3.3,1.3);
\draw [thick,colorj] (3.7,1.7)--(4,2);
\draw [thick,colork] (3,0)--(3.8,0.8);
\draw [thick,colork] (4.2,1.2)--(4.3,1.3);
\draw [thick,colork] (4.7,1.7)--(5,2);
\fill [colori] (3,2) circle (0.2);
\fill [colorj] (4,2) circle (0.2);
\fill [colork] (5,2) circle (0.2);
\node at (1,0) [colori,below] {$\scriptstyle 3$};
\node at (2,0) [colorj,below] {$\scriptstyle 2$};
\node at (3,0) [colork,below] {$\scriptstyle 1$};
\node at (5,0) [colori,below] {$\scriptstyle 3$};
\node at (6,0) [colorj,below] {$\scriptstyle 2$};
\node at (7,0) [colork,below] {$\scriptstyle 1$};
\node at (3,3) [colori,above] {$\scriptstyle 3$};
\node at (4,3) [colorj,above] {$\scriptstyle 2$};
\node at (5,3) [colork,above] {$\scriptstyle 1$};
\node at (3,2) [left] {$\scriptstyle  \mu_{\color{colori}3}$};
\node at (4.2,2.4) [left] {$\scriptstyle  \mu_{\color{colorj}2}$};
\node at (5,2) [right] {$\scriptstyle  \mu_{\color{colork}1}$};
\node  at (7.2,1)  {\circled{A}};
\draw [dotted] (8.3,-0.5)--(8.3,3.5);
\end{tikzpicture} 
\begin{tikzpicture}[xscale=.25,yscale=0.23]
\node  at (-0.5,1)  {};
 \draw[thick,colori,rounded corners] (0,0) -- (4,4) -- (4,5);
 \draw[thick,colori,rounded corners] (0,1) -- (0,2) -- (2,4) -- (2,5);
 \draw[thick,colori] (2,0) -- (1,1);
 \draw[thick,colorj,rounded corners] (4,0) -- (4,1) -- (2.8,2.2); 
 \draw[thick,colorj,rounded corners] (2.2,2.8) -- (1.8,3.2);
 \draw[thick,colorj,rounded corners] (1.2,3.8) -- (0,5);
 \node at (4,0) [colorj,below] {$\scriptstyle j$};
 \node at (2,0) [colori,below] {$\scriptstyle i$};
 \node at (0,0) [colori,below] {$\scriptstyle i$};
 \node at (0,5) [colorj,above] {$\scriptstyle j$}; 
 \node at (2,5) [colori,above] {$\scriptstyle i$};
 \node at (4,5) [colori,above] {$\scriptstyle i$};
 \fill [colori] (0,1) circle (0.2);
 \fill [colori] (1,1) circle (0.2);
\node  at (5.7,2.5){$=$};
\end{tikzpicture}
\begin{tikzpicture}[xscale=.25,yscale=0.23]
 \draw[thick,colori,rounded corners] (0,0) -- (0,1) -- (4,5);
 \draw[thick,colori,rounded corners] (2,0) -- (2,1) -- (4,3) -- (3,4);
 \draw[thick,colori] (2,4) -- (2,5);
 \draw[thick,colorj,rounded corners] (4,0) -- (4,1) -- (3.3,1.8);
 \draw[thick,colorj,rounded corners] (2.8,2.3) -- (2.3,2.8);
 \draw[thick,colorj,rounded corners] (1.8,3.3) -- (0,5);
 \node at (4,0) [colorj,below] {$\scriptstyle j$};
 \node at (2,0) [colori,below] {$\scriptstyle i$};
 \node at (0,0) [colori,below] {$\scriptstyle i$};
 \node at (0,5) [colorj,above] {$\scriptstyle j$}; 
 \node at (2,5) [colori,above] {$\scriptstyle i$};
 \node at (4,5) [colori,above] {$\scriptstyle i$};
 \fill [colori] (2,4) circle (0.2);
 \fill [colori] (3,4) circle (0.2);
 \node  at (6,1)  {\circled{B}};
 \draw [dotted] (8,-1)--(8,6);
\end{tikzpicture}
\begin{tikzpicture}[xscale=.25,yscale=0.28]
 \draw[thick,colori,rounded corners] (0,0) -- (4,4);
 \draw[thick,colori] (2,0) -- (1,1);
 \draw[thick,colorj,rounded corners] (4,0) -- (2.3,1.8); 
 \draw[thick,colorj,rounded corners] (1.8,2.3) -- (0,4);
 \node at (4,0) [colorj,below] {$\scriptstyle j$};
 \node at (2,0) [colori,below] {$\scriptstyle i$};
 \node at (0,0) [colori,below] {$\scriptstyle i$};
 \node at (0,4) [colorj,above] {$\scriptstyle j$}; 
 \node at (4,4) [colori,above] {$\scriptstyle i$}; 
 \fill [colori] (1,1) circle (0.2);
 \node  at (5.5,2){$=$};
\end{tikzpicture}
\begin{tikzpicture}[xscale=.25,yscale=0.28]
 \draw[thick,colori,rounded corners] (0,1) -- (4,5);
 \draw[thick,colori,rounded corners] (2,1) -- (4,3) -- (3,4);
 \draw[thick,colorj,rounded corners] (4,1) -- (3.3,1.8);
 \draw[thick,colorj,rounded corners] (2.8,2.3) -- (2.3,2.8);
 \draw[thick,colorj,rounded corners] (1.8,3.3) -- (0,5);
 \node at (4,1) [colorj,below] {$\scriptstyle j$};
 \node at (2,1) [colori,below] {$\scriptstyle i$};
 \node at (0,1) [colori,below] {$\scriptstyle i$};
 \node at (0,5) [colorj,above] {$\scriptstyle j$}; 
 \node at (4,5) [colori,above] {$\scriptstyle i$}; 
 \fill [colori] (3,4) circle (0.2);
 \node  at (6,2)  {\circled{C}};
\end{tikzpicture}
\caption{Braided tensor product of UAAs; YBE on $V_i \otimes V_i \otimes V_j$; naturality w.r.t. $\mu_i$}\label{pic:YBiij}
\end{figure}

To conclude, we need the equivalence \ref{item:alg} $\Leftrightarrow$ \ref{item:mixed}. It compares local and global properties of a braided system of UAAs. The following maps relate these two scales:
\begin{equation}\label{eqn:iota}
\iota_j=\nu_r \otimes \cdots \otimes \nu_{j+1} \otimes \Id_{j} \otimes \nu_{j-1} \otimes \cdots \nu_1 \colon V_j \to \leftV.
\end{equation}

Given a collection~$\xi_{i,j}$ from~\ref{item:mixed}, one checks (e.g., graphically) that~$\mu_{\leftV}$ and~$\nu_{\leftV}$ from~\ref{item:alg} define a UAA structure. This generalizes the verifications necessary to define the tensor product of algebras in a braided category. To show that all the conditions from~\ref{item:mixed} are needed, consider the associativity relation for~$\mu_{\leftV}$ composed with
\begin{itemize}
\item either $\iota_i \otimes \iota_j \otimes \iota_k\colon V_i \otimes V_j \otimes V_k \to \leftV^{ 3}$ on the right and the~$\epsilon_t$ at all the positions except for $i,j,k$ on the left (this gives the YBE on $V_i \otimes V_j \otimes V_k $, $i < j < k$);
\item or $\iota_i \otimes \iota_i \otimes \iota_j\colon V_i \otimes V_i \otimes V_j \to \leftV^{ 3}$ on the right and the~$\epsilon_t$ at all the positions except for $i,j$ on the left (this gives the naturality of~$\xi_{i,j}$ w.r.t.~$\mu_i$);
\item or $\iota_i \otimes \iota_j \otimes \iota_j\colon V_i \otimes V_j \otimes V_j \to \leftV^{\otimes 3} $ on the right and the~$\epsilon_t$ at all the positions except for $i,j$ on the left (this gives the naturality of~$\xi_{i,j}$ w.r.t.~$\mu_j$).
\end{itemize}
For example, in the second case the naturality of the~$\xi$ w.r.t. the units yields 
\[(\text{ass-ty for } \mu_{\leftV})(\iota_i \otimes \iota_i \otimes \iota_j) = (\iota_j \otimes \nu_{\leftV} \otimes \iota_i)(\text{nat-ty condition from Fig.~\ref{pic:YBiij}{C}}).\] 
Applying the~$\epsilon_t$, one gets rid of the term $(\iota_j \otimes \mu_{\leftV} \otimes \iota_i)$.
\end{proof}

The theorem gives a braided~\ref{item:braided}, an associative~\ref{item:alg}, and a mixed~\ref{item:mixed} interpretation of the same phenomenon. For certain structures, associativity verification can be considerably simplified by checking~\ref{item:braided} or~\ref{item:mixed} instead. 

\begin{definition}
A braided system of the type described in the theorem is called a \emph{(pointed) braided system of UAAs}, and the UAA~$\leftV$ is called the \emph{braided tensor product} of the UAAs $V_1, \ldots, V_r$, denoted (abusively) by $\leftV = V_r \underset{\xi}{\otimes} V_{r-1} \underset{\xi}{\otimes} \cdots \underset{\xi}{\otimes} V_1$.
\end{definition}

\begin{remark}\label{rmk:EpsilonsNotNecessary}
The~$\epsilon_i$ were used only to prove \ref{item:alg} $\Rightarrow$ \ref{item:mixed}, i.e., to go from the global setting to the local. One could instead impose~\ref{item:alg} for all subsystems of~$\oV$, and work in appropriate subsystems instead of composing with the~$\epsilon_i$. In particular, for $r=2$ the theorem holds true even without the normalized pair condition.
\end{remark}

\begin{remark}\label{rmk:pre_mirror}
Some or all of the morphisms $\xi_{i,i} = \sigma_{Ass}(V_i)$ can be replaced with their right versions $\sigma_{Ass}^r(V_i)$. The theorem still holds true, with analogous proof. 
\end{remark}

\begin{example}\label{ex:TensorProdOfAlgInBraidedCat}
Take UAAs~$V_i$ in a braided category~$\C$, and put $\xi_{i,j}=c_{V_i,V_j}$. The categorical braiding~$c$ is  natural w.r.t. everything, in particular the units. Proposition~\ref{thm:TrivialBrSystem} then translates as condition~\ref{item:braided} from the theorem. The UAA structure on~$\leftV$ deduced from~\ref{item:alg} recovers the usual tensor product of algebras in a braided category. 
\end{example}

In an additive category, the braided tensor product~$\leftV$ is alternatively described as $\raisebox{.5mm}{$T(\oplus_i V_i)$}\big / \raisebox{-.5mm}{$\langle \sigma - \Id, \, \nu - \Id  \rangle$}$, where the ideal we mod out is generated by the images of $\sigma_{i,j} - \Id_i \otimes \Id_j$ and $\nu_i - \Id_i$. Observation~\ref{rmk:MultiMod2} then suggests a representation-theoretic counterpart for the structure equivalence from Theorem~\ref{thm:ProdOfAlgebrasBrFamily}. More generally, 

\begin{proposition}\label{thm:ProdOfAlgebrasBrFamilyMod}
In the settings of Theorem~\ref{thm:ProdOfAlgebrasBrFamily}, the category of modules over the pointed braided system from~\ref{item:braided} is equivalent to the category of modules over the algebra~$\leftV$ from~\ref{item:alg}:
$\ModCatN_{(\oV,\overline{\xi},\onu)} \simeq \ModCat_{(\leftV, \mu_{\leftV},\nu_{\leftV})}$.
\end{proposition}

\begin{proof}
Observation~\ref{rmk:MultiMod} combined with Point~\ref{item:ModEquiv} of Theorem~\ref{thm:UAA} interpret a module structure over $(\oV,\overline{\xi},\onu)$ as module structures $(M,\rho_j)$ over UAAs $(V_j,\mu_j,\nu_j)$, compatible in the sense of~\eqref{eqn:MultiBrMod}. The map $\rho = \rho_1  (\rho_{2} \otimes \Id_{1})  \cdots  (\rho_{r} \otimes \Id_{r-1} \otimes \cdots \otimes \Id_{1})$ then turns~$M$ into a $\leftV$-module. Conversely, a $\leftV$-module $(M,\rho)$ becomes a $(\oV,\overline{\xi},\onu)$-module via $\rho_j = \rho  (\Id_M \otimes \iota_j)$, where the $\iota_j$ are defined in~\eqref{eqn:iota}. The identity functor of~$\C$ and this structure correspondence give the desired category equivalence.
\end{proof}

We now discuss factor permutation in braided tensor products of UAAs.

\begin{proposition}\label{thm:ProdOfAlgebrasBrFamilyInv}
In the settings of Theorem~\ref{thm:ProdOfAlgebrasBrFamily}, suppose one of the~$\xi_{i,i+1}$ invertible. Then
\begin{enumerate}
\item The UAAs $V_1, \ldots, V_{i-1},V_{i+1}, V_i,V_{i+2} \ldots, V_r$ endowed with the~$\xi$ from the system~$\oV$, completed with $\xi_{i,i+1}^{-1}$ on  $V_{i+1} \otimes V_i$, still form a braided system of UAAs. 
\item The braided tensor products~$\leftV$ and $s_i \cdot \leftV:=V_r \underset{\xi}{\otimes} \cdots \underset{\xi}{\otimes} V_{i+2} \underset{\xi}{\otimes} V_{i} \underset{\xi^{-1}}{\otimes} V_{i+1} \underset{\xi}{\otimes} V_{i-1} \underset{\xi}{\otimes} \cdots \underset{\xi}{\otimes} V_1$ are related by the algebra isomorphism (abusively denoted by~$s_i$)
\[\Id_{r} \otimes \ldots \otimes \Id_{{i+2}} \otimes \xi_{i,i+1}^{-1} \otimes \Id_{{i-1}} \otimes \ldots \otimes \Id_{{1}} \quad : \quad \leftV \longrightarrow s_i \cdot \leftV.\]
\item
The algebra isomorphism above induces an equivalence of modules categories: 
\begin{align*}
\ModCat_{\leftV} &\stackrel{\sim}{\longleftrightarrow}  \ModCat_{s_i \cdot \leftV},\\
(M,\rho_{\leftV}) &\longleftrightarrow (M,\rho_{\leftV}  (\Id_M \otimes s_i^{-1})).
\end{align*}
\end{enumerate}
\end{proposition}

\begin{proof}
\begin{enumerate}
\item Proposition~\ref{thm:BrSystInv} allows to swap the components~$V_i$ and~$V_{i+1}$ of the pointed braided system $(\oV,\overline{\xi},\onu)$ from Theorem~\ref{thm:ProdOfAlgebrasBrFamily} \ref{item:braided}. The new system $s_i(\oV,\overline{\xi},\onu)$ still satisfies the conditions from  Theorem~\ref{thm:ProdOfAlgebrasBrFamily} \ref{item:braided}, and is thus a braided system of UAAs (the naturality of~$\xi_{i,i+1}^{-1}$ w.r.t. the units follows from that of~$\xi_{i,i+1}$). 
\item Theorem~\ref{thm:ProdOfAlgebrasBrFamily} \ref{item:alg} then gives a UAA structure on~$s_i(\leftV)$. Applying the YBE several times, one sees that, in order to check that $\Id_{r} \otimes \ldots  \otimes \xi_{i,i+1}^{-1} \otimes \ldots \otimes \Id_{{1}}$ is an algebra morphism, it is sufficient to work with~$V_i$ and~$V_{i+1}$ only. Namely, one has to prove the identity $\xi_{i,i+1}^{-1}  (\nu_{i+1} \otimes \nu_i) = \nu_i \otimes \nu_{i+1}$,  which follows from the naturality of $\xi_{i,i+1}^{-1}$ w.r.t. the units, and from the equality
\begin{align*}
&(\mu_i \otimes \mu_{i+1})  (\Id_i \otimes \xi_{i,i+1}^{-1} \otimes \Id_{i+1})  (\xi_{i,i+1}^{-1} \otimes \xi_{i,i+1}^{-1}) =\\
\xi_{i,i+1}^{-1}  &(\mu_{i+1} \otimes \mu_i)  (\Id_{i+1} \otimes \xi_{i,i+1} \otimes \Id_{i})
\end{align*}
of morphisms $(V_{i+1} \otimes V_{i})^{ 2} \to V_{i} \otimes V_{i+1}$ (Fig.~\ref{pic:XiIsAlgMor}). The latter results from the naturality of $\xi_{i,i+1}$ (and hence $\xi_{i,i+1}^{-1}$ ) w.r.t.~$\mu_i$ and~$\mu_{i+1}$ (Theorem \ref{thm:ProdOfAlgebrasBrFamily} \ref{item:mixed}).
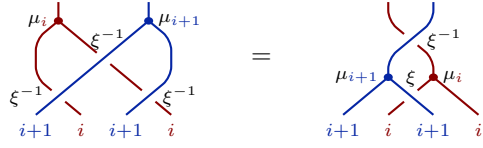
\begin{figure}\centering
\begin{tikzpicture}[xscale = 0.3, yscale=0.25]
 \draw[thick, colori,rounded corners] (0.7,1.3) -- (0,2) -- (0,4) -- (1,5);
 \draw[thick, colori] (2,0) -- (1.3,0.7);
 \draw[thick, colori] (6,0) -- (5.3,0.7);
 \draw[thick, colori] (4.7,1.3) -- (3.3,2.7);
 \draw[thick, colori] (2.7,3.3) -- (1,5) -- (1,6);
 \draw[thick, colorj,rounded corners] (4,0) -- (6,2) -- (6,4) -- (5,5); 
 \draw[thick, colorj] (0,0) -- (5,5) -- (5,6);
 \node at (6,0) [colori,below] {$\scriptstyle i$};
 \node at (4,0) [colorj,below] {$\scriptstyle {i+1}$};
 \node at (2,0) [colori,below] {$\scriptstyle i$};
 \node at (0,0) [colorj,below] {$\scriptstyle {i+1}$};
 \fill[colori] (1,5) circle (0.2);
 \fill[colorj] (5,5) circle (0.2);
 \node at (5,5) [right] {$\scriptstyle \mu_{\color{colorj}i+1}$};
 \node at (1,5) [left] {$\scriptstyle \mu_{\color{colori}i}$};
 \node at (3.2,3) [above] {$\scriptstyle \xi^{-1}$};
 \node at (5.2,1) [right] {$\scriptstyle \xi^{-1}$};
 \node at (0.8,1) [left] {$\scriptstyle \xi^{-1}$};
 \node  at (10,3){$=$};
\end{tikzpicture}
\begin{tikzpicture}[xscale = 0.3, yscale=0.25]
\node  at (-2,1.5){};
 \draw[thick, colori,rounded corners] (2,6) -- (2,5) -- (2.7,4.3);
 \draw[thick, colori,rounded corners] (4,2) -- (4,3) -- (3.3,3.7);
 \draw[thick, colori] (2.7,0.7) -- (2,0);
 \draw[thick, colori] (3.3,1.3) -- (4,2) -- (6,0);
 \draw[thick, colorj,rounded corners] (2,2) -- (2,3) -- (4,5) -- (4,6);
 \draw[thick, colorj] (0,0) -- (2,2) -- (4,0);
 \node at (6,0) [colori,below] {$\scriptstyle i$};
 \node at (4,0) [colorj,below] {$\scriptstyle {i+1}$};
 \node at (2,0) [colori,below] {$\scriptstyle i$};
 \node at (0,0) [colorj,below] {$\scriptstyle {i+1}$};
 \node at (2,2) [left] {$\scriptstyle \mu_{\color{colorj}i+1}$};
 \node at (4,2) [right] {$\scriptstyle \mu_{\color{colori}i}$};
 \node at (3,1) [above] {$\scriptstyle \xi$};
 \node at (3.3,4) [right] {$\scriptstyle \xi^{-1}$};
 \fill[colorj] (2,2) circle (0.2);
 \fill[colori] (4,2) circle (0.2);
\end{tikzpicture}
\caption{Checking that $\xi_{i,i+1}^{-1}$ is an algebra morphism}\label{pic:XiIsAlgMor}
\end{figure}
\item (The proofs of) Propositions~\ref{thm:BrSystInv} and~\ref{thm:ProdOfAlgebrasBrFamilyMod} yield the category equivalences
\begin{center}
\begin{tabular}{rcl}
$\ModCat_{\leftV}$& $\simeq \quad \ModCatN_{(\oV,\overline{\xi},\onu)} \quad \simeq \quad \ModCatN_{s_i(\oV,\overline{\xi},\onu)}$ & $\simeq\quad \ModCat_{s_i(\leftV)}$\\ 
$(M,\rho_{\leftV} )$ &$\longlongleftrightarrow $& $(M,\rho_{\leftV}  (\Id_M \otimes s_i^{-1}))$ 
\end{tabular}\qedhere
\end{center}
\end{enumerate}
\end{proof}

\begin{remark}\label{thm:UAABrSystInvSn}
As in Remark~\ref{thm:BrSystInvSn}, one gets \emph{partial $S_r$-actions} on rank~$r$ {braided systems} and {braided tensor products} of UAAs. Concretely, a permutation $\theta \in S_r$ with a minimal decomposition $\theta = s_{i_1}\cdots s_{i_k}$ sends $(\oV,\overline{\xi},\onu)$ to $s_{i_1}(\cdots(s_{i_k}(\oV,\overline{\xi},\onu))\cdots)$, and acts on UAA braided tensor products by the algebra morphism $s_{i_1}  \cdots  s_{i_k}$ (still denoted by~$\theta$), provided that the braiding components are invertible when necessary. These actions are mutually compatible, and induce module category equivalences via $(M,\rho_{\leftV}) \leftrightarrow (M,\rho_{\leftV}  (\Id_M \otimes \theta^{-1}))$.
\end{remark}

As a first illustration of the braided system theory, we now upgrade Theorem~\ref{thm:UAA} to the rank~$2$ level. A {braided} category $(\C,\otimes,\II,c)$ is needed here.

For $(V,\mu,\nu) \in \Alg(\C)$, the data $(\mu c_{V,V},\nu)$ define another, \emph{twisted} UAA structure on~$V$, denoted by~$V^{op}$. The associativity braiding becomes here $\sigma_{Ass}(V^{op})= \nu \otimes (\mu c_{V,V})$. This twisting is used to relate left and right modules:

\begin{lemma}\label{thm:lMod=rMod}
For $(V,\mu,\nu) \in \Alg(\C)$, the functors 
\begin{align}
\ModCat_{V^{op}} &\overset{\sim}{\longleftrightarrow} {_{V}}\!\ModCat, \notag\\
(M,\rho) & \longmapsto (M, \llambda(\rho) := \rho  c_{M,V}^{-1}),\label{eqn:rtol}\\
(M,\rrho(\lambda):= \lambda  c_{M,V}) &\longmapsfrom (M,\lambda),\label{eqn:ltor}
\end{align}
extended by identities on morphisms, yield a category equivalence.
\end{lemma}

Take now two UAAs $(V,\mu,\nu)$ and $(V',\mu',\nu')$. Returning to Example~\ref{ex:TensorProdOfAlgInBraidedCat}, one gets

\begin{lemma} The data $(V_1=V, V_2=V';\, \sigma_{1,1} = \sigma_{Ass}(V),\sigma_{2,2} = \sigma_{Ass}(V'^{op}), \sigma_{1,2} =$ $c_{V,V'})$ define a braided system of UAAs, denoted by $\Bimod(V,V')$. 
\end{lemma}

The proofs of the above lemmas are straightforward.

The module category equivalence from Proposition~\ref{thm:ProdOfAlgebrasBrFamilyMod} and permutation rules from Proposition~\ref{thm:ProdOfAlgebrasBrFamilyInv} apply to~$\Bimod(V,V')$. Using Observation~\ref{rmk:MultiMod} and Lemma~\ref{thm:lMod=rMod}, one interprets braided modules over this system as familiar \emph{algebra bimodules}:

\begin{proposition}\label{thm:AlgBimodAsMultiMod}
Take UAAs $(V,\mu,\nu)$ and $(V',\mu',\nu')$ in a braided category~$\C$. Let $_{V'}\!\ModCat_V$ be the category of $(V',V)$-bimodules. The following categories are equivalent:
\[\ModCat_{V'^{op}\underset{c}{\otimes} V} \simeq \ModCatN_{\Bimod(V,V')}\simeq _{V'}\!\ModCat_V \simeq \ModCatN_{s_2(\Bimod(V,V'))} \simeq \ModCat_{V \underset{c^{-1}}{\otimes} V'^{op}}.\]
\end{proposition}

Note that $V^{op}\underset{c}{\otimes} V$ is the \emph{enveloping algebra} of the algebra~$V$.

We then apply adjoint module theory to our bimodules. Recall Notation~\eqref{eqn:phi_i}.

\begin{proposition}\label{thm:BarCxIsBimod}
Take a bimodule $(M, \: \rho\colon M \otimes V \to M,\: \lambda \colon V' \otimes M \to M)$ over UAAs~$V$ and~$V'$ in a braided category~$\C$. The bar complex $(M\otimes T(V), d_{bar})$ for~$V$ with coefficients in~$M$ is a complex in $_{V'}\!\ModCat_V$. In other words, the differentials $(d_{bar})_n$ are bimodule morphisms, where a bimodule structure on $M \otimes V^{n}$ is given by 
\begin{align*}
\rho_{bar} &= \mu^{n+1} \colon M \otimes V^{ n} \otimes V \to M \otimes V^{n},&
\lambda_{bar} &= \lambda^1 \colon V' \otimes M \otimes V^{ n} \to M \otimes V^{ n}.
\end{align*}
\end{proposition}

\begin{proof}
Plug into Proposition~\ref{thm:adjoint_multi_coeffs} the system $\Bimod(V,V')$, the bimodule $(M,\rho, \lambda)$ (interpreted as a $\Bimod(V,V')$-module via Proposition~\ref{thm:AlgBimodAsMultiMod}), and $t=2$. One obtains the compatibility of the braided differential ${^{\rho}}\! d = d_{bar}$ (cf. Theorem~\ref{thm:UAA}, Point~\ref{item:bar}) with the braided $\Bimod(V,V')$-module structures on the $M\otimes V^n$. Using Proposition~\ref{thm:AlgBimodAsMultiMod} again, one interprets these braided modules as $(V',V)$-bimodules, with the explicit structure from Lemma~\ref{thm:lMod=rMod}:
\begin{align*}
{^\rho}\!\pi_1 &= \rho_1  (\Id_M \otimes \ossigma_{V^{ n},V}) = \mu^{n+1},\\
\lambda ({^{\rrho(\lambda)}}\!\pi_2) &={^{\rrho(\lambda)}}\!\pi_2  c_{M\otimes V^{ n} ,V'}^{-1} 
= (\lambda  c_{M,V'})^1  (\Id_M \otimes \ossigma_{V^{ n},V'})  c_{M\otimes V^{ n},V'}^{-1}\\
&= (\lambda  c_{M,V'})^1  (\Id_M \otimes c_{V^{ n},V'})  c_{M\otimes V^{ n},V'}^{-1}
= \lambda^1.\qedhere
\end{align*}
\end{proof}

This bimodule structure on the bar complex is fundamental for interpreting the \emph{Hochschild (co)homology} via the differential induced on coinvariants by~$d_{bar}$.

\section{A braided interpretation of crossed products}\label{sec:crossed}

We now present a rank~$3$ braided system. It recovers Panaite's braided treatment of \emph{two-sided crossed products}~\cite{Panaite2}, and its extension~\cite{IteratedTwisted} to the \emph{generalized two-sided crossed products} $A\gsm C\gtl B$ of Bulacu--Panaite--Van Oystaeyen \cite{PanaiteGen}. Our component permutation technique yields $6$ isomorphic versions of the algebra $A\gsm C\gtl B$. This extends algebra isomorphisms from~\cite{Panaite2,IteratedTwisted}, and simplifies their originally very technical proof. Further, our adjoint module machinery yields a $(B,A)$-bimodule structure on~$C^{n}$, used for constructing a bialgebra homology theory in Section~\ref{sec:bialg}.

First, we need categorical versions of some basic algebraic notions.

\begin{definition}\label{def:cat_bialg}
\begin{itemize}
\item 
A \emph{bialgebra} in a {braided} category $(\C,\otimes,\II,c)$ is an object~$H$ endowed with a UAA structure $(\mu,\nu)$ and a counital coassociative coalgebra (= \emph{coUAA}) structure $(\Delta,\varepsilon)$,  compatible in the following sense:
\begin{align}
\Delta  \mu &= (\mu \otimes \mu) c^2  (\Delta \otimes \Delta), &
\Delta  \nu &= \nu \otimes \nu,&
\varepsilon  \mu &= \varepsilon \otimes \varepsilon, & \varepsilon  \nu &= \Id_{\II}.\label{eqn:cat_bialg}
\end{align}
It is a \emph{Hopf algebra} if it carries an \emph{antipode}, i.e., an endomorphism~$s$ satisfying 
\begin{equation}\label{eqn:s}
\mu  (s\otimes \Id_H)  \Delta = \mu  (\Id_H \otimes s)  \Delta  = \nu  \varepsilon.
\end{equation}

\item
For a bialgebra~$H$ in~$\C$, a left $H$-\emph{module algebra} is a UAA $(M,\mu_M,\nu_M)$ endowed with a left $H$-module structure $\lambda\colon H \otimes M \to M$, such that~$\mu_M$ and~$\nu_M$ are $H$-module morphisms (Fig.~\ref{pic:Bialg}):
\begin{align}
\lambda  (\Id_H \otimes\mu_M) &= \mu_M  (\lambda \otimes \lambda)  c^2  (\Delta \otimes \Id_M^{\otimes 2}),
&\lambda  (\Id_H \otimes\nu_M) &= \nu_M  \varepsilon.\label{eqn:ModAlg}
\end{align}
Right $H$-module algebras and $H$-(bi)(co)module algebras are defined similarly.

\item The categories of bialgebras / Hopf algebras / $H$-(co)module algebras and their morphisms in~$\C$ are denoted by, respectively, $\Bialg(\C)$, $\HAlg(\C)$, ${}_H\!\ModAlg$, $\ModAlg_H$, ${}^H\!\ModAlg$, etc.
\end{itemize}
\end{definition}

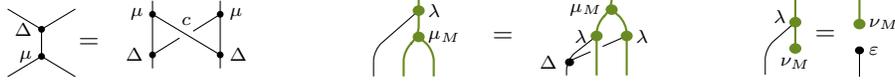
\begin{figure}\centering
\begin{tikzpicture}[xscale=0.9,yscale=0.9]
 \draw (0,0) -- (0.5,0.3) -- (0.5,0.7) -- (0,1);
 \draw (1,0) -- (0.5,0.3) -- (0.5,0.7) -- (1,1);
 \node at (0.5,0.3) [left] {$\scriptstyle \mu$};
 \node at (0.5,0.7) [left] {$\scriptstyle \Delta$};
 \fill  (0.5,0.3) circle (0.05);
 \fill  (0.5,0.7) circle (0.05);
 \node at (1.2,0.5) {$=$};
\end{tikzpicture}
\begin{tikzpicture}[xscale=0.9,yscale=0.9] 
 \draw (3,0) -- (3,1);
 \draw (4,1) -- (4,0);
 \draw (3,0.2) -- (3.4,0.45);
 \draw (3.6,0.55) -- (4,0.8);
 \draw (3,0.8) -- (4,0.2);
 \fill (3,0.2) circle (0.05);
 \fill (3,0.8) circle (0.05);
 \fill (4,0.8) circle (0.05);
 \fill (4,0.2) circle (0.05);
 \node at (3,0.8) [left] {$\scriptstyle \mu$};
 \node at (3,0.2) [left] {$\scriptstyle \Delta$};
 \node at (4,0.8) [right] {$\scriptstyle \mu$};
 \node at (4,0.2) [right] {$\scriptstyle \Delta$};
 \node at (3.5,0.5) [above] {$\scriptstyle c$};
 \node at (6,0) {};
\end{tikzpicture}
\begin{tikzpicture}[xscale=0.4,yscale=0.35]
 \draw [rounded corners] (0,0) -- (0,1) -- (1.5,2.5);
 \draw [colorM, thick, rounded corners] (1,0) -- (1,1) -- (1.5,1.5) -- (1.5,3);
 \draw [colorM, thick, rounded corners] (2,0) -- (2,1) -- (1.5,1.5);
 \node at (1.5,1.5) [right] {$\scriptstyle \mu_M$};
 \node at (1.5,2.5) [right] {$\scriptstyle \lambda$};
 \fill [colorM] (1.5,1.5) circle (0.2);
 \fill [colorM] (1.5,2.5) circle (0.2);
 \node at (4.3,1.5) {$=$};
\end{tikzpicture}
\begin{tikzpicture}[xscale=0.4,yscale=0.35]
 \draw [rounded corners] (0,0) -- (0,0.6) -- (1,1.5);
 \draw [rounded corners] (0,0.5) -- (0.8,0.9);
 \draw [rounded corners] (1.2,1.1) -- (2,1.5);
 \draw [colorM, thick, rounded corners] (1,0) -- (1,2) -- (1.5,2.5) -- (1.5,3);
 \draw [colorM, thick, rounded corners] (2,0) -- (2,2) -- (1.5,2.5);
 \node at (1.5,2.5) [left] {$\scriptstyle \mu_M$};
 \node at (0,0.5) [left] {$\scriptstyle \Delta$};
 \node at (1,1.5) [left] {$\scriptstyle \lambda$};
 \node at (2,1.5) [right] {$\scriptstyle \lambda$};
 \fill [colorM] (1.5,2.5) circle (0.2);
 \fill (0.1,0.5) circle (0.15);
 \fill [colorM] (1,1.5) circle (0.2);
 \fill [colorM] (2,1.5) circle (0.2);
 \node at (6,1.5) {};
\end{tikzpicture}
\begin{tikzpicture}[xscale=0.4,yscale=0.35]
 \draw [rounded corners] (0,0) -- (0,1) -- (1,2);
 \draw [colorM, thick, rounded corners] (1,1) -- (1,3);
 \node at (1,1) [below] {$\scriptstyle \nu_M$};
 \node at (1,2.2) [left] {$\scriptstyle \lambda$};
 \fill [colorM] (1,1) circle (0.2);
 \fill [colorM] (1,2) circle (0.2);
 \node at (2,1.5) {$=$};
\end{tikzpicture}
\begin{tikzpicture}[xscale=0.4,yscale=0.35]
 \draw [rounded corners] (0,0) -- (0,1);
 \draw [colorM, thick, rounded corners] (0,2) -- (0,3);
 \node at (0,2) [right] {$\scriptstyle \nu_M$};
 \node at (0,1) [right] {$\scriptstyle \varepsilon$};
 \fill [colorM] (0,2) circle (0.2);
 \fill (0,1) circle (0.15);
\end{tikzpicture}
\caption{Main bialgebra and module algebra axioms}\label{pic:Bialg}
\end{figure}

\begin{proposition}\label{thm:smash}
Take a bialgebra~$H$, a left $H$-module algebra $(A,\lambda)$, a right $H$-module algebra $(B,\rho)$, and an $H$-bicomodule algebra $(C,\,\delta_l\colon C \to H \otimes C,\,\delta_r\colon C \to C \otimes H)$ in a {symmetric} category $(\C,\otimes,\II,c)$. Then 
\begin{enumerate}
\item The UAAs $(B,C,A)$ form a braided system of UAAs, with
\begin{align*}
\xi_{1,2} &= (\Id_C \otimes \rho)  (c_{B,C} \otimes \Id_H)  (\Id_B \otimes \delta_r),&\xi_{1,3} &= c_{B,A},\\
\xi_{2,3} &= (\lambda \otimes \Id_C)  (\Id_H \otimes c_{C,A})  (\delta_l \otimes \Id_A ).&&
\end{align*}
\item Formulas \eqref{eqn:LeftV}-\eqref{eqn:LeftV'} for the $\xi_{i,j}$ above define a UAA structure on $A \otimes C \otimes B$. 
\item One has a module category equivalence 
$\ModCatN_{(B,C,A; \overline{\xi})} \simeq \ModCat_{A {\otimes_{\xi}} C {\otimes_{\xi}} B}$.
\end{enumerate}
\end{proposition}

The braiding from the proposition is shown in Fig.~\ref{pic:CrossedProd}. Here and below the underlying braiding of a symmetric category is depicted by a solid crossing.

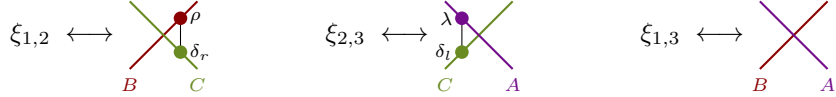
\begin{figure}\centering
\begin{tikzpicture}[scale=0.9]
 \node at (-1,0.5) {$\xi_{1,2} \;\longleftrightarrow $};
 \draw [thick, colori] (0,0) -- (1,1);
 \draw [thick, colorM] (0,1) -- (1,0);
 \draw  (0.75,0.25) -- (0.75,0.75);
 \node at (0.75,0.25) [right] {$\scriptstyle \delta_r$};
 \node at (0.75,0.75) [right] {$\scriptstyle \rho$};
 \fill [colorM] (0.75,0.25) circle (0.1);
 \fill [colori] (0.75,0.75) circle (0.1);
 \node at (0,0) [below,colori] {$\scriptstyle B$};
 \node at (1,0) [below,colorM] {$\scriptstyle C$};
 \node at (2.5,0.5) {};
\end{tikzpicture}
\begin{tikzpicture}[scale=0.9]
 \node at (-1,0.5) {$\xi_{2,3} \;\longleftrightarrow$};
 \draw [thick, colorM] (0,0) -- (1,1);
 \draw [thick, colork] (0,1) -- (1,0);
 \draw  (0.25,0.25) -- (0.25,0.75);
 \node at (0.25,0.25) [left] {$\scriptstyle \delta_l$};
 \node at (0.25,0.75) [left] {$\scriptstyle \lambda$};
 \fill [colorM] (0.25,0.25) circle (0.1);
 \fill [colork] (0.25,0.75) circle (0.1);
 \node at (1,0) [below,colork] {$\scriptstyle A$};
 \node at (0,0) [below,colorM] {$\scriptstyle C$};
 \node at (2.5,0.5) {};
\end{tikzpicture}
\begin{tikzpicture}[scale=0.9]
 \node at (-1,0.5) {$\xi_{1,3} \;\longleftrightarrow $};
 \draw [thick, colori] (0,0) -- (1,1);
 \draw [thick, colork] (0,1) -- (1,0);;
 \node at (0,0) [below,colori] {$\scriptstyle B$};
 \node at (1,0) [below,colork] {$\scriptstyle A$};
\end{tikzpicture}
\caption{A braided system for a two-sided crossed product}\label{pic:CrossedProd}
\end{figure}

\begin{proof}
The key point is to verify the conditions of Theorem~\ref{thm:ProdOfAlgebrasBrFamily} \ref{item:mixed} for the~$\xi$: 
\begin{itemize}
\item The YBE on $B \otimes C \otimes A$ rewrites (using the naturality of~$c$) as 
\begin{align*}
& (\llambda(\rho) \otimes \Id_M \otimes \rrho(\lambda)) (\Id_A \otimes \delta_{l,r} \otimes \Id_B) p =\\ &(\llambda(\rho) \otimes \Id_M \otimes \rrho(\lambda)) (\Id_A \otimes \delta_{r,l} \otimes \Id_B) p.
\end{align*}
Here $p = (c_{C,A} \otimes \Id_B)(\Id_C \otimes c_{B,A})(c_{B,C} \otimes \Id_A)$; $\llambda$ and $\rrho$ are defined by \eqref{eqn:rtol}-\eqref{eqn:ltor}; $\delta_{l,r} = (\delta_l \otimes \Id_H) \delta_r$ and $\delta_{r,l} = (\Id_H \otimes \delta_r) \delta_l$ are morphisms $C \to H \otimes C \otimes H$. Now, the left-right $H$-coaction compatibility for~$C$ yields $\delta_{l,r} = \delta_{r,l}$. 
\item The naturality of the~$\xi$ w.r.t. the~$\mu$ is a consequence of the defining properties of $H$-(bico)module algebras. Here we show that~$\xi_{1,2}$ is natural w.r.t.~$\mu_B$, the other cases being analogous:
\begin{align*}
\xi_{1,2}&  (\mu_B \otimes \Id_C)\\
 &\stackrel{1}{=}(\Id_C \otimes \rho)  (c_{B,C} \otimes \Id_H)  (\mu_B \otimes \delta_r) \\& \stackrel{2}{=} (\Id_C \otimes \rho) (\Id_C \otimes \mu_B \otimes \Id_H)  (c_{B \otimes B,C} \otimes \Id_H)  (\Id_B^{\otimes 2} \otimes \delta_r) 
\\&\stackrel{3}{=} (\Id_C \otimes \mu_B)  (\Id_C \otimes \rho^{\otimes 2}) (\Id_{C\otimes B} \otimes c_{B,H} \otimes \Id_H)  
(c_{B \otimes B,C} \otimes \Delta_H)  (\Id_B^{\otimes 2} \otimes \delta_r)
\\& \stackrel{4}{=} (\Id_C \otimes \mu_B)  (\Id_C \otimes \rho \otimes \Id_B)  (c_{B,C} \otimes \Id_{H\otimes B}) (\Id_B \otimes \delta_r \otimes \rho) \\ 
& \qquad\qquad (\Id_B \otimes c_{B,C} \otimes \Id_H)  (\Id_B^{\otimes 2} \otimes \delta_r) 
\\&\stackrel{5}{=} (\Id_C \otimes \mu_B)  (\xi_{1,2}  \otimes \Id_B)  (\Id_B \otimes \xi_{1,2}).
\end{align*}
We used the definition of~$\xi_{1,2}$ (steps 1 and 5), the naturality of~$c$ (2), the definition of right $H$-module algebra for~$B$ (3) and that of right $H$-comodule for~$C$ (4). The easiest way to follow this proof is to draw diagrams!
\item Similarly, the naturality of the~$\xi$ w.r.t. the units follows from the naturality of~$c$ and from the definition of $H$-(co)module algebras.
\end{itemize}

Theorem~\ref{thm:ProdOfAlgebrasBrFamily} \ref{item:braided} then confirms that the~$\xi$ together with the~$\sigma_{Ass}$ form a braiding, while Point~\ref{item:alg} asserts that $A {\otimes_{\xi}} C {\otimes_{\xi}} B$ is an UAA. Finally, Proposition~\ref{thm:ProdOfAlgebrasBrFamilyMod} gives the required module category equivalence.
\end{proof}

Our proposition recovers the \emph{generalized two-sided crossed product}  $A\gsm C\gtl B:= A {\otimes_{\xi}} C {\otimes_{\xi}} B$ from~\cite{PanaiteGen}. The choice $C=H$ (with $\Delta_H$ as coactions) yields the \emph{two-sided crossed product} $A\#H\#B:= A {\otimes_{\xi}} H {\otimes_{\xi}} B$ of Hausser--Nill \cite{HN}. We thus replace the original technical associativity and module-category-equivalence verifications for $A\gsm C\gtl B$ with a more conceptual proof. 

Further, forgetting the~$B$ (or~$A$) part of the structure and taking as~$C$ a left (respectively, right) $H$-comodule, one obtains rank~$2$ braided systems. This gives a braided treatment of (a generalized version of) left / right {\emph{crossed (or smash) products}} $A\#H:= A {\otimes_{\xi}} H$ and $H\#B:= H {\otimes_{\xi}} B$.

If~$H$ has an invertible antipode~$s$, then all the~$\xi$ are invertible:
\begin{align*}
\xi_{1,2}^{-1} &= ((\rho  c_{H,B})\otimes \Id_C)  (s^{-1} \otimes c_{C,B}) ((c_{C,H}  \delta_r) \otimes \Id_B ), &\xi_{1,3}^{-1} &= c_{A,B},\\
\xi_{2,3}^{-1} &= (\Id_C \otimes (\lambda  c_{A,H}))  (c_{A,C} \otimes s^{-1})  (\Id_A \otimes (c_{H,C}  \delta_l)).&&
\end{align*}
Proposition~\ref{thm:ProdOfAlgebrasBrFamilyInv} then allows to permute the components of $A {\otimes_{\xi}} C {\otimes_{\xi}} B$, producing six pairwise isomorphic UAAs with pairwise equivalent module categories. In particular, one recovers the algebra isomorphism $A\#H\#B \simeq (A\otimes B) \bowtie H$ from~\cite{HN}.

Next, after a preliminary general lemma, we apply our adjoint braided module theory to the braided system from Proposition~\ref{thm:smash}, with trivial coefficients $M=\II$. 

\begin{lemma}\label{thm:adjointBiModules}
Take a rank~$r$ braided system $(\oV,\osigma)$ in a {symmetric} category $(\C,\otimes,\II,c)$, with $\sigma_{1,r} = c_{V_1,V_r}$. For this system, take two braided characters~$\overline{\epsilon}$ and~$\overline{\zeta}$. Then the right $(V_r,\sigma_{r,r})$-module structure ${^\epsilon}\!\pi_r$ and the left $(V_1,\sigma_{1,1})$-module structure  $\pi_1^\zeta$ on $T(\oV)_n^\rightarrow$ commute:
${^\epsilon}\!\pi_r  (\pi_1^\zeta \otimes \Id_{r}) = \pi_1^\zeta   (\Id_{1} \otimes {^\epsilon}\!\pi_r) \, \colon V_1 \otimes T(\oV)_n^ \to \otimes V_r\rightarrow T(\oV)_n^\rightarrow$.
\end{lemma}

\begin{proof}
The categorical braiding~$c$ is natural w.r.t. the components~$\epsilon_r$ and~$\zeta_1$ of our braided characters. This allows to rewrite the desired identity as
\begin{align*}
& (\epsilon_r \otimes \Id_{T(\oV)_n^\rightarrow} \otimes \zeta_1)  (\ossigma_{T(\oV)_n^\rightarrow,V_r} \otimes \Id_1)  \ossigma_{V_1,T(\oV)_n^\rightarrow \otimes V_r} =\\
& (\epsilon_r \otimes \Id_{T(\oV)_n^\rightarrow} \otimes \zeta_1)  (\Id_r \otimes \ossigma_{V_1,T(\oV)_n^\rightarrow}) \ossigma_{V_1 \otimes T(\oV)_n^\rightarrow,V_r},
\end{align*}
which is checked by a repeated application of the YBE.
\end{proof}

We now return to two-sided crossed products. Recall the notation~$\varphi^i$ from~\eqref{eqn:phi_i}. Put
\begin{equation}\label{eqn:omega}
\omega_{2n} = \bigl( \begin{smallmatrix}
 1 & 2 & \ldots & n &  n+1 & n+2 & \ldots & 2n \\
 1 & 3 & \ldots & 2n-1 & 2 & 4 & \ldots & 2n \\  
\end{smallmatrix} \bigr)
\in S_{2n}.
\end{equation}

\begin{proposition}\label{thm:SmashadjointModules}
In the settings of Proposition~\ref{thm:smash}, choose algebra characters~$\epsilon_A$ and~$\epsilon_B$ for~$A$ and~$B$. The morphisms below turn the tensor powers~$C^{n}$ into bimodules:
\begin{align*}
{^{\epsilon_A}}\!\pi &=  (\epsilon_A)^1  \lambda^1 (\Id_H \otimes c_{C^{ n},A})  (\mu^1)^{(n-1)}  ((\omega_{2n}^{-1}  \delta_l^{\otimes n}) \otimes \Id_A) \quad : \quad C^{ n} \otimes A\rightarrow  C^{ n} ,\\
\pi^{\epsilon_B} &=  (\epsilon_B)^{n+1}  \rho^{n+1} (c_{B,C^{ n}} \otimes \Id_H)  (\mu^{n+2})^{(n-1)}  (\Id_B \otimes (\omega_{2n}^{-1}  \delta_r^{\otimes n})) \quad : \quad B \otimes C^{ n} \rightarrow  C^{ n}
\end{align*}
(Fig.~\ref{pic:SmashBimodStructure}), where~$S_{2n}$ acts on $C^{2n}$ via the symmetric braiding~$c$, and the notation $(\mu^i)^{(k)}$ stands for the map~$\mu^i$ iterated $k$~times.
\end{proposition} 

\begin{figure}\centering
\begin{tikzpicture}[xscale=0.65,yscale=0.6]
 \node at (-3.5,1.5)  {${^{\epsilon_A}}\!\pi \, \longleftrightarrow$};
 \draw[colorM, thick] (0,0) --  (0,3.5);
 \draw[colorM, thick] (1,0) -- (1,3.5);
 \draw[colorM, thick] (2,0) -- (2,3.5);
 \draw[colork, thick, rounded corners] (3,0) -- (3,0.5) -- (-1,2.5) -- (-1,3);
 \draw (0,0.5) -- (-1,2) -- (-1,2.5);
 \draw (1,0.5) -- (-1,2);
 \draw (2,0.5) -- (-1,2);
 \node at (0,0.5) [right] {$\scriptstyle \delta_l$};
 \node at (1,0.5) [right] {$\scriptstyle \delta_l$};
 \node at (2,0.5) [right] {$\scriptstyle \delta_l$};
 \node at (-1,2) [left] {$\scriptstyle \mu$};
 \node at (-1,2.7) [left] {$\scriptstyle \lambda$};
 \node at (-1,3) [colork,above] {$\scriptstyle \epsilon_A$};
 \fill [colorM] (0,0.5) circle (0.12);
 \fill [colorM] (1,0.5) circle (0.12);
 \fill [colorM] (2,0.5) circle (0.12);
 \fill (-1,2) circle (0.1);
 \fill [colork] (-1,2.5) circle (0.12);
 \fill [colork] (-1,3) circle (0.12);
 \node at (3,0) [colork,below] {$\scriptstyle A$};
 \node at (0,0) [colorM,below] {$\scriptstyle C$};
 \node at (1,0) [colorM,below] {$\scriptstyle C$}; 
 \node at (2,0) [colorM,below] {$\scriptstyle C$};
\end{tikzpicture}
\begin{tikzpicture}[xscale=0.65,yscale=0.6]
 \node at (-5,0) {};
 \node at (-1,1.5)  {$\pi^{\epsilon_B} \, \longleftrightarrow$};
 \draw[colori, thick,rounded corners] (0,0) --  (0,0.5) --  (4,2.5) --  (4,3);
 \draw[colorM, thick] (1,0) -- (1,3.5);
 \draw[colorM, thick] (2,0) -- (2,3.5);
 \draw[colorM, thick] (3,0) -- (3,3.5);
 \draw (1,0.5) -- (4,2) -- (4,2.5);
 \draw (2,0.5) -- (4,2);
 \draw (3,0.5) -- (4,2);
 \node at (1,0.5) [left] {$\scriptstyle \delta_r$};
 \node at (2,0.5) [left] {$\scriptstyle \delta_r$};
 \node at (3,0.5) [left] {$\scriptstyle \delta_r$};
 \node at (4,2) [right] {$\scriptstyle \mu$};
 \node at (4,2.7) [right] {$\scriptstyle \rho$};
 \node at (4,3) [colori,above] {$\scriptstyle \epsilon_B$};
 \fill [colorM] (1,0.5) circle (0.12);
 \fill [colorM] (2,0.5) circle (0.12);
 \fill [colorM] (3,0.5) circle (0.12);
 \fill (4,2) circle (0.1);
 \fill [colori] (4,2.5) circle (0.12);
 \fill [colori] (4,3) circle (0.12);
 \node at (0,0) [colori,below] {$\scriptstyle B$};
 \node at (3,0) [colorM,below] {$\scriptstyle C$};
 \node at (1,0) [colorM,below] {$\scriptstyle C$}; 
 \node at (2,0) [colorM,below] {$\scriptstyle C$};
\end{tikzpicture}
\caption{${_B}\! \ModCat_A$ structure on $C^{3}$}\label{pic:SmashBimodStructure}
\end{figure}
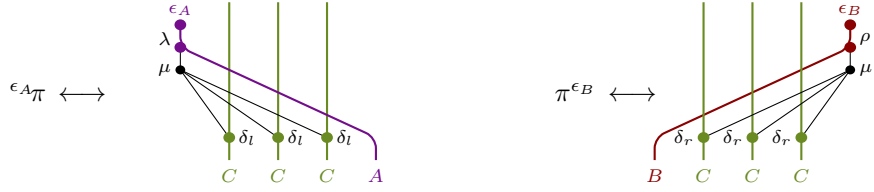

\begin{proof} 
Observe that for Point~1 of Proposition~\ref{thm:adjoint_multi_coeffs} to hold true, the additivity of~$\C$ is not necessary, and the module~$M$ can be taken in $\ModCat_{(\oV,\osigma)[t,r]}$ instead of $\ModCat_{(\oV,\osigma)}$. Thus apply Proposition~\ref{thm:adjoint_multi_coeffs} and its mirror version to the braided system of UAAs $(B,C,A)$ from Proposition~\ref{thm:smash} and to the algebra characters (hence braided characters) $\epsilon_A$ and~$\epsilon_B$. One gets right $(A,\sigma_{Ass}(A))$-module structures and left $(B,\sigma_{Ass}(B))$-module structures on all the $B^k \otimes C^{n} \otimes A^{m}$, and hence on~$C^{n}$. Further, since the~$\xi_{1,2}$ and~$\xi_{2,3}$ components of the braiding on $(B,C,A)$ are natural w.r.t. the units of~$A$ and~$B$, these units act on~$C^{n}$ trivially. Theorem~\ref{thm:UAA} then ensures that our braided module structures on~$C^{n}$ are actually module structures over the UAAs~$A$ and~$B$, which are easily checked to coincide with the desired ones. Compatibility between $A$- and $B$-actions follows from Lemma~\ref{thm:adjointBiModules}.
\end{proof}

\section{A braided interpretation of bialgebras and Hopf modules}\label{sec:bialg}

This section explores a rank~$2$ braided system~$\BBialg(H)$ encoding the \emph{bialgebra} structure on~$H$, in the same sense that $\sigma_{Ass}$ encodes the UAA structure (Table~\ref{tab:BrVsStr}). It is a particular case of the system constructed for crossed products in Proposition~\ref{thm:smash}. In~$\BBialg(H)$, the invertibility of the braiding component $\sigma_{1,2}$ is algebraically significant: it is equivalent to the existence of an \emph{antipode}. We identify braided $\BBialg(H)$-modules as \emph{Hopf modules} over~$H$, and show that the braided homology theory for~$\BBialg(H)$ includes \emph{Gerstenhaber--Schack bialgebra homology} and \emph{Panaite--{\c{S}}tefan Hopf module homology}.

Except for some general observations, we specialize here to the category $\C = \vect$ of finite-dimensional vector spaces over~$\k$. One could also work in a braided category~$\C$ and choose a bialgebra in~$\C$ admitting a dual. When working in~$\vect$, we use Sweedler's notation for comultiplications and coactions. A simplified notation
$ v_1 v_2 \ldots  v_n = v_1\otimes v_2\otimes\ldots\otimes v_n \in V^{n}$ 
is preferred for pure tensors in~$V^n$, leaving the symbol~$\otimes$ for 
$ v_1 v_2 \ldots  v_n \otimes w_1 w_2 \ldots w_m \in V^{ n}\otimes W^{ m}$.
The dual space of $V \in \vect$ is denoted by~$V^*$.  Letters~$h_i$ and~$l_j$ stand for elements of~$V$ and~$V^*$ respectively. The pairing $\left\langle,\right\rangle$ is the \emph{evaluation map} $ev\colon V^*\otimes V \to \k$, $l\otimes h \mapsto l(h)$. Multiplications on different spaces are denoted by~$\cdot$ when no confusion arises.

Consider a pairing $B\colon V\otimes W \to \k$ between $\k$-vector spaces. Table~\ref{tab:RainbowArched} presents its possible extensions to $B\colon V^{ n}\otimes W^{ n} \to \k$. The ``arched'' one is more common, but we use the ``rainbow'' one (like, for instance, Gurevich~\cite{Gurevich}), minimizing argument permutations and crossings in diagrams. This choice slightly changes some classical formulas. We use analogous conventions in the dual and multi-pairing situations. Taking as~$B$ the evaluation map~$ev$, one constructs out of a linear map $f \colon V_1\otimes \ldots \otimes V_n \to W_1\otimes \ldots \otimes W_m$ its \textit{dual} $f^*\colon W_m^*\otimes \ldots \otimes W_1^* \to V_n^*\otimes \ldots \otimes V_1^*$ (note the inverse order of factors). Graphically, this is the {central symmetry}, while the arched duality corresponds to the {horizontal mirror symmetry}. 
\begin{table}\centering
\begin{tabular}{|c|c|}
\hline 
$B(v_n\otimes w_1)\cdots B(v_1\otimes w_n)$&$B(v_1\otimes w_1)\cdots B(v_n\otimes w_n)$ \\
\hline
\begin{tikzpicture}[scale=0.5]
\foreach \r in {1,2,3} {
 \draw [rounded corners,dashed] (3-\r,0) -- (3-\r,.4*\r) -- (3,0.8*\r);
 \draw [rounded corners] (3,.8*\r) -- (3+\r,.4*\r) -- (3+\r,0);
 \node at (3+\r,0) [below] {$\scriptstyle W$};
 \node at (3-\r,0) [below] {$\scriptstyle V$};
 \node at (3,.8*\r) [below] {$\scriptstyle B$};
 \fill (3,.8*\r) circle (0.15);
}
\end{tikzpicture} & \begin{tikzpicture}[scale=0.5]
 \foreach \r in {1,2,3} {
 \draw [rounded corners,dashed] (\r,0) -- (\r,1) -- (\r+2,2);
 \draw [rounded corners] (\r+2,2) -- (4+\r,1) -- (4+\r,0);
 \node at (4+\r,0) [below] {$\scriptstyle W$};
 \node at (\r,0) [below] {$\scriptstyle V$};
 \node at (\r+2,2) [above] {$\scriptstyle B$};
 \fill (\r+2,2) circle (0.15);
}
\end{tikzpicture}\\
\hline
\bi{``rainbow''} & \bi{``arched''} \\
\hline
\end{tabular}
\caption{Two definitions of $B(v_1v_2\ldots v_n\otimes w_1w_2\ldots w_n)$}\label{tab:RainbowArched}
\end{table}

For example, the dual of a coalgebra $V \in \vect$ receives an \emph{induced algebra structure} via the rainbow extension of~$ev$: 
$\left\langle l_1l_2,h\right\rangle= \left\langle l_1,h_{(2)}\right\rangle \left\langle l_2,h_{(1)}\right\rangle$, $h \in V,\, l_1,l_2 \in V^*$ (Fig.~\ref{pic:DualViaRainbow}{A}). 
Multiplication and (co)units are dualized similarly. The same structure on~$V^*$ is obtained using the dual \emph{coevaluation map}~$coev$, or the twisted (co)pairings $ev \circ \tau\colon V \otimes V^* \to \k$ and $\tau \circ coev \colon \k\ \to  V \otimes V^*$. Here~$\tau$ is the factor transposition (which is the categorical braiding of~$\vect$). To simplify notations, we often write~$ev$ and~$coev$ even for the twisted maps.

\begin{observation}\label{thm:TwistedBialg}
Take a bialgebra $(H,\mu,\nu, \Delta,\varepsilon)$ in a {braided} category $(\C,\otimes,\II,c)$.
\begin{enumerate}
\item The data $H^{op}:=(H,\mu  c^{-1},\nu, \Delta,\varepsilon)$ and $H^{cop}:=(H,\mu, \nu, c^{-1}  \Delta,\varepsilon)$ define bialgebras in $(\C,\otimes,\II,c^{-1})$. The data  $H^{op,cop}:=(H,\mu  c^{-1},\nu, c\Delta,\varepsilon)$ and $H^{cop,op}:=(H,\mu  c, \nu, c^{-1}  \Delta,\varepsilon)$ define bialgebras in $(\C,\otimes,\II,c)$. 

\item If~$H$ is a Hopf algebra with an antipode~$s$, then so are $H^{op,cop}$ and $H^{cop,op}$, with the same antipode. If~$s$ is invertible, then $s^{-1}$ is an antipode for $H^{op}$ and $H^{cop}$.

\item One has the following bialgebra or Hopf algebra isomorphisms:
\[(H^{op})^* \simeq (H^*)^{cop}, \qquad (H^{cop})^* \simeq (H^*)^{op}, \qquad (H^{op,cop})^* \simeq (H^*)^{cop,op}.\]
\end{enumerate}
\end{observation}

\begin{notation}\label{not:TwistedCoMult}
Twisted (co)multiplication is denoted by $\mu^{op}=\mu  c^{-1}$, $\Delta^{cop}=c^{-1}\Delta$.
\end{notation}

Depending on the context, notations $H$, $H^*$, $H^{op}$, etc. will denote the corresponding bialgebra, Hopf algebra, (co)algebra, or vector space. 

Lemma~\ref{thm:lMod=rMod} allows one to switch between {left} $V$-modules and {right} $V^{op}$-modules. We now give an analogous transition tool for modules and {co}modules.

\begin{lemma}\label{thm:Mod=Comod}
For a coalgebra~$V$ in~$\vect$, the following functors (completed by identities on morphisms) yield a category equivalence:
\begin{align}
\ModCat^V &\overset{\sim}{\longleftrightarrow} \ModCat_{V^*}, \notag\\
(M,\delta) & \longmapsto (M,\delta^{co}:= (\Id_M \otimes ev)  (\delta \otimes \Id_{V^*})),\label{eqn:HtoH*}\\
(M,\rho^{co}:= (\rho  \otimes \Id_{V})   (\Id_M \otimes coev)) &\longmapsfrom (M,\rho).\label{eqn:H*toH}
\end{align}
\end{lemma} 

The proof is routine and is best done graphically. A diagrammatic version of the transformation~\eqref{eqn:HtoH*} is given in Fig.~\ref{pic:DualViaRainbow}{B}. With the arched dualities, one would have to take the category $\ModCat_{(V^*)^{op}}$ on the right.

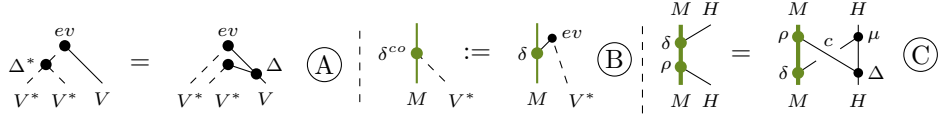
\begin{figure}\centering
\begin{tikzpicture}[scale=0.5]
 \draw [dashed] (0,0) -- (1,1);
 \draw [dashed] (1,0) -- (0.5,0.5);
 \draw (1,1) -- (2,0);
 \node at (0,0) [below] {$\scriptstyle V^*$};
 \node at (1,0) [below] {$\scriptstyle V^*$};
 \node at (2,0) [below] {$\scriptstyle V$};
 \node at (1,1) [above] {$\scriptstyle ev$};
 \fill (1,1) circle (0.15);
 \fill (0.5,0.5) circle (0.15);
 \node at (0.5,0.5) [left] {$\scriptstyle \Delta^*$};
 \node at (3,0.5) {$=$};
\end{tikzpicture}
\begin{tikzpicture}[scale=0.5]
 \draw [dashed] (0,0) -- (1,1);
 \draw [dashed] (0.5,0) -- (1,0.5);
 \draw (1,1) -- (2,0);
 \draw (1,0.5) -- (1.75,0.25);
 \node at (0,0) [below] {$\scriptstyle V^*$};
 \node at (1,0) [below] {$\scriptstyle V^*$};
 \node at (2,0) [below] {$\scriptstyle V$};
 \node at (1,1) [above] {$\scriptstyle ev$};
 \fill (1,1) circle (0.15);
 \fill (1,0.5) circle (0.15);
 \fill (1.75,0.25) circle (0.15);
 \node at (1.75,0.5) [right] {$\scriptstyle \Delta$};
 \node at (3.5,0.5) {\circled{A}};
 \draw [dashed] (4.5,-0.5) -- (4.5,1.5);
\end{tikzpicture}
\begin{tikzpicture}[scale=0.4]
 \draw [colorM, thick] (1,4) -- (1,2);
 \draw [dashed] (1,3) -- (2,2);
 \node at (1,2) [below] {$\scriptstyle M$}; 
 \node at (2.5,2) [below] {$\scriptstyle V^*$};
 \node at (1,3) [left] {$\scriptstyle \delta^{co}$};
 \fill [colorM] (1,3) circle (0.2);
\node  at (3,3){$:=$};
\end{tikzpicture}
\begin{tikzpicture}[scale=0.4]
 \draw [colorM, thick] (1,4) -- (1,2);
 \draw (1,3) -- (1.5,3.5);
 \draw [dashed] (2,2) -- (1.5,3.5);
 \fill (1.5,3.5) circle (0.15);
 \fill [colorM] (1,3) circle (0.2);
 \node at (1.5,3.2) [above right]{$\scriptstyle ev$};
 \node at (1,3) [left]{$\scriptstyle \delta$};
 \node at (1,2) [below] {$\scriptstyle M$}; 
 \node at (2.5,2) [below] {$\scriptstyle V^*$};
 \node at (3.5,2.8) {\circled{B}};
 \draw [dashed] (4.5,1) -- (4.5,4); 
\end{tikzpicture}
\begin{tikzpicture}[xscale=0.8, yscale = 0.8]
 \draw [colorM, ultra thick] (0.5,0) --  (0.5,1);
 \draw (1,0) -- (0.5,0.3);
 \draw (0.5,0.7) -- (1,1);
 \node at (0.5,0.3) [left] {$\scriptstyle \rho$};
 \node at (0.5,0.7) [left] {$\scriptstyle \delta$};
 \fill [colorM] (0.5,0.3) circle (0.1);
 \fill [colorM] (0.5,0.7) circle (0.1);
 \node at (1.5,0.5) {$=$};
 \node at (0.5,0) [below] {$\scriptstyle M$}; 
 \node at (1,0) [below] {$\scriptstyle H$};
 \node at (0.5,1) [above] {$\scriptstyle M$}; 
 \node at (1,1) [above] {$\scriptstyle H$};
\end{tikzpicture}
\begin{tikzpicture}[xscale=0.8, yscale = 0.8] 
 \draw[colorM, ultra thick] (3,0) -- (3,1);
 \draw (4,1) -- (4,0);
 \draw (3,0.2) -- (3.3,0.4); 
 \draw (3.7,0.6) -- (4,0.8);
 \draw (3,0.8) -- (4,0.2);
 \fill [colorM] (3,0.2) circle (0.1);
 \fill [colorM] (3,0.8) circle (0.1);
 \fill (4,0.8) circle (0.08);
 \fill (4,0.2) circle (0.08);
 \node at (3,0.8) [left] {$\scriptstyle \rho$};
 \node at (3,0.2) [left] {$\scriptstyle \delta$};
 \node at (4,0.8) [right] {$\scriptstyle \mu$};
 \node at (4,0.2) [right] {$\scriptstyle \Delta$};
 \node at (3.5,0.5) [above] {$\scriptstyle c$};
 \node at (5,0.5) {\circled{C}}; 
 \node at (3,0) [below] {$\scriptstyle M$}; 
 \node at (4,0) [below] {$\scriptstyle H$};
 \node at (3,1) [above] {$\scriptstyle M$}; 
 \node at (4,1) [above] {$\scriptstyle H$};
\end{tikzpicture}
\caption{Multiplication-comultiplication and action-coaction dualities, and Hopf compatibility} \label{pic:DualViaRainbow}
\end{figure}

\begin{convention}
Here and below thin lines stand for the basic vector space, dashed lines for its dual, and thick colored lines for different types of modules over it.
\end{convention}

\begin{lemma}\label{thm:ModAlg=ComodAlg}
For a bialgebra~$H$ in~$\vect$, the functors from Lemmas~\ref{thm:lMod=rMod} and~\ref{thm:Mod=Comod} induce category equivalences 
\begin{align*}
\ModAlg^H &\overset{\sim}{\longleftrightarrow} \ModAlg_{(H^*)^{cop}}, &{_{H}}\!\ModAlg &\overset{\sim}{\longleftrightarrow}  \ModAlg_{H^{op}}.
\end{align*}
\end{lemma}

We now include the groupoid $^*\!\Bialg(\vect)$ of bialgebras and bialgebra isomorphisms in~$\vect$ into the groupoid of bipointed rank~$2$ braided systems in~$\vect$, taking inspiration from the pointed rank~$1$ system interpretation of UAAs (Theorem~\ref{thm:UAA}). 

\begin{definition}\label{def:Norm2}
Given a monoidal category~$\C$, let $^*\!\BrSystBP_r (\C)$ be the category of
\begin{itemize}
\item rank~$r$ \emph{bipointed braided systems}, i.e., $(\oV,\osigma) \in \BrSyst_r(\C)$ enriched with distinguished morphisms $\onu= (\nu_i \colon \II \to V_i)_{1 \le i \le r}$ and  $\overline{\varepsilon}=(\varepsilon_i \colon V_i \to \II )_{1 \le i \le r}$, called \emph{units} and \emph{counits}, forming normalized  pairs $(\nu_i,\varepsilon_i)$ for all~$i$, and
\item isomorphisms from $\BrSyst_r(\C)$ preserving the units and the counits.
\end{itemize}
\end{definition}

\begin{definition}\label{def:cat_Hopf}
A \emph{(right-right) Hopf module} over a bialgebra~$H$ in a braided category~$\C$ is an object~$M$ endowed with right module and comodule structures $\rho \colon M \otimes H \to M$, $\delta \colon M \to M  \otimes H$, satisfying the \emph{Hopf compatibility condition} (Fig.~\ref{pic:DualViaRainbow}{C}):
\begin{equation}\label{eqn:rrHopfMod}
\delta  \rho = (\rho \otimes \mu) (\Id_M \otimes c_{H,H} \otimes \Id_H)  (\delta \otimes \Delta) \quad : \quad M \otimes H \rightarrow  M \otimes H.
\end{equation}
The category of such modules and their morphisms is denoted by $\ModCat_H^H$. 
\end{definition}

An important example of $H$-Hopf module is $H$ itself, with $\rho=\mu_H$, $\delta= \Delta_H$.

We now return to our category $\vect$, omitted in further notations.

\begin{theorem}\label{thm:Bialg}
\begin{enumerate}
\item\label{item:CatInclBi}
One has a fully faithful functor
\begin{align}
\F :\; {^*}\!\Bialg \longhookrightarrow&  ^*\!\BrSystBP_2 \label{eqn:BialgAsBrSyst}\\
(H,\mu,\nu,\Delta,\varepsilon) \longmapsto& \BBialg(H):=(V_1:=H,V_2:=H^*; \notag\\
& \quad \sigma_{1,1}:=\sigma_{Ass}^r(H),\sigma_{2,2}:=\sigma_{Ass}(H^*),\sigma_{1,2}=\sigma_{bi};\, \nu,\varepsilon^*; \varepsilon,\nu^*),\notag\\
f \longmapsto& (f,(f^{-1})^*),\notag
\end{align}

where $\sigma_{bi}(h \otimes l)= \left\langle l_{(1)},h_{(2)}\right\rangle l_{(2)} \otimes h_{(1)}$ (Fig.~\ref{pic:sigmaHH*}).

\item\label{item:non_invBi} For a bialgebra~$H$, $\sigma_{bi}$ is invertible if and only if~$H$ has an antipode.

\item\label{item:EqnEquivBi} Take an $H \in \vect$ with UAA and coUAA structures $(\mu,\nu)$ and $(\Delta,\varepsilon)$. Suppose the pair $(\nu,\varepsilon)$ normalized. Then the YBE on $H\otimes H \otimes H^*$ (symmetrically, on $H\otimes H^* \otimes H^*$) for~$\BBialg(H)$, together with the naturality of~$\sigma_{bi}$ with respect to the units, are equivalent to the bialgebra compatibility conditions~\eqref{eqn:cat_bialg} for~$H$. 

\item\label{item:ModEquivBi} For a bialgebra~$H$, one has category equivalences 

$\qquad\qquad\qquad$ {\renewcommand{\arraystretch}{1.2}
\begin{tabular}{ccccc} 
$\ModCat_H^H$ & $\stackrel{\sim}{\longrightarrow}$ & $\ModCatN_{\BBialg(H)}$ & $\stackrel{\sim}{\longrightarrow}$ & $\ModCat_{H^* \underset{\sigma_{bi}}{\otimes} H}$\\ 
$(M,\rho, \delta)$ & $\longmapsto$ & $(M;\rho, \delta^{co})$ & $\longmapsto$ & $(M, \delta^{co} \otimes  \rho)$
\end{tabular}}

If~$H$ is a Hopf algebra with an antipode~$s$, then this chain can be continued on the left by $\,\ModCat_{H \underset{\theta}{\otimes} H^*} \simeq \ModCatN_{s_1 \cdot \BBialg(H)} \simeq  \ModCat^H_H$, where $\theta=\sigma_{bi}^{-1}$.
\end{enumerate}
\end{theorem} 

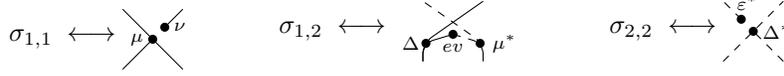
\begin{figure}\centering
\begin{tikzpicture}[scale=0.8]
 \node at (-1,0.5) {$\sigma_{1,1} \; \longleftrightarrow $};
 \draw (0,1) -- (1,0);
 \draw (0,0) -- (0.5,0.5);
 \draw (0.7,0.7) -- (1,1);
 \node at (0.5,0.5) [left]{$\scriptstyle \mu$};
 \node at (0.7,0.7) [right] {$\scriptstyle \nu$}; 
 \fill (0.7,0.7) circle (0.075);
 \fill (0.5,0.5) circle (0.075);
\end{tikzpicture}
\begin{tikzpicture}[scale=0.8]
 \node at (-3.5,0) {};
 \node at (-1.5,0.5) {$\sigma_{1,2} \; \longleftrightarrow $};
 \draw [dashed,rounded corners] (1,0)  -- (1,0.3) -- (0,1);
 \draw [rounded corners] (0,0) -- (0,0.3) -- (1,1);
 \draw [dashed,rounded corners] (0.5,0.45) -- (1,0.3) -- (1,0);
 \draw [rounded corners](0,0) -- (0,0.3) -- (0.5,0.45);
 \node at (0.5,0.45) [below] {$\scriptstyle ev$};
 \fill (0.5,0.45) circle (0.075);
 \fill (0.05,0.3) circle (0.075);
 \fill (0.95,0.3) circle (0.075);
 \node at (0.1,0.3) [left] {$\scriptstyle \Delta$};
 \node at (1,0.3) [right] {$\scriptstyle \mu^*$};
\end{tikzpicture}
\begin{tikzpicture}[scale=0.8]
 \node at (-3,0) {};
 \node at (-1,0.5) {$\sigma_{2,2} \; \longleftrightarrow $};
 \draw [dashed](0,0) -- (1,1);
 \draw [dashed](1,0) -- (0.5,0.5);
 \draw [dashed](0.3,0.7) -- (0,1);
 \node at (0.5,0.5) [right]{$\scriptstyle \Delta^*$};
 \node at (0.4,0.7) [above] {$\scriptstyle \varepsilon^*$}; 
 \fill (0.3,0.7) circle (0.075);
 \fill (0.5,0.5) circle (0.075);
\end{tikzpicture}
\caption{A braiding encoding the bialgebra structure}\label{pic:sigmaHH*}
\end{figure}

The graphical interpretation suggests that, applied to the dual bialgebra~$H^*$ instead of~$H$, the construction yields a vertical mirror version of the system~$\BBialg(H)$.

\begin{proof}
Take a bialgebra~$H$. Recall Notation~\ref{not:TwistedCoMult}. Consider the left $H^*$-comodule algebra $(H^*,\Delta^*,\varepsilon^*, \mu^*)$. (A left version of) Lemma~\ref{thm:ModAlg=ComodAlg} transforms it into a left $H^{cop}$-module algebra $(H^*,\Delta^*,\varepsilon^*,(\mu^*)^{co})$. Together with the $H^{cop}$-bicomodule algebra $(H^{cop},\mu,\nu,\Delta^{cop}, \Delta^{cop})$, it can be fed into Proposition~\ref{thm:smash} as the~$A$ and~$C$ parts (as explained after that proposition, the~$B$ part can be omitted). The~$\xi_{2,3}$ component of the braided system from that proposition coincides with~$\sigma_{bi}$. Further, $H^{cop}$ and~$H$ share the same UAA structure, hence our~$\sigma_{i-1,i-1}$  can be chosen as the~$\xi_{i,i}$ components (Remark~\ref{rmk:pre_mirror}). Proposition~\ref{thm:smash} then implies that~$\BBialg(H)$ is a braided system of UAAs. It is clearly bipointed. Moreover, the braiding on~$\BBialg(H)$, the units and the counits suffice to recover all ingredients of the bialgebra structure on~$H$, hence the functor~$\F$ is injective on objects.

To prove Point~\ref{item:CatInclBi}, it remains to understand, for bialgebras~$H$ and~$K$, isomorphisms of bipointed braided systems $(f,g)\colon \BBialg(H) \to \BBialg(K)$. By definition, they consist of bijections $f \colon H \to K$, $g \colon H^* \to K^*$ intertwining the braidings of~$\BBialg(H)$ and~$\BBialg(K)$ and respecting the (co)units. Due to Theorem \ref{thm:UAA} (Point~\ref{item:functor}), this means that $f$ and~$g$ are UAA isomorphisms compatible with counits ($\varepsilon_K  f = \varepsilon_H$, $\nu_K^*  g = \nu_H^*$), and satisfy
\begin{equation}\label{eqn:RespectSigmaBi}
\sigma_{bi}(K)  (f \otimes g) = (g \otimes f)  \sigma_{bi}(H)
\end{equation}
(Fig.~\ref{pic:SigmaBi}{A}). Applying $\nu_K^* \otimes \varepsilon_K$ to both sides of~\eqref{eqn:RespectSigmaBi}, using the compatibility of~$f$ and~$g$ with the counits, and playing with dualities, one deduces  $g^* f = \Id_H$, hence $g =(f^{-1})^*$. Since~$g$ is a UAA isomorphism, so is~$g^{-1}$, hence $f=(g^{-1})^*$ is a coUAA morphism, which completes its properties and shows that it is a bialgebra isomorphism. Reversing the argument, one checks that the choice $g =(f^{-1})^*$ for a bialgebra isomorphism~$f$ implies~\eqref{eqn:RespectSigmaBi}. Thus the bipointed braided system isomorphisms are precisely the pairs $(f,(f^{-1})^*)$ for bialgebra isomorphisms~$f$. Hence the functor~$\F$ is well defined, full and faithful. This finishes the proof of Point~\ref{item:CatInclBi}.
 
In Point~\ref{item:EqnEquivBi}, the compatibility between~$\Delta$ and~$\nu$ follows by applying $\nu^* \otimes \Id_H$ to the naturality condition for~$\sigma_{bi}$ w.r.t.~$\nu$. Symmetrically, the $\mu$-$\varepsilon$ compatibility follows from the naturality of~$\sigma_{bi}$ w.r.t.~$\varepsilon^*$. The converse (compatibility $\Rightarrow$ naturality) is easy. According to (the proof of) Theorem~\ref{thm:ProdOfAlgebrasBrFamily}, the YBE on $H\otimes H \otimes H^*$ is equivalent to the naturality condition of~$\sigma_{bi}$ w.r.t.~$\mu$ (Fig.~\ref{pic:SigmaBi}{B}), which implies the bialgebra $\mu$-$\Delta$ compatibility (apply $\nu^* \otimes \Id_H$ to both sides and use duality). Conversely, the bialgebra compatibility suffices to deduce the above naturality. By symmetry, one gets a proof for $H \otimes H^* \otimes H^*$. 

The ``if'' part of Point~\ref{item:non_invBi} can be proved by exhibiting an explicit formula for~$\sigma_{bi}^{-1}$:
\begin{equation}\label{eqn:sigma_bi_inv}
\sigma_{bi}^{-1}(l \otimes h) = \left\langle l_{(1)},s(h_{(2)})\right\rangle  h_{(1)} \otimes l_{(2)} 
\end{equation}
(or by using the remarks after Proposition~\ref{thm:smash} and Point~2 of Observation~\ref{thm:TwistedBialg}). The ``only if'' part is more delicate. Suppose the existence of~$\sigma_{bi}^{-1}$ and put
\begin{equation}
\widetilde{s}=(((\varepsilon\otimes\nu^*)\sigma_{bi}^{-1})\otimes\Id_H)  (\Id_{H^*} \otimes c_{H,H}) (coev \otimes \Id_H) \quad : \quad H\to H\notag
\end{equation}
(Fig.~\ref{pic:SigmaBi}). Let us prove that~$\widetilde{s}$ is the antipode. The part
\begin{equation}\label{eqn:ws_left}
\mu  (\widetilde{s}\otimes \Id_H)  \Delta  = \nu  \varepsilon
\end{equation}
of the defining relation~\eqref{eqn:s} follows from $\sigma_{bi}^{-1} \sigma_{bi} = \Id_{H\otimes  H^*}$ by duality manipulations. Surprisingly, the remaining part $\mu  (\Id_H \otimes \widetilde{s})  \Delta  = \nu  \varepsilon$ does not seem to follow from $\sigma_{bi} \sigma_{bi}^{-1} = \Id_{H^*\otimes  H}$. Algebraic tricks come into play instead. Mimicking~\eqref{eqn:sigma_bi_inv}, set 
\[\widetilde{\sigma}=(\Id_H \otimes (ev  (\widetilde{s} \otimes \Id_{H^*})) \otimes \Id_{H^*})  (\Delta \otimes \mu^* )  c_{H^*,H} \quad : \quad H^*\otimes  H \to H\otimes  H^*.\]
Relation~\eqref{eqn:ws_left} implies $\widetilde{\sigma}  \sigma_{bi} = \Id_{H\otimes H^*}$. Then $\widetilde{\sigma}$ coincides with~$\sigma_{bi}^{-1}$, giving $\sigma_{bi} \widetilde{\sigma}= \Id_{H^*\otimes H}$. Applying $\nu^* \otimes \varepsilon$ to both sides, one recovers the second part of~\eqref{eqn:s} for~$\widetilde{s}$.

\begin{figure}\centering
\begin{tikzpicture}[xscale=0.7,yscale=0.6]
 \draw [dashed,rounded corners] (1,-0.25)  -- (1,0.3) -- (0,1) -- (0,1.25);
 \draw [dashed] (1,-0.5)  -- (1,-0.25);
 \draw [rounded corners] (0,-0.25) -- (0,0.3) -- (1,1) -- (1,1.25);
 \draw  (0,-0.5) -- (0,-0.25);
 \draw [dashed,rounded corners] (0.5,0.4) -- (0.7,0.3) -- (1,0);
 \draw [rounded corners](0,0) -- (0.3,0.3) -- (0.5,0.4);
 \fill (0.5,0.4) circle (0.075);
 \fill (0,-0.25) circle (0.075);
 \fill (1,-0.25) circle (0.075);
 \node at (0,-0.25)  [left] {$\scriptstyle f$};
 \node at (1,-0.25)  [right] {$\scriptstyle g$};
 \node at (0,-0.5)  [below] {$\scriptstyle H$};
 \node at (1,-0.5)  [below] {$\scriptstyle H^*$};
 \node at (0,1.25)  [above] {$\scriptstyle K^*$};
 \node at (1,1.25)  [above] {$\scriptstyle K$};
 \node at (2,0.3) {$=$};
\end{tikzpicture}
\begin{tikzpicture}[xscale=0.7,yscale=0.6]
 \draw [dashed,rounded corners] (1,-0.25)  -- (1,0.3) -- (0,1);
 \draw [dashed] (0,1.5)  -- (0,1);
 \draw [rounded corners] (0,-0.25) -- (0,0.3) -- (1,1);
 \draw  (1,1) -- (1,1.5);
 \draw [dashed,rounded corners] (0.5,0.4) -- (0.7,0.3) -- (1,0);
 \draw [rounded corners](0,0) -- (0.3,0.3) -- (0.5,0.4);
 \fill (0.5,0.4) circle (0.075);
 \fill (0,1) circle (0.075);
 \fill (1,1) circle (0.075);
 \node at (0,1)  [left] {$\scriptstyle g$};
 \node at (1,1)  [right] {$\scriptstyle f$};
 \node at (0,-0.25)  [below] {$\scriptstyle H$};
 \node at (1,-0.25)  [below] {$\scriptstyle H^*$};
 \node at (0,1.5)  [above] {$\scriptstyle K^*$};
 \node at (1,1.5)  [above] {$\scriptstyle K$};
 \node at (1.75,0.25) {\circled{A}};
 \draw [dotted] (2.25,-0.75) -- (2.25,2);
\end{tikzpicture}
\begin{tikzpicture}[xscale=0.4,yscale=0.4]
 \draw[rounded corners] (0,0) -- (4,4);
 \draw (2,0) -- (1,1);
 \draw[dashed,rounded corners] (4,0) -- (4,1) -- (2,3) -- (2,4);
 \draw (1.5,1.5) -- (2.5,2);
 \draw[dashed] (3.5,1.5) -- (2.5,2);
 \fill (1,1) circle (0.15);
 \fill (1.5,1.5) circle (0.15);
 \fill (3.5,1.5) circle (0.15);
 \fill (2.5,2) circle (0.15);
 \node  at (5,2){${=}$};
 \node at (5,-0.25) {}; 
\end{tikzpicture}
\begin{tikzpicture}[xscale=0.4,yscale=0.4]
 \draw[rounded corners] (0,0) -- (0,0.5) -- (3.5,4);
 \draw[rounded corners] (2,0) -- (2,0.5) -- (3,3.5);
 \draw[dashed,rounded corners] (4,0) -- (4,0.5) -- (0.5,4);
 \draw (3,1) -- (2,0.5);
 \draw[dashed] (4,0.5) -- (3,1);
 \draw (2,2) -- (1,1.5);
 \draw[dashed] (3,1.5) -- (2,2);
 \fill (3,3.5) circle (0.15);
 \fill (2,0.5) circle (0.15);
 \fill (1,1.5) circle (0.15);
 \fill (3,1.5) circle (0.15);
 \fill (4,0.5) circle (0.15);
 \fill (3,1) circle (0.15);
 \fill (2,2) circle (0.15);
 \node at (5,1.5) {\circled{B}};
 \node at (5,-0.25) {}; 
 \draw [dotted] (6,0) -- (6,4);
\end{tikzpicture}
\begin{tikzpicture}[xscale=.45,yscale=.45]
 \node at (-1.5,2) {$\widetilde{s} \,\longleftrightarrow$};
 \draw [dashed] (1,1)  -- (0,2);
 \draw [dashed] (1,3) -- (1,3.5);
 \draw [rounded corners](2,0) -- (2,1) -- (1,2);
 \draw [rounded corners](1,1) -- (2,2) -- (2,4);
 \draw (0,3) -- (0,3.5);
 \node at (1,1) [below] {$\scriptstyle coev$};
 \fill (1,1) circle (0.07);
 \fill (1,3.5) circle (0.1);
 \node at (1,3.5) [above] {$\scriptstyle \nu^*$};
 \fill (0,3.5) circle (0.1);
 \node at (0,3.5) [above] {$\scriptstyle \varepsilon$};
 \draw (-0.2,3) rectangle (1.2,2);
 \node at (0.5,2.5) {$\sigma_{bi}^{-1}$};
 \node at (3,1.75) {\circled{C}}; 
\end{tikzpicture}
\caption{Naturality and invertibility issues for~$\sigma_{bi}$}\label{pic:SigmaBi}
\end{figure}
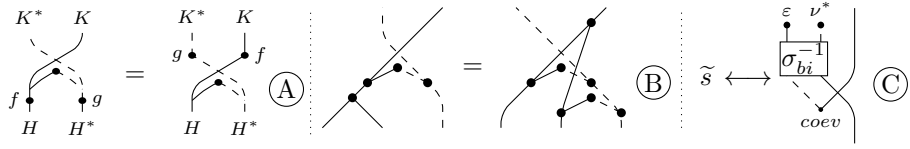

We now move to Point~\ref{item:ModEquivBi}. Equivalence $\ModCatN_{\BBialg(H)} \stackrel{\sim}{\rightarrow} \ModCat_{H^* \underset{\sigma_{bi}}{\otimes} H}$ follows from Proposition~\ref{thm:ProdOfAlgebrasBrFamilyMod}. Further, Observation~\ref{rmk:MultiMod}, combined with Point~\ref{item:ModEquiv} of Theorem~\ref{thm:UAA}, present a right $\BBialg(H)$-module~$M$ via right module structures~$\rho_H$ and~$\rho_{H^*}$ over the UAAs~$H$ and~$H^*$ respectively, compatible in the sense of~\eqref{eqn:MultiBrMod}:
\begin{equation}\label{eqn:CompatForHbiMod}
 \rho_{H^*}  (\rho_H \otimes \Id_{H^*})=\rho_H  (\rho_{H^*} \otimes \Id_{H}) (\Id_M\otimes (\tau  (\Id_H \otimes ev \otimes Id_{H^*})  (\Delta \otimes \mu^*))).
\end{equation}
On the other hand, due to the module-comodule duality from Lemma~\ref{thm:Mod=Comod}, a right-right Hopf module structure over~$H$ can also be viewed as right module structures over the UAAs~$H$ and~$H^*$, with the compatibility condition obtained by applying $\Id_M \otimes ev$ to the defining condition~\eqref{eqn:rrHopfMod} of Hopf modules (tensored with~$\Id_{H^*}$ on the right) and turning $H$-comodule structures into $H^*$-module structures. The condition obtained coincides with~\eqref{eqn:CompatForHbiMod}, implying $\ModCat^H_H \simeq \ModCatN_{\BBialg(H)}$. 

In the Hopf algebra case, Point~\ref{item:non_invBi} gives the invertibility of~$\sigma_{bi}$. The component permuting Proposition~\ref{thm:ProdOfAlgebrasBrFamilyInv} proves then the desired equivalences.
\end{proof}

All the remarks following Theorem~\ref{thm:UAA} remain relevant in the bialgebra case. One particular feature of the bialgebra setting is to be added to that list:

\begin{remark} It is essential to work in the \emph{groupoid}, and not just in the category of bialgebras, if one wants a bialgebra morphism $H\rightarrow G$ to induce a morphism of dual bialgebras $H^*\rightarrow G^*$, so that the functor~\eqref{eqn:BialgAsBrSyst} can be defined on morphisms.
\end{remark}

Denote by $\Hei'(H)=H \underset{\theta}{\otimes} H^*$ one of the braided tensor products of UAAs from the theorem. Then $\Hei(H):=\Hei'(H^*)$ is the well-known \emph{Heisenberg double} of the Hopf algebra~$H$ (cf. for example \cite{Montgomery,CibilsRosso}). 

Our next goal are explicit braided complexes for~$\BBialg(H)$. After detailed calculations with certain braided characters as coefficients, we discuss the general case of Hopf module coefficients.

First, for a bialgebra~$H$, we study adjoint actions of~$H^*$ on~$H^n$.

\begin{lemma}\label{thm:AdjActionsForBialg}
The tensor powers of a bialgebra $(H,\mu,\nu,\Delta,\varepsilon)$ in~$\vect$ can be endowed with an $H^*$-bimodule structure via the following formulas (Fig.~\ref{pic:HAsH*Bimod}):
\begin{align*}
\pi^{H^*} &=\pi^{\varepsilon_{H^*}} = ev^1  ev^2 \cdots  ev^n   (((\mu^*)^1)^{(n-1)} \otimes (\omega_{2n}^{-1}  \Delta^{\otimes n})) \quad : \quad H^* \otimes H^{ n} \to H^{ n},\\
{^{H^*}}\!\pi &={^{\varepsilon_{H^*}}}\!\pi = ev^{n+1}  ev^{n+2} \cdots  ev^{2n}  ((\omega_{2n}^{-1}  \Delta^{\otimes n}) \otimes ((\mu^*)^1)^{(n-1)} ) \quad : \quad H^{ n} \otimes H^* \to H^{ n},
\end{align*}
where notations~\eqref{eqn:phi_i} and~\eqref{eqn:omega} are used.
\end{lemma}

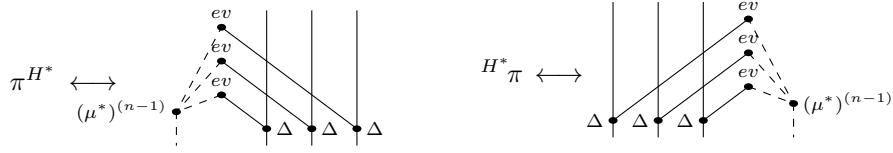
\begin{figure}\centering
\begin{tikzpicture}[xscale=0.6,yscale=0.45]
 \node at (-4.5,2) {$\pi^{H^*} \, \longleftrightarrow$};
 \draw (0,0) --  (0,4);
 \draw (1,0) -- (1,4);
 \draw (2,0) -- (2,4);
 \draw[dashed] (-1,1.5) -- (-2,1) -- (-2,0);
 \draw[dashed] (-1,2.5) -- (-2,1) -- (-1,3.5);
 \draw (0,0.5) -- (-1,1.5);
 \draw (1,0.5) -- (-1,2.5);
 \draw (2,0.5) -- (-1,3.5);
 \node at (0,0.5) [right] {$\scriptstyle \Delta$};
 \node at (1,0.5) [right] {$\scriptstyle \Delta$};
 \node at (2,0.5) [right] {$\scriptstyle \Delta$};
 \node at (-2,1) [left] {$\scriptstyle (\mu^*)^{(n-1)}$};
 \node at (-1,1.5) [above] {$\scriptstyle ev$};
 \node at (-1,2.5) [above] {$\scriptstyle ev$};
 \node at (-1,3.5) [above] {$\scriptstyle ev$};
 \fill (0,0.5) circle (0.1);
 \fill (1,0.5) circle (0.1);
 \fill (2,0.5) circle (0.1);
 \fill (-2,1) circle (0.1);
 \fill (-1,1.5) circle (0.1);
 \fill (-1,2.5) circle (0.1);
 \fill (-1,3.5) circle (0.1);
\end{tikzpicture}
\begin{tikzpicture}[xscale=0.6,yscale=0.45]
 \node at (-3.5,0) {};
 \node at (-0.8,2) {${^{H^*}}\!\pi \, \longleftrightarrow$};
 \draw (1,0) -- (1,4);
 \draw (2,0) -- (2,4);
 \draw (3,0) -- (3,4);
 \draw[dashed] (4,1.5) -- (5,1) -- (5,0);
 \draw[dashed] (4,2.5) -- (5,1) -- (4,3.5);
 \draw (1,0.5) -- (4,3.5);
 \draw (2,0.5) -- (4,2.5);
 \draw (3,0.5) -- (4,1.5);
 \node at (1,0.5) [left] {$\scriptstyle \Delta$};
 \node at (2,0.5) [left] {$\scriptstyle \Delta$};
 \node at (3,0.5) [left] {$\scriptstyle \Delta$};
 \node at (5,1) [right] {$\scriptstyle (\mu^*)^{(n-1)}$};
 \node at (4,1.5) [above] {$\scriptstyle ev$};
 \node at (4,2.5) [above] {$\scriptstyle ev$};
 \node at (4,3.5) [above] {$\scriptstyle ev$};
 \fill (1,0.5) circle (0.1);
 \fill (2,0.5) circle (0.1);
 \fill (3,0.5) circle (0.1);
 \fill (5,1) circle (0.1);
 \fill (4,1.5) circle (0.1);
 \fill (4,2.5) circle (0.1);
 \fill (4,3.5) circle (0.1);
\end{tikzpicture}
\caption{$H^{n}$ as an $H^*$-bimodule}\label{pic:HAsH*Bimod}
\end{figure}
On the level of elements, the formulas can be written as
\begin{align*}
\pi^{H^*}&(l \otimes h_1\ldots h_n) = \left\langle l_{(1)},h_{n(1)}\right\rangle \left\langle l_{(2)},h_{n-1(1)}\right\rangle\ldots \left\langle l_{(n)},h_{1(1)}\right\rangle h_{1(2)}\ldots h_{n(2)},\\
{^{H^*}}\!\pi&(h_1\ldots h_n \otimes l) = \left\langle l_{(1)},h_{n(2)}\right\rangle \left\langle l_{(2)},h_{n-1(2)}\right\rangle\ldots \left\langle l_{(n)},h_{1(2)}\right\rangle h_{1(1)}\ldots h_{n(1)}.
\end{align*}

\begin{proof}
In the proof of Theorem~\ref{thm:Bialg}, we observed that Proposition~\ref{thm:smash} applies to $A=(H^*,(\mu^*)^{co})\in\, _{H^{cop}}\!\ModAlg$ and  $C=(H^{cop},\Delta^{cop}, \Delta^{cop})\in\, ^{H^{cop}}\!\ModAlg^{H^{cop}}$ (recall Notation~\ref{not:TwistedCoMult}). Symmetry considerations allow to complete this couple with $B=(H^*,(\mu^*)^{co})$ $\in \ModAlg_{H^{cop}}$, and feed it into Proposition~\ref{thm:SmashadjointModules} together with the counit $\varepsilon_{H^*}=(\nu_H)^*$ of~$H^*$. This counit is an algebra character of~$H^*$ and hence of~$A$ and~$B$. The output yields the desired actions.
\end{proof}

Interchanging the roles of~$H$ and~$H^*$, one gets $H$-bimodules $((H^*)^{m},\pi^{H},{^{H}}\!\pi)$. By abuse of notation, we define, for all $m,n \in \NN$ for which this makes sense, the following morphisms from $H^{ n}\otimes (H^*)^{ m}$ to $H^{ (n-1)}\otimes (H^*)^{ m}$ or to $H^{ n}\otimes (H^*)^{ (m-1)}$:
\begin{align*}
{^{H^*}}\!\pi &= {^{H^*}}\!\pi \otimes \Id_{H^*}^{\otimes (m-1)},&
\pi^{H^*} &= (\pi^{H^*} \otimes \Id_{H^*}^{\otimes (m-1)})  \tau_{H^{ n}\otimes (H^*)^{(m-1)},H^*},\\
\pi^{H} &= \Id_{H}^{\otimes (n-1)} \otimes \pi^{H},&
{^{H}}\!\pi &= (\Id_{H}^{\otimes (n-1)} \otimes {^{H}}\!\pi)  \tau_{H,H^{(n-1)}\otimes (H^*)^{m}}.
\end{align*}

\begin{lemma}\label{thm:AdjActionsForBialgCommute}
These four endomorphisms of $T(H)\otimes T(H^*)$ pairwise commute.
\end{lemma}

\begin{proof}
Lemma~\ref{thm:AdjActionsForBialg} implies the commutativity of~${^{H^*}}\!\pi$ and~$\pi^{H^*}$. Replacing~$H$ with~$H^*$, one gets the commutativity of~${^{H}}\!\pi$ and~$\pi^{H}$. Next, returning to the braided interpretation of the adjoint actions, $\pi^{H}$ corresponds to pulling the rightmost $H$-strand to the right of all the $H^*$-strands (using~$\sigma_{bi}$) and applying~$\varepsilon_H$,  while~${^{H^*}}\!\pi$ means pulling the leftmost $H^*$-strand to the left of all the $H$-strands and applying~$\varepsilon_{H^*}$. Thus~$\pi^{H}$ and~${^{H^*}}\!\pi$ commute. The case of~$\pi^{H^*}$ and~${^{H}}\!\pi$ is analogous.

For the two remaining pairs, consider the linear isomorphisms
\[ \Delta_n \otimes  \Id_{H^*}^{\otimes m} \colon H^{n}\otimes (H^*)^{m} \overset{\sim}{\longrightarrow} (H^{op})^{n}\otimes ((H^{op})^*)^{m}, \quad \quad \Delta_n := \bigl( \begin{smallmatrix}
 1 & 2 & \cdots & n\\ n & n-1 & \cdots & 1
\end{smallmatrix} \bigr)
\in S_n,\]
where~$S_n$ acts on~$H^{n}$ by component permutation. These isomorphisms transport the endomorphisms ${^{H^*}}\!\pi$, $\pi^{H^*}$, $\pi^{H}$, and ${^{H}}\!\pi$ of $H^{\otimes n}\otimes (H^*)^{\otimes m}$ to, respectively, ${^{(H^{op})^*}}\!\pi$, $\pi^{(H^{op})^*}$, ${^{H^{op}}}\!\pi$, and $\pi^{H^{op}}$. Thus the commutativity of ${^{(H^{op})^*}}\!\pi$ and $\pi^{H^{op}}$ induces that of ${^{H^*}}\!\pi$ and ${^{H}}\!\pi$, and similarly for $\pi^{H^*}$ and $\pi^{H}$.
\end{proof}

Further, recall the \emph{bar} and (the dual of the) \emph{cobar differentials} on $T(H)\otimes T(H^*)$:
\begin{align}
d_{bar}&(h_1\ldots h_n \otimes l_1\ldots l_m) = \sum\nolimits_{i=1}^{n-1} (-1)^{i}h_1\ldots (h_i \cdot h_{i+1}) \ldots h_n \otimes l_1\ldots l_m,\label{eqn:dbar}\\
d_{cob}&(h_1\ldots h_n \otimes l_1\ldots l_m) = \sum\nolimits_{i=1}^{m-1} (-1)^{i}h_1\ldots h_n \otimes l_1\ldots (l_i\cdot l_{i+1}) \ldots l_m.\label{eqn:dcob}
\end{align}

\begin{proposition}\label{thm:bialg2}
For a finite-dimensional $\k$-linear bialgebra $(H,\mu,\nu,\Delta,\varepsilon)$, the bigraded vector space $T(H)\otimes T(H^*)= \bigoplus_{n,m \in \NN}H^{n}\otimes (H^*)^{m}$ can be endowed with four \emph{bicomplex} structures, presented in Table~\ref{tab:BialgBidiff}. Being a bicomplex means here satisfying
\begin{align*}
d_{n-1,m}  d_{n,m} &= 0, & d'_{n,m-1}  d'_{n,m} &= 0, & d_{n,m-1}  d'_{n,m} + d'_{n-1,m}  d_{n,m} &= 0.
\end{align*}
\begin{table}\centering
\setlength{\tabcolsep}{3pt}
\begin{tabular}{|c|c|c|}
\hline \rowcolor{lightgrey}
& $d_{n,m}: H^{ n}\otimes (H^*)^{ m} \to H^{n-1}\otimes (H^*)^{ m}$
& $d'_{n,m}: H^{ n}\otimes (H^*)^{ m} \to H^{ n}\otimes (H^*)^{m-1}$ \\
\hline
1&$ d_{bar}$&$(-1)^{n}d_{cob}$ \\
\hline 
2&$d_{bar}+(-1)^{n}\pi^{H}$ &$(-1)^{n}d_{cob}+(-1)^{n}({^{H^*}}\!\pi)$ \\
\hline 
3& $d_{bar}+{^{H}}\!\pi$& $(-1)^{n}d_{cob}+(-1)^{n+m}{\pi^{H^*}}$\\
\hline 
4&$d_{bar}+(-1)^{n}\pi^{H}+{^{H}}\!\pi$ & $(-1)^{n}d_{cob}+(-1)^{n}({^{H^*}}\!\pi)+(-1)^{n+m}{\pi^{H^*}}$\\
\hline 
\end{tabular}\setlength{\tabcolsep}{6pt}
\caption{Bicomplex structures on $T(H)\otimes T(H^*)$}\label{tab:BialgBidiff}
\end{table}
\end{proposition}

\begin{proof}
\begin{enumerate}
\item Maps~$d_{bar}$ and~$d_{cob}$ are well known to be differentials (see also their interpretation as braided differentials in Theorem~\ref{thm:UAA}). They affect disjoint parts $T(H)$ and $T(H^*)$ of $T(H)\otimes T(H^*)$, and thus commute. The sign~$(-1)^{n}$ then assures the anticommutativity.
\item Return to the braided system~$\overline{H}_{bi}$, which we no longer consider as bipointed. The counit~$\varepsilon_H$ of~$H$ is an algebra character, hence a braided character for $(H,\sigma_{Ass}^r(H))$. Extended to~$H^*$ by zero, it becomes a braided character for~$\overline{H}_{bi}$ (Example~\ref{rmk:single_br_char}). Similarly, $\varepsilon_{H^*}$  extended to~$H$ by zero is also a braided character for~$\overline{H}_{bi}$. Choosing them as coefficients, one gets the following braided bidifferential, which coincides with the desired one up to a sign:
\begin{align*}
{^{\varepsilon_{H^*}}}\! d &=(-1)^{n}d_{cob}+(-1)^{n}({^{H^*}}\!\pi),&
d{^{\varepsilon_{H}}}&= -(d_{bar}+(-1)^{n}\pi^{H}).
\end{align*}
\item Symmetrically, one gets a bidifferential $((-1)^{m}(d_{bar}+{^{H}}\!\pi),d_{cob}+(-1)^{m}\pi^{H^*})$, hence $(d_{bar}+{^{H}}\!\pi,(-1)^{n}d_{cob}+(-1)^{n+m}\pi^{H^*})$.
\item The last point follows from the preceding ones using an elementary observation:
\begin{lemma}
Take an Abelian group $(S,+,0,a\mapsto -a)$ endowed with an operation~$\cdot$, distributive with respect to~$+$. Then, for any $a,b,c,d,e,f \in S$,
\begin{align*}
(a+b)\cdot & (d+e) = (a+c)\cdot (d+f) = a \cdot d = b \cdot f + c \cdot e = 0\\
&\Longrightarrow \quad (a+b+c)\cdot (d+e+f) = 0.
\end{align*}
\end{lemma}

\begin{proof}
$(a+b+c)\cdot (d+e+f) = (a+b)\cdot (d+e) + (a+c)\cdot (d+f) - a \cdot d +( b \cdot f + c \cdot e)$.
\end{proof}
Now take $S=\End_\k(T(H)\otimes T(H^*))$ with the usual addition and, as the second operation, $a \cdot b:= a  b$ (for proving that the two morphisms from the $4$th line of our table are differentials), or  $a \cdot b:= a  b + b  a$ (for proving that the two morphisms anti-commute). Choose $a= d_{bar}$, $b=(-1)^{n}\pi^{H}$, $c={^{H}}\!\pi$, $d= d_{bar}$ or $d=(-1)^{n}d_{cob}$, etc. The equalities of the type $b \cdot f + c \cdot e = 0$ follow from the pairwise anti-commutativity of $(-1)^{n}({^{H^*}}\!\pi)$, $(-1)^{n+m}\pi^{H^*}$, $(-1)^{n}\pi^{H}$, and ${^{H}}\!\pi$ (Lemma~\ref{thm:AdjActionsForBialgCommute}), and the remaining ones from Points 1-3. \qedhere
\end{enumerate}
\end{proof}

One recognizes in $d_{bar}+(-1)^{n}\pi^{H}+{^{H}}\!\pi$ the Hochschild differential for~$H$ with  coefficients in the $H$-bimodule $T(H^*)$ (Lemma~\ref{thm:AdjActionsForBialg}). Dually, $d_{cob}+{^{H^*}}\!\pi+(-1)^{m}{\pi^{H^*}}$ is the Hochschild differential for~$H^*$. Thus the last bicomplex from Table~\ref{tab:BialgBidiff} yields the \emph{Gersten\-haber--Schack bialgebra homology}~\cite{GS90}. See Taillefer's thesis~\cite{TailleferThese} for computations and comparison with other homologies, and the work of Mastnak--Witherspoon \cite{MW} for explicit formulas and the transition from $\Hom_\k(H^{m},H^{n})$ to $ H^{n}\otimes (H^*)^{m}$. 

Now, instead of the braided characters~$\varepsilon_H$ and~$\varepsilon_{H^*}$  for~$\BBialg(H)$, take general braided modules $(M,\rho,\delta) \in \ModCat_H^H \simeq \ModCatN_{\BBialg(H)}$ and $(N,\lambda,\gamma) \in {_{H^*}^{H^*}}\!\ModCat \simeq {_{\BBialg(H)}}\!\ModCatN$. On $M \otimes H^{n}\otimes (H^*)^{m} \otimes N$, define the maps~$\pi^{H}$ and ${^{H^*}}\!\pi$ using adjoint actions as before:
\begin{align*}
\pi^{H}&(a \otimes h_1\ldots h_n \otimes l_1 \ldots l_m\otimes b) = \\
&\left\langle l_{1(1)},h_{n(m+1)}\right\rangle \ldots \left\langle l_{m(1)},h_{n(2)}\right\rangle \left\langle b_{-1},h_{n(1)}\right\rangle 
a \otimes h_1\ldots h_{n-1} \otimes l_{1(2)}\ldots l_{m(2)} \otimes b_0,\\
{^{H^*}}\!\pi&(a \otimes h_1\ldots h_n \otimes l_1 \ldots l_m \otimes b) = \\
&\left\langle l_{1(1)},h_{n(2)}\right\rangle \ldots \left\langle l_{1(n)},h_{1(2)}\right\rangle \left\langle l_{1(n+1)},a_1\right\rangle 
a_0 \otimes h_{1(1)}\ldots h_{n(1)} \otimes l_2 \ldots l_m\otimes b.
\end{align*}
 Further, let~${^{H}}\!\pi$ be the action~$\rho$ applied to the two leftmost factors, and let~$\pi^{H^*}$ be the action~$\lambda$ applied to the two rightmost factors. We still denote by~$d_{bar}$ and~$d_{cob}$ the differentials \eqref{eqn:dbar}-\eqref{eqn:dcob} tensored with~$\Id_M$ on the left and with~$\Id_N$ on the right. Repeating the argument of Proposition~\ref{thm:bialg2} for these maps, one shows that $d_{bar}+(-1)^{n}\pi^{H}+{^{H}}\!\pi$ and $(-1)^{n}d_{cob}+(-1)^{n}({^{H^*}}\!\pi)+(-1)^{n+m}{\pi^{H^*}}$ define a bicomplex on $M \otimes T(H)\otimes T(H^*) \otimes N$. If~$N$ is finite dimensional, then one can see $M \otimes H^{n}\otimes (H^*)^{m} \otimes N$ as $\Hom(N^* \otimes H^m,M \otimes H^n)$, with $ N^* \in \ModCat_H^H $. One recovers (a variation of) the \emph{deformation (co)homology of Hopf modules}, due to Panaite--{\c{S}}tefan \cite{Panaite}.

\section{A braided interpretation of Hopf bimodules}\label{sec:HBimod}

In this section, the braided system~$\BBialg(H)$ for a bialgebra~$H$ is upgraded to a more complicated rank~$4$ system~$\BBialgg(H)$. Braided $\BBialgg(H)$-modules are identified as \emph{Hopf bimodules} over~$H$, or else as modules over the algebras~$\X$, $\Y$, and~$\Z$ of Cibils--Rosso and Panaite. These algebras are included into a list of $24$ braided tensor products of UAAs, shown pairwise isomorphic by component permuting techniques. Braided bidifferentials for~$\BBialgg(H)$ recover the \emph{Hopf bimodule (co)homology} of Ospel--Taillefer.

\begin{definition}\label{def:cat_Hopfbi}
In a braided category~$\C$, a \emph{Hopf bimodule} over a bialgebra~$H$ is an object~$M$ with a bimodule structure  $ M \otimes H \overset{\rho}{\to} M, H \otimes M \overset{\lambda}{\to} M$ and a bicomodule structure $M \overset{\delta}{\to} M \otimes H, M \overset{\gamma}{\to} H \otimes M$, satisfying~\eqref{eqn:rrHopfMod} and $3$ other \emph{Hopf compatibility conditions} (Fig.~\ref{pic:HopfBiMod}):
\begin{align*}
\delta  \lambda &= (\lambda \otimes \mu) (\Id_H \otimes c_{H,M} \otimes \Id_H)  (\Delta \otimes \delta) \quad : \quad H \otimes M \to  M \otimes H, \\
\gamma  \rho &= (\mu \otimes \rho) (\Id_H \otimes c_{M,H} \otimes \Id_H)  (\gamma \otimes \Delta) \quad : \quad M \otimes H \to  H \otimes M, \\
\gamma  \lambda &= (\mu \otimes \lambda) (\Id_H \otimes c_{H,H} \otimes \Id_M)  (\Delta \otimes \gamma) \quad : \quad H \otimes M \to  H \otimes M.
\end{align*}
The category of Hopf bimodules over~$H$ and their morphisms is denoted by ${_H^H}\!\ModCat_H^H$. 
\end{definition}

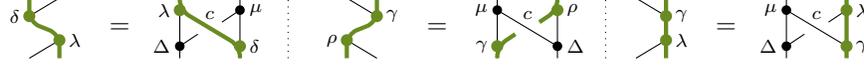
\begin{figure}\centering
\begin{tikzpicture}[xscale=0.8, yscale = 0.8]
 \draw [colorM, ultra thick,rounded corners] (1,0) --(1,0.4) --(0.5,0.6) --  (0.5,1);
 \draw (0.5,0) -- (1,0.3);
 \draw (0.5,0.7) -- (1,1);
 \node at (1,0.3) [right] {$\scriptstyle \lambda$};
 \node at (0.5,0.7) [left] {$\scriptstyle \delta$};
 \fill [colorM] (1,0.3) circle (0.1);
 \fill [colorM] (0.5,0.7) circle (0.1);
 \node at (2,0.5) {$=$};
 \draw (4,1) -- (4,0);
 \draw (3,0.2) -- (3.3,0.4); 
 \draw (3.7,0.6) -- (4,0.8);
 \draw (3,0) -- (3,1);
 \draw [colorM, ultra thick,rounded corners] (4,0) -- (4,0.2) -- (3,0.8) -- (3,1);
 \fill [colorM] (4,0.2) circle (0.1);
 \fill [colorM] (3,0.8) circle (0.1);
 \fill (4,0.8) circle (0.08);
 \fill (3,0.2) circle (0.08);
 \node at (3,0.8) [left] {$\scriptstyle \lambda$};
 \node at (3,0.2) [left] {$\scriptstyle \Delta$};
 \node at (4,0.8) [right] {$\scriptstyle \mu$};
 \node at (4,0.2) [right] {$\scriptstyle \delta$};
 \node at (3.5,0.5) [above] {$\scriptstyle c$};
 \draw [dotted] (4.8,0) -- (4.8,1);
 \node at (5,0.5) {};  
\end{tikzpicture}
\begin{tikzpicture}[xscale=0.8, yscale = 0.8]
 \draw [colorM, ultra thick, rounded corners] (0.5,0) --  (0.5,0.4) --  (1,0.6) --  (1,1);
 \draw (1,0) -- (0.5,0.3);
 \draw (1,0.7) -- (0.5,1);
 \node at (0.5,0.3) [left] {$\scriptstyle \rho$};
 \node at (1,0.7) [right] {$\scriptstyle \gamma$};
 \fill [colorM] (0.5,0.3) circle (0.1);
 \fill [colorM] (1,0.7) circle (0.1);
 \node at (2,0.5) {$=$};
 \draw (4,1) -- (4,0);
 \draw (3,1) -- (3,0);
 \draw (3,0.8) -- (4,0.2);
 \draw[colorM, ultra thick, rounded corners] (3,0) -- (3,0.2) -- (3.3,0.4); 
 \draw[colorM, ultra thick, rounded corners] (3.7,0.6) -- (4,0.8) -- (4,1); 
 \fill [colorM] (3,0.2) circle (0.1);
 \fill [colorM] (4,0.8) circle (0.1);
 \fill (3,0.8) circle (0.08);
 \fill (4,0.2) circle (0.08);
 \node at (3,0.8) [left] {$\scriptstyle \mu$};
 \node at (3,0.2) [left] {$\scriptstyle \gamma$};
 \node at (4,0.8) [right] {$\scriptstyle \rho$};
 \node at (4,0.2) [right] {$\scriptstyle \Delta$};
 \node at (3.5,0.5) [above] {$\scriptstyle c$};
 \draw [dotted] (4.8,0) -- (4.8,1);
 \node at (5,0.5) {};  
\end{tikzpicture}
\begin{tikzpicture}[xscale=0.8, yscale = 0.8]
 \draw [colorM, ultra thick] (1,0) --  (1,1);
 \draw (0.5,0) -- (1,0.3);
 \draw (1,0.7) -- (0.5,1);
 \node at (1,0.3) [right] {$\scriptstyle \lambda$};
 \node at (1,0.7) [right] {$\scriptstyle \gamma$};
 \fill [colorM] (1,0.3) circle (0.1);
 \fill [colorM] (1,0.7) circle (0.1);
 \node at (2,0.5) {$=$};
 \draw[colorM, ultra thick] (4,0) -- (4,1);
 \draw (3,1) -- (3,0);
 \draw (3,0.2) -- (3.3,0.4); 
 \draw (3.7,0.6) -- (4,0.8);
 \draw (3,0.8) -- (4,0.2);
 \fill [colorM] (4,0.2) circle (0.1);
 \fill [colorM] (4,0.8) circle (0.1);
 \fill (3,0.8) circle (0.08);
 \fill (3,0.2) circle (0.08);
 \node at (3,0.8) [left] {$\scriptstyle \mu$};
 \node at (3,0.2) [left] {$\scriptstyle \Delta$};
 \node at (4,0.8) [right] {$\scriptstyle \lambda$};
 \node at (4,0.2) [right] {$\scriptstyle \gamma$};
 \node at (3.5,0.5) [above] {$\scriptstyle c$};
\end{tikzpicture}
\caption{Hopf compatibility conditions}\label{pic:HopfBiMod}
\end{figure}

We now return to our category $\C = \vect$, as usual omitted from notations.

\begin{theorem}\label{thm:HBiMod}
\begin{enumerate}
\item\label{item:CatInclBi4}
One has a fully faithful functor
\begin{align}
\FF :\; {^*}\!\Bialg & \longhookrightarrow {^*}\!\BrSystBP_4 \label{eqn:BialgAsBrSyst4}\\
(H,\mu,\nu,\Delta,\varepsilon) &\longmapsto  \BBialgg(H):=(V_1:=H, V_2:=H^{op}, V_3:=H^*, V_4:=(H^{cop})^*; \notag\\
&  \sigma_{i,i}:=\sigma_{Ass}(V_i),\sigma_{1,2}:=\tau_{H,H^{op}},\sigma_{3,4}:=\tau_{H^*, (H^{cop})^*}, \sigma_{1,3}:=\sigma_{bi}(H), \notag\\
&   \sigma_{2,3}:=\sigma_{bi}(H^{op}),\sigma_{1,4}:=\sigma_{bi}(H^{cop}),\sigma_{2,4}:=\sigma_{bi}(H^{op,cop});\notag\\
&   \nu,\nu,\varepsilon^*,\varepsilon^*; \varepsilon,\varepsilon,\nu^*,\nu^*),\notag\\
f &\longmapsto (f,f,(f^{-1})^*,(f^{-1})^*),\notag
\end{align}
where~$\tau$ is the transposition of the corresponding factors, and $\sigma_{bi}(A)$ denotes the map~$\sigma_{bi}$ from Theorem~\ref{thm:Bialg} for the bialgebra~$A$ (Fig.~\ref{pic:sigmaHBiMod}).

\item\label{item:ModEquivBi4} For a bialgebra~$H$, one has category equivalences 

{\renewcommand{\arraystretch}{1.2}\setlength{\tabcolsep}{2pt}
\begin{tabular}{ccccc} 
${_H^H}\!\ModCat_H^H$ & $\stackrel{\sim}{\rightarrow}$ & $\ModCatN_{\BBialgg(H)}$ & $\stackrel{\sim}{\rightarrow}$ & $\ModCat_{\W(H)}$\\ 
$(M,\rho, \lambda, \delta, \gamma)$ & $\mapsto$ & $(M;\rho, \rrho(\lambda), \delta^{co}, \rrho(\gamma^{co}))$ & $\mapsto$ & $(M, \rrho (\gamma^{co}) \otimes \delta^{co} \otimes \rrho (\lambda) \otimes \rho)$
\end{tabular}
\setlength{\tabcolsep}{6pt}}

where~$\rrho$ is the correspondence from Lemma~\ref{thm:lMod=rMod}, and~$\W(H)$ is the braided tensor product of UAAs 
$\W(H)=(H^{cop})^* \underset{\xi}{\otimes} H^*\underset{\xi}{\otimes} H^{op}  \underset{\xi}{\otimes} H$.

\item\label{item:ModEquivBi4'}  If~$H$ is a Hopf algebra, then, for any $\theta \in S_4$, one has category equivalences 
\[{_H^H}\!\ModCat_H^H \stackrel{\sim}{\longrightarrow} \ModCatN_{\theta \BBialgg(H)} \stackrel{\sim}{\longrightarrow} \ModCat_{\theta\cdot \W(H)},\]
where the bipointed braided system~$\theta \BBialgg(H)$ is obtained from~$\BBialgg(H)$ by a component permutation from Remark~\ref{thm:BrSystInvSn}, and the UAA~$\theta \cdot \W(H)$, isomorphic to~$\W(H)$, is obtained from~$\W(H)$ by a component permutation from Remark~\ref{thm:UAABrSystInvSn}. 
\end{enumerate}
\end{theorem} 

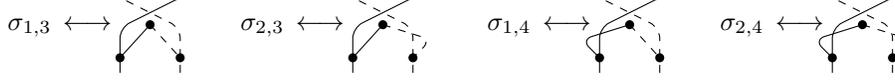
\begin{figure}\centering
\begin{tikzpicture}[scale=0.8]
 \node at (-1,0.5) {$\sigma_{1,3} \, \longleftrightarrow $};
 \draw [dashed,rounded corners] (1,-0.25)  -- (1,0.5) -- (0,1);
 \draw [rounded corners] (0,-0.25) -- (0,0.5) -- (1,1);
 \draw [dashed,rounded corners] (0.5,0.55)-- (1,0);
 \draw [rounded corners](0,0) -- (0.5,0.55);
 \fill (0.5,0.55) circle (0.075);
 \fill (0,0) circle (0.075);
 \fill (1,0) circle (0.075);
\end{tikzpicture}
\begin{tikzpicture}[scale=0.8]
 \node at (-2.5,0) {};
 \node at (-1,0.5) {$\sigma_{2,3} \, \longleftrightarrow $};
 \draw [dashed,rounded corners] (1,-0.25)  -- (1,0.5) -- (0,1);
 \draw [rounded corners] (0,-0.25) -- (0,0.5) -- (1,1);
 \draw [dashed,rounded corners] (0.5,0.55) -- (1.3,0.3) -- (1,0);
 \draw [rounded corners](0,0) -- (0.5,0.55);
 \fill (0.5,0.55) circle (0.075);
 \fill (0,0) circle (0.075);
 \fill (1,0) circle (0.075);
\end{tikzpicture}
\begin{tikzpicture}[scale=0.8]
 \node at (-2.5,0) {};
 \node at (-1,0.5) {$\sigma_{1,4} \, \longleftrightarrow $};
 \draw [dashed,rounded corners] (1,-0.25)  -- (1,0.5) -- (0,1);
 \draw [rounded corners] (0,-0.25) -- (0,0.5) -- (1,1);
 \draw [dashed,rounded corners] (0.5,0.55) -- (1,0);
 \draw [rounded corners](0,0) -- (-0.3,0.3) -- (0.5,0.55);
 \fill (0.5,0.55) circle (0.075);
 \fill (0,0) circle (0.075);
 \fill (1,0) circle (0.075);
\end{tikzpicture}
\begin{tikzpicture}[scale=0.8]
 \node at (-2.5,0) {};
 \node at (-1,0.5) {$\sigma_{2,4} \, \longleftrightarrow $};
 \draw [dashed,rounded corners] (1,-0.25)  -- (1,0.5) -- (0,1);
 \draw [rounded corners] (0,-0.25) -- (0,0.5) -- (1,1);
 \draw [dashed,rounded corners] (0.5,0.55) -- (1.3,0.3) -- (1,0);
 \draw [rounded corners](0,0) -- (-0.3,0.3) -- (0.5,0.55);
 \fill (0.5,0.55) circle (0.075);
 \fill (0,0) circle (0.075);
 \fill (1,0) circle (0.075);
\end{tikzpicture}
\caption{Some braiding components for~$\BBialgg(H)$}\label{pic:sigmaHBiMod}
\end{figure}

\begin{proof}
Let~$\FF_{i,j}$ be the composition of~$\FF$ with the forgetful functor $\For_{i,j}\colon {^*}\!\BrSystBP_4 \to {^*}\!\BrSystBP_2$ which picks the $i$th and $j$th components, $i < j$. For $i \le 2 < j$ one recognizes in~$\FF_{i,j}$ the functor~\eqref{eqn:BialgAsBrSyst} from Theorem~\ref{thm:Bialg} and its slight modifications which send a bialgebra~$H$ to $\BBialg(H^{op})$, $\BBialg(H^{cop})$, or $\BBialg(H^{op,cop})$ (with some $\sigma_{Ass}^r$-type braiding components replaced with their $\sigma_{Ass}$ versions). Further, $\FF_{1,2}(H)$ and $\FF_{3,4}(H)$ coincide with the braided systems of UAAs~$\Bimod(H,H)$ and $\Bimod(H^*,H^*)$ respectively. Hence all the~$\xi_{i,j}$ for $i < j$ are natural w.r.t. the units and the multiplications. They also satisfy the YBEs required by Theorem~\ref{thm:ProdOfAlgebrasBrFamily}~\ref{item:mixed}. Indeed, on $V_1 \otimes V_2 \otimes V_k$, $k \in \{3,4\}$, the YBE follows from the associativity of~$\mu$, and on $V_k \otimes V_3 \otimes V_4$, $k \in \{1,2\}$ from the coassociativity of~$\Delta$. Theorem~\ref{thm:ProdOfAlgebrasBrFamily} then asserts that $\BBialgg(H)$ is a braided system of UAAs. It is clearly bipointed. 

To show that~$\FF$ is well defined on morphisms, it suffices to check this for all the $\FF_{i,j}$, $i < j$. For $i \le 2 < j$ it follows from Theorem~\ref{thm:Bialg}. For~$\FF_{1,2}$ and~$\FF_{3,4}$, observe that the~$\xi_{1,2}$ and~$\xi_{3,4}$ components of our braidings are simply transpositions, ensuring the defining property~\eqref{eqn:BrMor} of braided morphisms. Further, take a braided isomorphism $(f,g,h,k) \colon \BBialgg(H) \to \BBialgg(K)$ for bialgebras~$H$ and~$K$. Applying forgetful functors $\For_{i,j}$, $i \le 2 < j$, and using Theorem~\ref{thm:Bialg} again, one sees that~$f$ is a bialgebra isomorphism, and that $f=g=(h^*)^{-1}=(k^*)^{-1}$. Hence $\FF$ is full and faithful.

Let us turn to modules. Take $(M,\rho, \lambda, \delta, \gamma) \in {_H^H}\!\ModCat_H^H$. Transform left structures~$\lambda$ and~$\gamma$ into right structures $\rrho(\lambda)$ and~$\rrho(\gamma)$, and then comodule structures $\delta$ and~$\rrho(\gamma)$ into module structures $\delta^{co}$ and $\rrho(\gamma)^{co} =\rrho(\gamma^{co})$. Thus the Hopf bimodule~$M$ over~$H$ becomes a module over UAAs $H=V_1$, $H^{op}=V_2$, $H^*=V_3$, and $(H^{cop})^*=V_4$. Further, the $4$ Hopf compatibility conditions coincide with the braided module compatibility conditions on $V_i \otimes V_j$, $i \le 2 < j$, and left-right action (or coaction) compatibility conditions cover the case $i=1$, $j=2$ (respectively, $i=3$, $j=4$). Observation~\ref{rmk:MultiMod} then yields the desired category equivalence ${_H^H}\!\ModCat_H^H \simeq \ModCatN_{\BBialgg(H)}$.

The remaining assertions follow from the correspondence between braided modules and modules over braided tensor products (Proposition~\ref{thm:ProdOfAlgebrasBrFamilyMod}), the invertibility of~$\sigma_{bi}$ in the Hopf algebra case, the properties of twisted Hopf algebras (Observation~\ref{thm:TwistedBialg}; recall that in the finite-dimensional case, an antipode is always invertible), and the component permuting Propositions~\ref{thm:BrSystInv} and~\ref{thm:ProdOfAlgebrasBrFamilyInv}.
\end{proof}

The category ${_H^H}\!\ModCat_H^H$ for a Hopf algebra~$H$ is known to be equivalent to the categories of right modules over $3$ UAAs: the twisted product of Cibils--Rosso \cite{CibilsRosso}:
\begin{align*}
\X(H) &= (H \otimes H^{op})\underline{\otimes}(H^* \otimes (H^*)^{op}),
\end{align*}
and the two-sided and diagonal crossed products of Panaite~\cite{Panaite2}:
\begin{align*}
\Y(H) &= H^* \# (H^{op} \otimes H)\# (H^*)^{op},&
\Z(H) &= (H^* \otimes (H^*)^{op}) \bowtie (H^{op} \otimes H).
\end{align*}
Here we adapt Panaite's notations to our conventions. For instance, he uses the arched duality, so his dual bialgebra~$H^*$ corresponds to our~$(H^*)^{op,cop}$. Also, he sees Hopf bimodules {over $H^*$} as {left} modules over~$\X(H)$, while we interpret Hopf bimodules {over $H$} as {right} modules. The algebras $\X,\Y,\Z$ are of the form $\theta\cdot \W(H)$, with as~$\theta$ the permutations $(14)(23)$, $(1234)$, and $(34)$. Point~\ref{item:ModEquivBi4'} of our theorem includes them into a family of $\# S_4 = 24$ UAAs and gives explicit isomorphisms between them, inducing equivalences for their module categories (Remark~\ref{thm:UAABrSystInvSn}). We thus generalize and conceptually explain the central results of \cite{CibilsRosso,Panaite2}, minimizing technical computations.

Braided adjoint actions allow to regard the bar complex with bimodule coefficients as a complex of bimodules (Proposition~\ref{thm:BarCxIsBimod}). The same is true for Hopf bimodules:

\begin{proposition}\label{thm:BarCxIsHopfBimod}
Take a Hopf bimodule $(M, \, M \otimes H \overset{\rho}{\to} M,\,  H \otimes M \overset{\lambda}{\to} M, \, M \overset{\delta}{\to} M \otimes H, \, M \overset{\gamma}{\to} H \otimes M )$ over a bialgebra $(H,\mu,\nu, \Delta,\varepsilon)$ in~$\vect$. The bar complex $(M\otimes T(H), d_{bar})$ for~$H$ with coefficients in~$M$ is a complex in ${_H^H}\!\ModCat_H^H$. In other words, the differentials $(d_{bar})_n$ are Hopf bimodule morphisms, with the following Hopf bimodule structure on $M \otimes H^{n}$ (Fig.~\ref{pic:BarCxIsBicomod}):
\begin{align*} & \left. 
   \begin{array}{r c l r c l}
\rho_{bar} & = &  \mu^{n+1}, & \qquad\quad
\lambda_{bar} & = &  \lambda^1,
   \end{array}
  \right\} \quad \text{\parbox{0.4\textwidth}{peripheral  actions}} \\
& \left. 
   \begin{array}{r c l}
\delta_{bar} & = & (\mu^{n+2})^{(n)}  \omega_{2(n+1)}^{-1}  (\delta \otimes \Delta^{\otimes n}),\\
\gamma_{bar} & = & (\mu^{1})^{( n)}  \omega_{2(n+1)}^{-1}  (\gamma \otimes \Delta^{\otimes n}).   \end{array}
  \right\} \quad \text{\parbox{0.2\textwidth}{diagonal \\ coactions}}  
\end{align*}  
Here $\omega_{2(n+1)} \in S_{2(n+1)}$ from~\eqref{eqn:omega} acts on $M \otimes H^{2n+1}$ by factor permutation. 
\end{proposition}

\begin{figure}\centering
\begin{tikzpicture}[xscale=0.5, yscale = 0.4]
 \node at (-4,0) {};
 \node at (-2,2) {$\delta_{bar}  \; \longleftrightarrow$}; 
 \draw [colorM,ultra thick] (0,0) --  (0,4);
 \draw (1,0) -- (1,4);
 \draw (2,0) -- (2,4);
 \draw (3,0) -- (3,4);
 \draw (0,1) -- (4,3) -- (4,4);
 \draw (1,1) -- (4,3);
 \draw (2,1) -- (4,3);
 \draw (3,1) -- (4,3);
 \node at (0,1) [left] {$\scriptstyle \delta$};
 \node at (1,1) [left] {$\scriptstyle \Delta$};
 \node at (2,1) [left] {$\scriptstyle \Delta$};
 \node at (3,1) [left] {$\scriptstyle \Delta$};
 \node at (4,3) [right] {$\scriptstyle \mu$};
 \fill [colorM] (0,1) circle (0.12);
 \fill (1,1) circle (0.1);
 \fill (2,1) circle (0.1);
 \fill (3,1) circle (0.1);
 \fill (4,3) circle (0.1);
\end{tikzpicture}
\begin{tikzpicture}[xscale=0.5, yscale = 0.4]
 \node at (-6,0) {};
 \node at (-3.2,2) {$\gamma_{bar}  \; \longleftrightarrow$}; 
 \draw [colorM,ultra thick] (0,0) --  (0,4);
 \draw (1,0) -- (1,4);
 \draw (2,0) -- (2,4);
 \draw (3,0) -- (3,4);
 \draw (0,1) -- (-1,3) -- (-1,4);
 \draw (1,1) -- (-1,3);
 \draw (2,1) -- (-1,3);
 \draw (3,1) -- (-1,3);
 \node at (0,1) [right] {$\scriptstyle \gamma$};
 \node at (1,1) [right] {$\scriptstyle \Delta$};
 \node at (2,1) [right] {$\scriptstyle \Delta$};
 \node at (3,1) [right] {$\scriptstyle \Delta$};
 \node at (-1,3) [left] {$\scriptstyle \mu$};
 \fill [colorM] (0,1) circle (0.12);
 \fill (1,1) circle (0.1);
 \fill (2,1) circle (0.1);
 \fill (3,1) circle (0.1);
 \fill (-1,3) circle (0.1);
\end{tikzpicture}
\caption{Diagonal bicomodule structure on the bar complex}\label{pic:BarCxIsBicomod}
\end{figure}
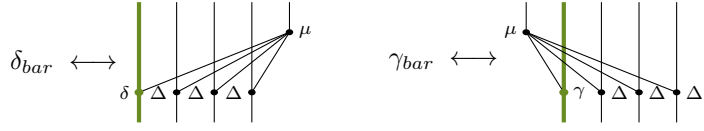

\begin{proof}
Theorem~\ref{thm:HBiMod} asserts that~$M$ is a $\BBialgg(H)$-module. Proposition~\ref{thm:adjoint_multi_coeffs} for $t=1$ then yields a $\BBialgg(H)$-module structure on $M \otimes T(H)$, compatible with the braided differential ${^{\rho}}\!d$. By Theorem~\ref{thm:UAA}, the latter is the bar differential. Using Theorem~\ref{thm:HBiMod} again, one transforms the $\BBialgg(H)$-module structure on $M \otimes T(H)$ into a Hopf bimodule structure over~$H$, which coincides with the desired one.
\end{proof}

This Hopf bimodule structure on the bar complex, and its dual structure on the cobar complex, are essential for defining the \emph{Hopf bimodule (co)homology}, introduced by Ospel in the one-module case~\cite{OspelThese} and by Taillefer \cite{TailleferThese,Taillefer} for two modules.

Now, take two Hopf bimodules $M \in {_H^H}\!\ModCat_H^H \simeq \ModCatN_{\BBialgg(H)}$ and $N \in {_{H^*}^{H^*}}\!\ModCat_{H^*}^{H^*} \simeq {_{\BBialgg(H)}}\!\ModCatN$. Mimicking the constructions for Hopf modules from the previous section, one gets a tetra-complex structure on the tetra-graded vector space $M \otimes T(H) \otimes T(H^{op})\otimes T(H^*)\otimes T((H^{cop})^*) \otimes N$. If~$N$ is finite dimensional, then this space can be regarded as $\Hom(T(H) \otimes N^* \otimes T(H), T(H) \otimes M \otimes T(H))$, with $ N^* \in {_H^H}\!\ModCat_H^H $ (here in order to get rid of twisted (co)multiplications, we moved $T(H^{op})$ to the left of~$M$, reversing the order of its factors, and similarly for $T((H^{cop})^*)$). This generalizes an alternative (co)homological approach to Hopf bimodules from~\cite{Taillefer}.

\def\cprime{$'$} \def\cprime{$'$} \def\cprime{$'$} \def\cprime{$'$}
  \def\cprime{$'$} \def\cprime{$'$}

\end{document}